\documentclass[reqno,twoside]{amsart}
\usepackage{amscd}
\usepackage{amsfonts}
\usepackage{amsmath}
\usepackage{amssymb}
\usepackage{amsthm}
\usepackage{amsaddr}
\usepackage{fancyhdr}
\usepackage{latexsym}
\usepackage[colorlinks=true, pdfstartview=FitV, linkcolor=blue, citecolor=blue, urlcolor=blue]{hyperref}
\usepackage{enumitem}      
\usepackage{mathtools}            
\usepackage{indentfirst} 
\usepackage{color}
\usepackage{caption}          
\usepackage[normalem]{ulem}
\usepackage{upgreek}
\usepackage{mathrsfs}
\newcommand{\eee}[1]{\begin{equation}#1\end{equation}}

\newcommand{\aaa}[1]{\begin{alignat}{2}#1\end{alignat}}
\newcommand{\ddd}[1]{\begin{equation}\begin{aligned}#1\end{aligned}\end{equation}}
\newcommand{\nn}{\nonumber}
\newcommand{\p}{\partial}
\newcommand{\sech}{\text{sech}}
\newcommand{\ve}{\varepsilon}

\newcommand{\no}[1]{\left\| #1 \right\|}

\renewcommand{\r}{\mathbb R}
\newcommand{\n}{\mathbb N}
\newcommand{\z}{\mathbb Z}
\renewcommand{\l}{\mathbb L}

\usepackage{accents}
\renewcommand*{\dot}[1]{%
\accentset{\mbox{\small\bfseries .}}{#1}}
\renewcommand*{\ddot}[1]{%
\accentset{\mbox{\small\bfseries ..}}{#1}}
\usepackage{scalerel,stackengine}
\stackMath
\newcommand\wwhat[1]{%
\savestack{\tmpbox}{\stretchto{%
  \scaleto{%
    \scalerel*[\widthof{\ensuremath{#1}}]{\kern-.6pt\bigwedge\kern-.6pt}%
    {\rule[-\textheight/2]{1ex}{\textheight}}
  }{\textheight}%
}{0.5ex}}%
\stackon[1pt]{#1}{\tmpbox}%
}
\usepackage{thmtools}
\declaretheoremstyle[headfont=\normalfont\bfseries, bodyfont=\itshape, spaceabove=9pt, spacebelow=9pt]{theorem} 
\theoremstyle{theorem} 
\newtheorem{theorem}{Theorem}[] 
\newtheorem{remark}{Remark}[]
\newtheorem{proposition}{Proposition}[]
\newtheorem{lemma}{Lemma}[]
\newtheorem{corollary}{Corollary}[]

\theoremstyle{definition}

\allowdisplaybreaks
\numberwithin{equation}{section}
\numberwithin{figure}{section}
\usepackage{geometry}
\usepackage[breakable, theorems, skins]{tcolorbox}
\tcbset{enhanced}

\usepackage{tikz}
\usetikzlibrary{arrows.meta}
\usetikzlibrary{decorations.markings}
\tikzset{->-/.style={decoration={
  markings,
  mark=at position #1 with {\arrow{>}}},postaction={decorate}}}
  \tikzset{middlearrow/.style={
        decoration={markings,
            mark= at position 0.55 with {\arrow{#1}} ,
        },
        postaction={decorate}
    }
}
\makeatletter
  \renewcommand\subsection{\@startsection{subsection}{2}%
  \z@{-1\linespacing\@plus-0.7\linespacing}{0.7\linespacing}%
  {\bfseries}}
\makeatother
\let\OLDthebibliography\thebibliography
\renewcommand\thebibliography[1]{
  \OLDthebibliography{#1}
  \setlength{\parskip}{0pt}
  \setlength{\itemsep}{0pt plus 0.3ex}
}
\AtBeginDocument{
   \def\MR#1{}
}

\makeatletter
\@namedef{subjclassname@2020}{%
  \textup{2020} Mathematics Subject Classification}
\makeatother

\begin{document}

\title{Blow-up for the  integrable defocusing Ablowitz-Ladik equation}

\author{Dirk Hennig$^{\MakeLowercase{a}}$, 
Nikos I. Karachalios$^{\MakeLowercase{a}}$, 
Dionyssios Mantzavinos$^{\MakeLowercase{b}}$,
\\[1mm]
Dimitrios Mitsotakis$^{\MakeLowercase{c}}$, 
Barbara Prinari$^{\MakeLowercase{d},\MakeLowercase{e}}$ 
}
	\address{\normalfont $^a$Department of Mathematics,  University of Thessaly, 35100, Lamia, Greece
	\\
	\normalfont $^b$Department of Mathematics, University of Kansas, Lawrence, KS 66045, USA
	\\ 
	\normalfont $^c$Victoria University of Wellington, School of Mathematics and Statistics,  Wellington 6140, New Zealand
	\\ 
	\normalfont $^d$Department of Mathematics, State University of New York at Buffalo, Buffalo, NY 14260, USA
    \\ 
    \normalfont $^e$Department of Mathematics, University of Ioannina, Ioannina 45110, Epirus, Greece}
    \email{
    karan@uth.gr \textnormal{(corresponding author)}}

\thanks{\textit{Acknowledgements.} The third author gratefully acknowledges support from the U.S. National Science Foundation (NSF-DMS 2206270 and NSF-DMS 2509146) and the Simons Foundation (SFI-MPS-TSM-00013970). The fifth author gratefully acknowledges partial support from the U.S. National Science Foundation (NSF-DMS-2406626), the Fulbright Foundation in Greece and the Fulbright program, and the Mathematics Department of the University of Ioannina, Greece, for the kind hospitality.
}
\date{September 28, 2026}

\begin{abstract}
This work establishes finite-time blow-up for the completely integrable defocusing Ablowitz-Ladik equation on both infinite and finite lattices. We first establish local well-posedness  for arbitrary initial data in the space $\ell^\infty$ of uniformly bounded sequences, with no decay assumptions and under general boundary conditions. We then show that the unit circle acts as a separatrix for the lattice dynamics: sites with an initial modulus inside, on, or outside the unit circle maintain their modulus within their respective region for as long as the solution exists. As a consequence, initial data in the closure of the $\ell^\infty$ unit ball generate global-in-time solutions. On the other hand, the unit circle separatrix opens up the possibility for finite-time blow-up, which is restricted to sites with initial modulus outside the unit circle. We establish this scenario by constructing a class of initial data whose corresponding solutions are singular, with explicitly given exact blow-up times. We also construct solutions emanating from initial data with an arbitrarily large single-site excitation outside the unit circle that exist globally in time, a behavior driven by the convergence of the corresponding modulus to the unit circle equilibrium. In the presence of such solutions, along with other important classes of analytical, global-in-time  solutions that remain outside the closed $\ell^\infty$ unit ball for all times, we examine the fundamental question of the dominant dynamics for general initial data outside the closed unit ball. By utilizing the constructed class of singular solutions as subsolutions for the general problem, we prove that the  dynamics for the completely integrable defocusing Ablowitz-Ladik lattice is dominated by finite-time blow-up. We conclude with numerical studies of illustrative examples of general initial data, which demonstrate excellent agreement with our main theoretical results.
\end{abstract}

\maketitle
\markboth
{D. Hennig, N. Karachalios, D. Mantzavinos, D. Mitsotakis,   B. Prinari}
{Blow-up for the integrable defocusing Ablowitz-Ladik equation}

\vspace*{-3mm}
{\hypersetup{linkcolor=black}\tableofcontents}

\section{Introduction and main results}

This work establishes finite-time blow-up for the completely integrable defocusing Ablowitz-Ladik (AL) equation~\cite{al1976b,al1976a}
\begin{equation}\label{al}
	i \frac{du_n}{dt} = (u_{n+1}-2u_n+u_{n-1}) - |u_n|^2 (u_{n+1}+u_{n-1}), \quad n\in \l, \ t\in\mathbb{R},
\end{equation}
where $\l \subseteq \z$ denotes either an infinite or a finite lattice (including, in particular, the case of $\l = \z$), under suitable conditions on a broad class of initial data that can accommodate generic boundary conditions on an infinite or finite lattice. 
This is a noteworthy result that highlights a crucial difference between the dynamics in a discrete versus a continuous setting.
Specifically, under the change of variables 	$u_n \mapsto h u_n, \, t \mapsto h^{-2} t$, equation~\eqref{al} takes the form
\begin{equation} \label{e:ALq}
	i\frac{du_n}{dt}=\frac{1}{h^2}(u_{n+1}-2u_n+u_{n-1}) - |u_n|^2 (u_{n+1}+u_{n-1}), \quad n\in \l, \ t\in\mathbb{R}, 
\end{equation}
whose formal continuous limit as $h \to 0$ is the celebrated completely integrable defocusing nonlinear Schr\"odinger (NLS) equation
\begin{equation}\label{nls}
	i \frac{\p u}{\p t} =  \frac{\p^2 u}{\p x^2} - 2|u|^2 u, \quad x\in\mathbb{R}, \ t\in\mathbb{R},
\end{equation}
which does \textit{not} exhibit blow-up. 

Specifically, we recall that the NLS equation \eqref{nls} is one of the prototypical examples of completely integrable nonlinear dispersive equations. As such, for initial data of sufficient smoothness and decay, it can be solved via the inverse scattering transform  \cite{zs1971, ac1991}. 
In the case of zero boundary conditions at infinity, a fundamental result of the inverse scattering transform analysis for \eqref{nls} is that the relevant spectral problem  exhibits a purely continuous spectrum on the real axis, with a complete absence of discrete eigenvalues. Consequently, since discrete eigenvalues correspond directly to soliton solutions, the defocusing cubic NLS equation \eqref{nls} with zero boundary conditions does \textit{not} possess any solitons, bright or dark; solutions exist globally in time and consist entirely of purely radiative, dispersive wave packets. 
In particular, concerning the fundamental issue of global existence and long-time asymptotics, the seminal work \cite{dz2003} established that global existence in time and dispersive dynamics hold for initial data $u_0 \in H^{1,1}(\mathbb{R}) = \left\{ f : f, f', xf \in L^2(\mathbb{R}) \right\}$
\textit{independently} of their size.

We also refer to \cite{c2003, t2006} for the dispersive dynamics of equation \eqref{nls} supplemented with $H^1(\mathbb{R})$ initial data, noting that the assumption $u_0\in H^{1,1}(\mathbb{R})$ guarantees the necessary smoothing properties of the reflection coefficient involved in the inverse scattering transform analysis and the associated Riemann-Hilbert problem.
On the other hand, in the case of symmetric, constant-in-modulus nonzero boundary conditions at infinity, the NLS equation~\eqref{nls} possesses the renowned dark/grey soliton solutions, and global existence is again guaranteed by the inverse scattering transform \cite{dpmv2013,cj2016} for arbitrarily large initial data such that its difference from the nonzero boundary values at $\pm \infty$ belongs to  $ L^{1,2}(\mathbb{R}_\pm) = \left\{ f : (1+|x|)^2 f \in L^1(\mathbb{R}_\pm) \right\}$.

Our results on finite-time singularities of the AL equation \eqref{al} cover the widest  class of physically significant boundary conditions, namely:
\begin{enumerate}[label=(\arabic*)]
\item \emph{zero boundary conditions at infinity};
	\item  \emph{nonzero boundary conditions at infinity};  
	\item \emph{periodic conditions} or, in the setting of a finite lattice, \emph{Dirichlet boundary conditions}.
\end{enumerate}
Cases (1) and (2) correspond to an infinite-dimensional dynamical system, while case (3) corresponds to a finite-dimensional one. 

We now outline the strategy for establishing the finite-time collapse for the AL equation \eqref{al}. In order to address the question in a unified manner across all of the aforementioned boundary conditions (1)-(3), we first prove  local well-posedness of \eqref{al} for initial data in $\ell^\infty$.
Although this is a departure from the standard Hilbert space setting of $\ell^2$ sequences, requiring the sequence $u = \{u_n\}_{n\in\l}$ to be in $\ell^\infty$ is the minimal assumption necessary for the solutions of the AL equation to accommodate all of the above boundary conditions (1)-(3). Indeed, our result covers all types of boundary conditions associated with uniformly bounded sequences as solutions of the corresponding  problems and can be stated as follows:
\begin{theorem}\label{lwp-t} 
\textnormal{(i) (Local well-posedness)} For any initial datum $u(0)\in \ell^\infty$, there exists a $T_{\max} = T_{\max}(u(0))>0$ such that the associated initial value problem for the AL equation \eqref{al} has a unique local solution   
$$
u\in C^1([0,T_{\max});\ell^\infty),
$$
which depends continuously on the initial condition with respect to the norm of $C([0,T],\ell^\infty)$ for any $0<T<T_{\max}$.
\\[1mm]
\textnormal{(ii)} In addition, the following alternative holds: either $T_{\max}=\infty$ (global existence), or $T_{\max}<\infty$ and $\displaystyle \lim_{t \to  T_{\max}}\no{u(t)}_{\ell^\infty}=\infty$ (blow-up in finite time). 
\end{theorem}
Theorem \ref{lwp-t} is proved in Section \ref{wp-s} and readily implies the uniform boundedness of each $u_n(t)$ over the maximal interval of existence $[0, T_{\max})$, allowing us to prove the existence of a unique solution to the initial value problem for the \textit{modulus equation}
\begin{equation}\label{mod_intro}
		\frac{d}{dt}|u_{n}(t)|^2 
        =
        2R_n(t)(|u_{n}(t)|^2-1), \quad n \in \l, \ t\in \r, 
\end{equation}
for each individual $|u_n(t)|$ and any initial datum $|u_n(0)|$ away from zero (since the nonlinearity in \eqref{mod_intro} is not locally Lipschitz at zero), 
where the important quantity $R_n(t)$ is given by
\eee{\label{rn}
R_n(t) := -\text{Im} \left[ \left(u_{n+1}(t)+u_{n-1}(t)\right)\overline{u}_{n}(t) \right].
}
This result, combined with the \emph{autonomy} of the AL equation in $\ell^\infty$,  enables us to establish the important structural fact that \textit{the unit circle $|u_n|=1$ is dynamically invariant at the level of each individual lattice site}. Consequently, the quantity $|u_n(t)|-1$  preserves its sign pointwise in $n$ for the lifetime of the solution (see Theorem~\ref{sep-t}).
An immediate consequence of this pointwise rigidity is our second main result on the   global well-posedness of \eqref{al} for  initial data in the closed unit ball of $\ell^\infty$:
\begin{theorem}[Global well-posedness]
\label{gwp-t}
Suppose that $\no{u(0)}_{\ell^\infty}\leq 1$. Then, the associated local solution $u\in C^1([0, T_{\max}), \ell^\infty)$ to the AL equation \eqref{al} guaranteed by Theorem~\ref{lwp-t} is in fact a global solution, i.e. $T_{\max}=\infty$.
\end{theorem}

Theorem \ref{gwp-t} is consistent with the fact that the \textit{strict} inequality $\no{u(0)}_{\ell^\infty} < 1$ coupled with the rapid decay condition
$
u(0) \in \ell^{2,1} := \left\{ u=\{u_n\}_{n\in\z}: \sum_{n\in\z} (1+|n|^2)|u_n|^2 < \infty \right\}
$
(with $\ell^{2,1}$ being the discrete analogue of the space $H^{1,1}(\mathbb{R}$) used in \cite{dz2003}) constitutes the standard set of assumptions for equation~\eqref{al} to be solvable via the inverse scattering transform method \cite{apt2004, XiaFokas18, chf2024} in the case of zero boundary conditions. However, note that Theorem \ref{gwp-t} holds for all initial data satisfying the \textit{non-strict} inequality $\no{u(0)}_{\ell^\infty} \leq 1$, namely even if finitely or infinitely many lattice sites initially lie on the unit circle. Moreover, importantly, Theorem \ref{gwp-t} is valid for both zero and nonzero boundary conditions of the three types (1)-(3) described earlier (rather than just zero boundary conditions).

Yet another implication of the pointwise sign invariance of $|u_n(t)|-1$ noted above is that \textit{the source of any potential blow-up is limited to the sites where $|u_n(0)|>1$}. The main purpose of this work is the rigorous proof of this finite-time collapse scenario (see Section \ref{bu-s}). 

The first step in this direction is the construction of an analytical, singular in time,  fundamental solution arising from the simplest possible configuration: a single site with a modulus greater than one, referred to as the \emph{central site}, symmetrically framed by two neighbors pinned to stationary states of the AL equation. The simplest version of such a configuration, which nevertheless permits the remaining lattice sites to occupy arbitrary positions in the complex plane, assumes that the two sites directly adjacent to the central one are pinned to the unit circle. In the specific case where all sites except the central one lie on the unit circle, this setup represents a discrete $\delta$-function perturbation of the unit circle. The construction of this fundamental solution in Proposition \ref{FSS}, where we take (only for simplicity and without any loss of generality) the central site at $n=0$, makes essential use of the conservation of the localized energy functional  
\eee{
\label{E_loc}
E_{\textnormal{loc}}(t) := \left(u_{-1}(t) + u_1(t)\right) \overline u_0(t)
+
\left(\overline u_{-1}(t) + \overline u_1(t)\right) u_0(t)
=
E_{\textnormal{loc}}(0)
}
together with the sign condition  $R_0(0)>0$, where $R_0$ is defined by \eqref{rn}. The blow-up time of the fundamental singular solution is obtained explicitly in terms of the arbitrary initial central modulus $|u_0(0)|>1$. 

We remark that conservation laws, sign or monotonicity conditions on functionals, and their behavior in the vicinity of unstable stationary states are, in general, the main ingredients for identifying energy levels that may lead to finite-time singularities in evolution equations \cite{b1977, b1978}. In particular, in the context of the \emph{focusing} NLS equation, we refer the reader to \cite{c2003, t2006, ss1999}. 

In the second step towards the proof of finite-time collapse, we establish that the above strategy is also effective in the case of the defocusing AL equation \eqref{al}. Although the above ingredients were already manifested in Proposition~\ref{FSS} with the conservation law \eqref{E_loc} and the relevant conditions on $|u_0(0)|$ and $R_0(0)$, in Proposition~\ref{bu-p} we enrich considerably the class of singular solutions by relaxing the invariance conditions only to the two sites adjacent  to the central one, giving rise to the conservation of the new functional 
\eee{
\label{Jfun}
J(t) := 
\left|u_1(t) + u_{-1}(t)\right|^2 (|u_0(t)|^2 - 1) - R_0(t)^2
=
J(0).
}
The sign of $J(0)$ enables us to determine the dynamics of $|u_0(t)|$ and $R_0(t)$ and to construct new singular solutions that also feature  explicit blow-up times. This new class of singular solutions embraces a broad family of initial solitonic-type profiles with various rates of decay at infinity (e.g. exponential or algebraic) on either zero or nonzero backgrounds, as long as they satisfy the invariance conditions at the immediate  neighbors. These conditions actually determine the width of the initial solitonic profiles for any arbitrary amplitude greater than one. 

Furthermore, the analysis based on the functional \eqref{Jfun} allows us to also address the following important question which emerges naturally from Theorem \ref{gwp-t}:
\begin{center}
\textit{Do global-in-time solutions exist for initial data such that $\no{u(0)}_{\ell^\infty} > 1$, i.e.  when the closed unit ball threshold for global existence is violated?} 
\end{center}
In Proposition~\ref{ge-p}, we prove that there exist energy levels for initial data satisfying the invariance conditions at the immediate neighbors that guarantee not only the global existence of solutions, but also the convergence of the central site, even with an arbitrary initial amplitude $|u_0(0)|>1$, to the unit circle. In terms of global existence, this result is fully consistent with the construction of analytical solutions of the defocusing AL equation satisfying nonzero boundary conditions, namely classes of Kuznetsov-Ma (KM) rogue wave solutions, Akhmediev breathers or solitons on a nonzero background greater than one, obtained either via the so-called Grammian method \cite{oy2014} or the inverse scattering transform \cite{op2019, bckop2025, cp2026}. %
In terms of our $\ell^\infty$ phase space setting, all of these non-singular solutions $u(t)$ satisfy $\no{u(0)}_{\ell^\infty} > 1$. However, we remark that, within the class of solutions on an arbitrarily large background, analytical solutions $u(t)$ that satisfy $\no{u(0)}_{\ell^\infty} > 1$ and blow up in finite time, such as the singular KM breathers and Akhmediev breathers of the defocusing AL equation, have also been constructed \cite{oy2014, op2019, bckop2025, cs2024}. We remark that the dichotomy between global existence and blow-up for the   generalized AL lattice on a nonzero background and its connection with the quasi-collapse of the generalized discrete NLS equation were investigated in~\cite{hkmm2026} (and \cite{hkmm2025} for the continuous setting) via appropriate proximity estimates (see also \cite{lcmmkk2025}).

Consequently, in $\ell^\infty$, which is the largest phase space in terms  of uniformly bounded sequences to be considered for the AL dynamical system, the third and final step of the proof of finite-time blow-up becomes of essential importance, as it investigates the following question:  
\begin{center}
\textit{In light of the  specific orbits (manifolds)  in $\ell^\infty$ defined by analytical solutions $u(t)$ that satisfy  $\no{u(0)}_{\ell^\infty} > 1$ and exist globally in time, what is the dominant dynamics of general initial data outside the closed unit ball of $\ell^\infty$?}
\end{center}
Employing the set of singular solutions constructed in Propositions \ref{FSS} and \ref{bu-p}, we establish that the answer to this question is that \textit{the dynamics is dominated by blow-up in finite time}: 
\begin{theorem}[Finite-time blow-up]
\label{mainthe1}
Given initial data $u(0)\in\ell^\infty$ with $|u_0(0)|>1$, consider the corresponding solution $u\in C^1([0,T_{\max});\ell^\infty)$ of the AL equation \eqref{al} with maximal lifespan $T_{\max}$, as guaranteed by Theorem \ref{lwp-t}. Let $\mathcal S$ be the set of singular solutions $u^s$ of Propositions~\ref{FSS} and \ref{bu-p} such that 
\eee{\label{cond}
|u_0(0)| > |u_0^s(0)| > 1,
}
with associated blow-up times $T_{\max}^s <\infty$. 
Then, the only scenario in which $u$ \textbf{may} exist globally in time is if, \textbf{for every} $u^s \in \mathcal S$, there exists an $\ve^s = \ve(u^s) >0$  such that  $|u_0(t)| < |u_0^s(t)|$ for all $t\in (T_{\max}^s-\ve^s, T_{\max}^s)$. In all other cases, $u$ blows up in finite time, i.e. $T_{\max} < \infty$ and $\displaystyle \lim_{t \to  T_{\max}} \no{u(t)}_{\ell^\infty} = \infty$.
\end{theorem}

\begin{remark}[General initial data]
Theorem \ref{mainthe1} covers the case of general initial data, in the sense of any data in $\ell^\infty$ except for those that correspond to specific orbits associated with analytical solutions that may persist for arbitrary time (for example, see Proposition \ref{ge-p} and other exact solutions with amplitude above unit modulus~\cite{oy2014, op2019, bckop2025, cp2026}). 
 \end{remark}
 
The scenarios covered by Theorem \ref{mainthe1} are illustrated by Figure \ref{seq}. We emphasize that the choice of the central site at $n=0$ is arbitrary and hence the result of Theorem \ref{mainthe1} applies  to general initial data that meet the hypothesis at a different site. 
The key idea behind the proof of Theorem~\ref{mainthe1} is that the set $\mathcal{S}$ serves as a family of subsolutions for general initial data satisfying the condition \eqref{cond}. This allows us to extend the blow-up dynamics well beyond the explicitly solvable examples that define the manifolds of globally existing solutions in the $\ell^\infty$ phase space. Moreover, when establishing the sole scenario that allows for global existence, we analyze the behavior of the integral curve $|u_0(t)|$ of a potentially non-singular solution and how it may escape from every $u^s \in \mathcal{S}$. This is achieved by exploiting the behavior of the analytical blow-up times of $u^s$ as the initial amplitude $|u_0^s(0)|$ approaches the unit circle (see Figure \ref{seq}).

A practical application of Theorem \ref{mainthe1} is the following: since, in general, the set $\mathcal{S}$ acts as a class of subsolutions for solutions satisfying the assumptions of Theorem \ref{mainthe1}, we expect that applying the blow-up criteria of $u^s$ to general initial data can serve as an \textit{empirical, heuristic ``rule of thumb''} for determining their singular dynamics. This allows us to select a suitable subsolution $u^s$ whose analytical blow-up time serves as an upper bound satisfying $T_{\max} < T^{s}_{\max}$.

The theoretical implications as well as the practical applications of Theorem \ref{mainthe1} are illustrated numerically in Section \ref{num-s} by utilizing suitable {\em time-adaptive} 
numerical schemes whose accuracy is also discussed in detail. The  corresponding  numerical studies are  performed with several representative initial conditions and illustrate the power of our theoretical results: all solutions emanating from initial data that satisfy the general blow-up criteria are indeed found to exhibit blow-up in finite time, consistent with (and, in fact, strictly earlier than) the blow-up time of the selected singular solution $u^s \in \mathcal{S}$ acting as their subsolution. We note that Section \ref{num-s} also provides further analytical results that complement and explain the numerical findings.

\section{Well-posedness in $\ell^\infty$ and the unit circle as a separatrix}
\label{wp-s}

In this section, we establish the local well-posedness of the  initial value problem of the AL equation~\eqref{al} in the case of initial data 
$
u(0):=(u_n(0))_{n\in\l}\in \ell^\infty$. 
Note that this class of data can accommodate any type of zero or nonzero boundary conditions at infinity, as outlined by cases (1)-(3) in the introduction. Then, we show that the unit circle acts as a separatrix for the dynamics at the level of individual lattice sites (Theorem \ref{sep-t}), with global well-posedness on the closed $\ell^\infty$ unit ball as an immediate consequence (Theorem \ref{gwp-t}).

\subsection{Local well-posedness}

We begin with the proof of the local well-posedness result of Theorem \ref{lwp-t}, which generalizes (in the defocusing regime) the corresponding result of \cite{hkmm2026} to any initial condition in $\ell^\infty$ (i.e. without any decay/smoothness requirements).

\begin{proof}[Proof of Theorem \ref{lwp-t}]
(i) We regard the AL equation \eqref{al} as the abstract ordinary differential equation 
\begin{equation}
	\label{fl1}
	\frac{du}{dt} = f(u),
   \quad
f(u):= i \left(-\Delta u+\mathcal{F} u\right),
\end{equation}
where $\Delta u = (\Delta u)_{n\in\l} := u_{n+1} - 2u_n+u_{n-1}$ is the discrete Laplacian and the nonlinear operator $\mathcal{F} u$ is defined by 
\begin{equation}
\label{op2}
\mathcal{F} u = (\mathcal{F} u)_{n\in\l}:=(u_{n+1}+u_{n-1})|u_n|^2. 
\end{equation} 
Integrating \eqref{fl1} with respect to $t$, we have
\begin{equation}
	\label{mildsl2}
	u(t)=e^{-i\Delta t}u(0)+i\int_{0}^{t}e^{-i\Delta(t-s)}\mathcal{F}u(s)ds,
\end{equation}

For the proof of local well-posedness, we shall apply the generalized Picard-Lindel\"{o}f theorem~\cite[Theorem~3.A, p. 78]{z1986} to the integral formula \eqref{mildsl2}. 
Indeed, note that $\Delta$ is a bounded linear operator from $\ell^\infty$ to  $\ell^\infty$, since for any $u \in \ell^\infty$ we have 
\begin{equation}
	\label{op1}
	\no{\Delta u}_{\ell^\infty}\leq 4\no{u}_{\ell^\infty}.
\end{equation}
In addition, $\Delta$ is globally Lipschitz continuous, since for any $u, v \in \ell^\infty$ we have 
\begin{equation}
	\label{op1a}
	\no{\Delta u-\Delta v}_{\ell^\infty}\leq 4\no{u-v}_{\ell^\infty}.
\end{equation}
The continuity of the operator $\Delta:\ell^\infty\rightarrow\ell^\infty$ implies that the propagator $e^{-i\Delta t}$ belongs to the space $C^{\infty}(\mathbb{R},\mathcal{B}(\ell^\infty))$, where $\mathcal{B}(\ell^\infty)$ is the space of linear bounded operators from $\ell^\infty$ to $\ell^\infty$ (see \cite[Section 1.2, p.~2]{ch1998}).
More specifically, writing 
$
e^{-i \Delta t} u = \sum_{n=0}^\infty \frac{\left(-it\right)^n}{n!} \Delta^n u_n 
$
and noting that $\ell^\infty$ is a Banach algebra, we use the linearity of $\Delta$ from $\ell^\infty$ to $\ell^\infty$ (see \eqref{op1}) to infer 
\aaa{
\no{e^{-i \Delta t}}_{\mathcal B(\ell^\infty)} 
\equiv
\sup_{\no{u}_{\ell^\infty}\leq 1} \no{e^{-i \Delta t} u}_{\ell^\infty} 
\leq
\sup_{\no{u}_{\ell^\infty} \leq 1}
\sum_{n=0}^\infty \frac{\left|-it\right|^n}{n!} \no{\Delta^n u}_{\ell^\infty} 
\leq
\sup_{\no{u}_{\ell^\infty} \leq 1}
\sum_{n=0}^\infty \frac{t^n}{n!} \no{\Delta}_{\mathcal B(\ell^\infty)}^n \no{u}_{\ell^\infty} 
\leq
\sum_{n=0}^\infty \frac{(4t)^n}{n!}.
\nn
}
As the right side converges absolutely to $e^{4t}$,  $e^{-i \Delta t}$ defines a continuous operator for any $t\in\r$. 
Thus, in order to invoke the generalized Picard-Lindel\"{o}f theorem for \eqref{mildsl2}, it only remains to establish the local Lipschitz continuity of the operator $\mathcal{F}:\ell^\infty\rightarrow\ell^\infty$.

Letting $B_R(0)$ denote the closed ball of center $0$ and radius $R$ in $\ell^\infty$, for any $u\in B_R(0)$ we have 
\begin{equation}
	\label{op3}
\no{\mathcal{F} u}_{\ell^\infty}\leq 2\no{u}_{\ell^\infty}^3\leq 2R^3, 
\end{equation}
hence, as an operator, $\mathcal{F}:\ell^\infty\rightarrow\ell^\infty$ is well defined and locally bounded. Moreover, for any $u, v \in B_R(0)$, 
\begin{align}
\mathcal{F} u - \mathcal{F} v
&=(u_{n+1}+u_{n-1})|u_n|^2-(v_{n+1}+v_{n-1})|v_n|^2\nonumber\\
&=(u_{n+1}+u_{n-1})(|u_n|^2-|v_n|^2)+|v_n|^2\left[(u_{n+1}-v_{n+1})
+(u_{n-1}-v_{n-1})\right]
\nn
\end{align}
thus, via the elementary inequality 
$
	\left| |u_n|^2-|v_n|^2\right| = (|u_n|+|v_n|)\big| |u_n|-|v_n|\big|\leq 
	(|u_n|+|v_n|)|u_n-v_n|$,
    we obtain
\begin{equation}
	\label{op5}
	\no{\mathcal{F} u-\mathcal{F} v}_{\ell^\infty}\leq
	2\left( \no{u}_{\ell^\infty}^2+\no{v}_{\ell^\infty}^2+ \no{u}_{\ell^\infty} \no{v}_{\ell^\infty}\right) \no{u-v}_{\ell^\infty}
\leq 6R^2\no{u-v}_{\ell^\infty},
\end{equation}
which shows that $\mathcal{F}:\ell^\infty\rightarrow\ell^\infty$ is indeed locally Lipschitz continuous.

Since the hypothesis of the aforementioned  generalized Picard-Lindel\"{o}f theorem is satisfied for the integral formula \eqref{mildsl2}, given any initial datum $u(0)\in \ell^\infty$ we can choose $T>0$ such that there exists a unique solution $u\in C^1([0, T]; \ell^\infty)$ to the initial value problem for the AL equation \eqref{al}. Furthermore, this solution depends continuously on $u(0)$ with respect to the norm of $C([0, T]; \ell^\infty)$. The maximal interval of existence is $[0, T_{\max})$, where $T_{\max}>0$ is the supremum of the $T$'s such that the solution exists on $[0, T]$.
\\[2mm]
(ii) The stated alternative follows from the seminal work of Segal \cite[Theorem 1, p. 343]{s1963}. %
\end{proof}

\begin{remark}\label{c1-r}
As a direct consequence of Theorem \ref{lwp-t}, $u_n \in C^1([0, T_{\max}))$ for all $n\in\l$. Indeed, since $u\in C^1([0, T_{\max}),\ell^\infty)$, we have $\lim_{t\to a} \left\| u(t) - u(a) \right\|_{\ell^\infty} = 0$, i.e. given $\ve>0$ there is a $\delta>0$ such that  
$
\left\| u(t) - u(a) \right\|_{\ell^\infty} < \ve$
whenever $|t-a|<\delta$. 
Hence, by the definition of the $\ell^\infty$ norm, for all $n\in \l$ we have $\left| u_n(t) - u_n(a) \right| < \ve$ whenever $|t-a|<\delta$, 
which shows that $u_n \in C([0, T_{\max}))$ for all $n\in\l$. Similarly, it follows that  $u_n \in C^1([0, T_{\max}))$ for all $n\in\l$.
\end{remark}

\subsection{The unit circle as a separatrix for the dynamics}
\label{sep-ss}

We begin by emphasizing that the AL equation \eqref{al} is an autonomous system. This is an important property that will be used crucially in our analysis and, in particular, when considering the steady states of the system and their role in the associated dynamics. Hence, as noted in \cite[Section 1.7]{h1980}, for any $a \in \r$, if $u(t)$ is a solution of~\eqref{al} on $[0, T]$, then $u(t-a)$ is a solution of the system on $[a, T+a]$. 
Due to the importance of this property for our arguments, we explicitly illustrate its validity for the AL system \eqref{al}.

Suppose we have a solution $z(t)\in C^1([0,T_{\max});\ell^\infty)$ that is defined on an interval $[a, T+a]$.  We define a new time variable $s$ and a new solution function $u(s)$ as follows. \uline{Time shift}: $s = t - a\Rightarrow t = s + a$; \uline{Solution shift}: $u(s) = z(s + a)$. 
Under this change of variables, the time $t = a$ corresponds to $s = 0$ and $u(0) = z(0 + a) = z(a)$. 
In order to verify that $u(s)$ satisfies the original AL equation \eqref{al}, we compute the derivative of $u$ with respect to the new time variable $s$. For each $n\in\l$, we have
$$
	\frac{du_n}{ds} = \frac{d}{ds}z_n(s + a)=\frac{d}{dt} z_n(s + a) \cdot \frac{dt}{ds} = \frac{d}{dt} z_n(s + a) \cdot  \frac{d}{ds}(s + a) = \frac{d}{dt} z_n(s + a).
$$
Thus, since $z(t)$ satisfies \eqref{al},  we obtain
\aaa{
	i\frac{du_n}{ds} &= \left[(z_{n+1}-2z_n+z_{n-1})- |z_n|^2  (z_{n+1}+z_{n-1})\right] \Big|_{t=s+a}
    \nn\\
    &\equiv\left(u_{n+1}(s)-2u_n(s)+u_{n-1}(s)\right)- |u_n(s)|^2  \left(u_{n+1}(s)+u_{n-1}(s)\right), 
    \nn
}
concluding that $u$ satisfies \eqref{al} on $[0, T]$.

Having illustrated the autonomy property of the AL system  \eqref{al}, we are ready to establish the following result, which demonstrates the critical role of the unit circle steady state as a threshold separating solutions within the AL dynamics. 

\begin{theorem}
\label{sep-t} 
Given $u(0) \in \ell^\infty$,  consider the corresponding local-in-time unique solution   $u\in C^1([0, T_{\max}), \ell^\infty)$ to the AL equation~\eqref{al} guaranteed by Theorem \ref{lwp-t}. 
\\[1.5mm]
\textnormal{(i)} 
If $|u_{n_*}(t)| = 1$ for some  $n_* \in \l$ and all $t\in (0, T_{\max})$, then $|u_{n_*}(0)| = 1$.
\\[1.5mm]
\textnormal{(ii)} 
Conversely, if $|u_{n_*}(0)| = 1$ for some $n_* \in \l$, then $|u_{n_*}(t)| = 1$  for all $t\in [0, T_{\max})$.
\\[1.5mm]
\textnormal{(iii)} \textnormal{(Unit modulus threshold)}
If $|u_n(0)|< 1$ (respectively $|u_n(0)|>1$), then $|u_n(t)| < 1$ (respectively $|u_n(t)|>1$) for all $t \in [0, T_{\max})$.
\end{theorem}

\begin{proof}
(i) 
Since $|u_{n_*}(t)|=1$ for some $n_* \in \l$ and all $t\in (0, T_{\max})$, and since $u_{n_*+1}(t)$, $u_{n_*-1}(t)$ exist and are bounded as $C^1$ functions of $t$ on the interval $[0, T_{\max})$ (see Remark \ref{c1-r}),  we may evaluate \eqref{al} at $n=n_*$ to deduce that the site $u_{n_*}$ satisfies the \textit{decoupled} linear ODE
\begin{equation}
\label{e:ALuD1}
i\frac{du_{n_*}}{d t}=-2u_{n_*},
\quad t\in (0, T_{\max}),
\end{equation}
whose general solution is given by
\begin{equation}
\label{e:ALuD2}
u_{n_*}(t) = e^{2it} C_{n_*},
\quad
t\in (0, T_{\max}),
\end{equation}
where $C_{n_*}\in\mathbb{C}$ is an arbitrary constant.

The fact that $u_n \in C^1[0, T_{\max})$ for all $n\in\l$ guarantees that $\lim_{t \to  0^+}u_{n_*}(t)=u_{n_*}(0)$. Hence, we may pass to the limit of the solution~\eqref{e:ALuD2} as $t \to  0^+$ to infer that $C_{n_*}=u_{n_*}(0)$ and, in turn,
\begin{equation}
\label{e:ALuD3}
u_{n_*}(t)=e^{2it} u_{n_*}(0), \quad t\in [0,T_{\max}).
\end{equation}
In particular, this yields $|u_{n_*}(t)|=|u_{n_*}(0)|$ for all $t\in [0,T_{\max})$, which readily implies $|u_{n_*}(0)|=1$.
\\[2mm]
(ii) 
In general, the evolution of each site $u_n(t)$ is associated with its initial condition $u_n(0)$ as well as the evolution of its neighboring sites $u_{n+1}(t)$, $u_{n-1}(t)$. We shall now prove that if $|u_{n_*}(0)| = 1$ for some $n_* \in \l$, then the evolution of $u_{n_*}(t)$ is not affected by its immediate neighbors $u_{n_*+1}(t)$, $u_{n_*-1}(t)$.

Letting 
\eee{\label{vdef}
u_n(t) = e^{2it} v_n(t)
}
turns equation \eqref{al} into 
\begin{equation}\label{al2}
i\frac{dv_n}{dt}=(v_{n+1}+v_{n-1})(1-|v_n|^2).
\end{equation}
Since $u\in C^1([0, T_{\max}), \ell^\infty)$ by Theorem \ref{lwp-t}, we have $v_{n+1}(t), v_{n-1}(t) \in C^1[0, T_{\max})$. Hence, recalling also that $u(0) \in \ell^\infty$, there exist constants $M_{n+1}(T_{\max}), M_{n-1}(T_{\max})$ such that
\begin{equation}
\label{BUFP5}
|v_{n+1}(t)|<M_{n+1}(T_{\max}),\;\;|v_{n-1}(t)|<M_{n-1}(T_{\max}),\quad t\in [0,T_{\max}).
\end{equation}

Now, consider the initial value problem for \eqref{al2}, namely
\eee{	\label{BUPf6}
	\frac{d}{dt}v_{n}(t) =i \left(v_{n+1}(t)+v_{n-1}(t)\right)(|v_{n}(t)|^2-1),
    \quad
     v_{n}(0)=v_{n,0},
     }
as well as the corresponding 	problem for the modulus $|v_{n}(t)|$:
\eee{
	\label{BUPf7} 
	\frac{d}{dt}|v_{n}(t)|^2=
	-2 \, \text{Im} \left[\left(v_{n+1}(t)+v_{n-1}(t)\right)\overline v_{n}\right](|v_{n}(t)|^2-1),
    \quad
	|v_{n}(0)|=|v_{n,0}|.
	}
	Letting $y_{n} = |v_{n}|^2$ and substituting $v_{n} = \sqrt{y_{n}} \, e^{i\theta_{n}}$ with $\theta_{n} = \arg(v_{n})$ allows us to write  \eqref{BUPf7} in the form
\eee{
\label{BUPf7b}  
\frac{d}{dt}y_{n}
=
g(v_{n-1}, v_{n+1}, \theta_{n}) \sqrt{y_{n}} \left(y_{n}-1\right),
\quad
y_{n}(0) = y_{n, 0},
}
where $g(v_{n-1}, v_{n+1}, \theta_{n}) := -2 \, \mathrm{Im}\left[(v_{n+1}+v_{n-1}) \, e^{-i\theta_{n}} \right]$. 
Note that $g$ is $C^1[0, T_{\max})$ as a function of $t$ and, from~\eqref{BUFP5}, 
\eee{
\left|g(v_{n-1}, v_{n+1}, \theta_{n})\right|
\leq
2 \left(|v_{n+1}|+|v_{n-1}|\right) 
\leq
2 \left[M_{n-1}(T_{\max}) + M_{n+1}(T_{\max})\right].
}
Moreover, the nonlinearity $\sqrt{y_{n}} \left(y_{n}-1\right)$ is locally Lipschitz except at $y_{n} = 0$. 

Therefore, for any initial condition $y_{n, 0}$ away from $0$, by the generalized Picard-Lindel\"{o}f Theorem \cite[Theorem~3.A]{z1986} the initial value problem \eqref{BUPf7b} has a unique solution $y_{n} \in C^1[0, T_{\max})$.
In particular, since $y_{n_*, 0} \equiv |v_{n_*, 0}|^2 = 1$, we infer that the unique solution of \eqref{BUPf7b} at $n=n_*$ is $y_{n_*}(t) \equiv 1$, namely 
\eee{
\left|v_{n_*}(t)\right| \equiv 1, \quad t\in [0, T_{\max}).
}
In turn, the initial value problem \eqref{BUPf6} at $n=n_*$ takes the form
\eee{\label{BUPf62}
\frac{dv_{n_*}}{dt} = 0,
\quad 
v_{n_*}(0) =v_{n_*,0},
}
implying that $v_{n_*}(t) = v_{n_*, 0}$ for all $t\in [0,T_{\max})$. In turn, $|u_{n_*}(t)| = |u_{n_*}(0)| = 1$ for all $t\in [0, T_{\max})$, as claimed. 
\\[2mm]
(iii) Given $n\in \l$, if $|u_n(0)|< 1$ but $|u_n(t)|$ were not strictly less than $1$ for all  $t \in (0, T_{\max})$, then by continuity in $t$ there would be a time $t_* \in (0, T_{\max})$ such that $|u_{n}(t_*)| = 1$. However, this would violate the uniqueness of solutions to the modulus equation \eqref{BUPf7b} for initial conditions away from $0$ (see proof of part (ii) above). Hence, it must be that  $|u_n(t)|< 1$ for all $t\in (0, T_{\max})$. 
A similar argument holds if $|u_n(0)|>1$.
\end{proof}

As a direct consequence of Theorem \ref{sep-t}, we have
\begin{corollary}[Unit circle as a separatrix]
\label{inv-r}
Given $u(0)\in \ell^\infty$, any infinite or finite lattice $\l$ associated with the AL equation \eqref{al} can be decomposed into $\l = \mathcal{N}_b\cup\mathcal{N}_a\cup \mathcal{N}_s$, where 
\eee{\label{n-def}
\mathcal{N}_b = \left\{n \in \l: |u_{n}(0)|<1\right\}, 
\quad
\mathcal{N}_s = \left\{n \in \l: |u_{n}(0)|=1\right\},
\quad
\mathcal{N}_a = \left\{n \in \l: |u_{n}(0)|>1\right\}.
}
Then, the relevant local solution $u\in C^1([0, T_{\max}), \ell^\infty)$ guaranteed by Theorem \ref{lwp-t} has the following property: for all $t\in [0, T_{\max})$, 
\textnormal{(i)} $|u_{n}(t)| < 1$ for all $n\in \mathcal{N}_b$, 
\textnormal{(ii)} $|u_n(t)| = 1$ for all $n\in\mathcal{N}_s$, and 
\textnormal{(iii)} $|u_n(t)|>1$  for all $n\in\mathcal{N}_a$.
\end{corollary}

\begin{remark}
The invariance of the open unit disc under the defocusing AL flow is, in the finite and square-summable settings, a classical fact rooted in
the theory of orthogonal polynomials on the unit circle (OPUC). There, the coefficients $u_n$ play the role of Verblunsky coefficients, which satisfy $|u_n|<1$ for any nontrivial probability measure on the unit circle, and equation \eqref{al} arises as an isospectral deformation of the associated unitary five-diagonal CMV matrix, which provides a Lax representation of the AL lattice \cite{Nen05,KN07}; see also \cite{Sim05} for the general OPUC framework. The factor $1-|u_n|^2$ appearing in equation~\eqref{mod_intro}, which is responsible for the invariance of the unit circle established below, also enters the non-canonical Poisson structure of the defocusing AL system~\cite{GN09}. Closely related is the Schur flow $\dot u_n=(1-|u_n|^2)(u_{n+1}-u_{n-1})$, introduced in~\cite{AG94} and studied in~\cite{FG99,Gol06,Sim07}, for which the same factor confines the dynamics to the unit disc. In contrast to these works, where the phase space is (a product of copies of) the open unit disc, Theorems~\ref{lwp-t},  \ref{gwp-t} and \ref{sep-t} hold in $\ell^\infty$ without any decay assumptions, and they also allow lattice sites to lie on or outside the unit circle; it is precisely the latter regime that gives rise to the blow-up and equilibration phenomena studied in this work.
\end{remark}

\subsection{Global well-posedness on the closed unit ball}

We now combine  Theorem \ref{sep-t} and  Corollary \ref{inv-r} in order to establish the global well-posedness result of Theorem~\ref{gwp-t}.
\begin{proof}[Proof of Theorem \ref{gwp-t}]
In view of Corollary \ref{inv-r}, the assumption $\no{u(0)}_{\ell^\infty}\leq 1$ on the initial data implies that $\l=\mathcal{N}_b\cup\mathcal{N}_s$ and $|u_{n}(t)|\leq 1$ for all $n\in \l$ and all $t\in [0, T_{\max})$. 
We argue by contradiction and suppose that the maximal time of existence is finite, i.e. $T_{\max}<\infty$. Then, according to the blow-up alternative of Theorem~\ref{lwp-t}, 
\begin{equation}
\label{arg1}
\lim_{t \to  T_{\max}}\no{u(t)}_{\ell^\infty}=\infty.
\end{equation}
By Theorem \ref{lwp-t}, for any $T \in [0, T_{\max})$ we have $u(T) \in \ell^\infty$.
Then, considering the AL equation \eqref{al} with initial condition $u(T)$, by the autonomy property stated in the beginning of Section \ref{sep-ss} we infer that  $u(t-T)$ is a solution to \eqref{al} that belongs to $C^1([T, 2T]; \ell^\infty)$. Furthermore, Corollary~\ref{inv-r} implies that $\no{u(t-T)}_{\ell^\infty} \leq 1$ for all $t\in [T, 2T]$. 
Thus, overall, there exists a solution $u \in C^1([0, 2T], \ell^\infty)$ with $\no{u(t)}_{\ell^\infty} \leq 1$ for all $t\in [0, 2T]$. 
Taking supremum over $T$ while recalling that $T_{\max}$  is defined in Theorem \ref{lwp-t} as the supremum of the times $T>0$ such that the solution exists on $[0, T]$, we deduce that $u \in C^1([0, 2T_{\max}), \ell^\infty)$ with $\no{u(t)}_{\ell^\infty} \leq 1$ for all $t\in [0, 2T_{\max})$, leading to a contradiction of the maximality of $T_{\max}$ as well as of \eqref{arg1}.
\end{proof}

\section{Blow-up in finite time}
\label{bu-s}

Theorem \ref{gwp-t} guarantees global existence for any initial datum such that $|u_n(0)|\leq 1$ for all $n\in \l$.
On the other hand, even in the case of an initial datum with $\mathcal{N}_a=\{n_a\}$ (recall the definition \eqref{n-def}), i.e. when just a single site   initially has $|u_{n_a}(0)|>1$, global existence is not guaranteed. Indeed, in this latter case, Corollary \ref{inv-r} only ensures that the modulus of the solution at each site in $\mathcal N_a$ remains outside the unit circle. 
An important question on the potential blow-up of solutions of evolution equations concerns \textit{when and how}. For the defocusing AL equation~\eqref{al}, according to Corollary \ref{inv-r} blow-up can be driven only by the sites $n \in \mathcal{N}_a$, since the sites $n\in\mathcal{N}_b\cup\mathcal{N}_s$ satisfy  $|u_n(t)|\leq 1$ for all $t\in (0, T_{\max})$ i.e. they  never escape the closed unit circle.  Therefore, blow-up can actually be manifested via point or localized singularities.  In this connection, we next construct explicit singular solutions with $\mathcal{N}_a=\{n_a\}$ that blow up in finite time. 

\subsection{The fundamental singular solution}
\label{bu-ss}

Suppose hereafter that $u(0) \in \ell^\infty$ with  $|u_{-1}(0)| = |u_1(0)| = 1$ (note, however, that the construction below can also be carried out for any initial configuration which guarantees that $|u_{\pm 1}|$ remain stationary). Then, by Theorem~\ref{sep-t}, $|u_{-1}(t)| = |u_1(t)| = 1$ for all $t \in [0, T_{\max})$. Equivalently, for $v(t)$ given by \eqref{vdef},  $|v_{-1}(t)| = |v_1(t)| = 1$ for all $t \in [0, T_{\max})$ and, more specifically, by \eqref{BUPf62} it follows that
$
v_{\pm 1}(t) = v_{\pm 1}(0) = \text{const.}$ for all $t\in [0, T_{\max})$.
Thus, evaluating equation~\eqref{al2}  at $n=0$, we obtain
\begin{equation}\label{al0}
\dot v_0(t) = -i \sigma \left(|v_0(t)|^2-1\right)
\end{equation}
where the overdot denotes differentiation in $t$ and the constant quantity $\sigma$ is defined by
\eee{\label{sigma}
\sigma := -\left(v_1(0) + v_{-1}(0)\right) = -\left(v_1(t) + v_{-1}(t)\right).
}

\begin{lemma}\label{eloc-l}
Suppose $v(0) \in \ell^\infty$ with  $|v_{-1}(0)| = |v_1(0)| = 1$. Then, the local energy functional 
\eee{
E_{\textnormal{loc}}(t)
:=
\left(v_{-1}(t)+v_1(t)\right)\overline v_{0}(t)
+
\left(\overline{v}_{-1}(t)+\overline{v}_1(t)\right)v_{0}(t) 
=-2\, \textnormal{Re} \left(\sigma \overline{v}_{0}(t)\right)
}
is conserved, i.e. $E_{\textnormal{loc}}(t)=E_{\textnormal{loc}}(0)$  for all $t\in [0, T_{\max})$.
\end{lemma}

\begin{proof}
By Theorem \ref{sep-t}, $v_{\pm 1}(t) = v_{\pm 1}(0) = \text{const.}$
Differentiating in $t$, we find
\aaa{
\dot{E}_{\text{loc}}(t) 
&= 
\left(v_{-1}(0)+v_1(0)\right)\overline{\dot{v}}_{0}(t)
+
\left(\overline{v}_{-1}(0)+\overline{v}_1(0)\right)\dot{v}_{0}(t)
=-2 \, \text{Re}\left(\sigma \overline{\dot{v}}_{0}(t)\right).\label{319}
}
In view of \eqref{al0}, the above becomes
\eee{\label{319b}
\dot E_{\textnormal{loc}}(t) 
=
-2 \, \text{Re} \left(i |\sigma|^2 \left(|v_0(t)|^2-1\right) \right) = 0,
}
which is the claimed result.
\end{proof}

\begin{corollary}\label{0-c}
Under the hypothesis of Lemma \ref{eloc-l} and the additional assumption that $\textnormal{Re} \left(\sigma \overline v_0(0)\right) = 0$, it holds that $\textnormal{Re} \left(\sigma \overline v_0(t)\right) = 0$ for all $t\in [0, T_{\max})$. 
\end{corollary}

\begin{proof}
In view of \eqref{319} and \eqref{319b}, we have $\text{Re} \left(\sigma \overline{\dot v}_0(t)\right) = 0$ for all $t \in (0, T_{\max})$. 
However, $\sigma$ given by \eqref{sigma} is constant, hence we have the ODE $\frac{d}{d t} \text{Re} \left(\sigma \overline  v_0(t)\right) = 0$. Since $\text{Re} \left(\sigma \overline v_0(0)\right) = 0$, the claim follows.
\end{proof}

With the above preparations, we are ready to construct our first singular solution, which we refer to as the \textit{fundamental singular solution}.
\begin{proposition}[Fundamental singular solution]
\label{FSS} 
Let $\sigma$ be defined by \eqref{sigma} and consider  an initial datum $u(0)\in \ell^\infty$ with the following properties: 
\begin{enumerate}[label=\textnormal{(\alph*)}]
\item $|u_0(0)|>1$;
\item $|u_{-1}(0)| = |u_1(0)| = 1$ and such that $\sigma \neq 0$;
\item $\textnormal{Re} \left(\sigma  \overline u_0(0)\right) = 0$ and $ \textnormal{Im} \left(\sigma \overline u_0(0)\right)>0$.
\end{enumerate}
Then, the well-posedness lifespan of Theorem \ref{lwp-t} is finite, namely $T_{\max}<\infty$. In particular,  
\eee{\label{ANTmax}
T_{\max}=T_{\max}^{\textit{fs}} := \frac{1}{2|\sigma|} \log\left(\frac{|u_0(0)|+1}{|u_0(0)|-1}\right)
} 
and $\displaystyle \lim_{t \to  T_{\max}^{\textit{fs}}}\big|u_0^{\textit{fs}}(t)\big| =\infty$, implying that $\displaystyle \lim_{t \to  T^{\textit{fs}}_{\max}} \big\|u^{\textit{fs}}(t)\big\|_{\ell^\infty}=\infty$.
\end{proposition}

\begin{proof} 
Combining \eqref{rn} with \eqref{vdef} and \eqref{sigma}, we have 
\eee{\label{r0t}
R_0(t) = \text{Im} \left(\sigma \overline v_{0}(t)\right).
}
Then, introducing the notation
\eee{\label{m0def}
M_0(t) = |v_0(t)|^2, \quad A_0(t) = M_0(t) - 1, 
}
we write the ODE in \eqref{BUPf7} at $n=0$ as
\eee{\label{m0-eq}
\dot M_0(t) 
=
2 R_0(t) A_0(t).
} 
By Theorem \ref{sep-t}, if $M_0(0)>1$ (equivalently, $A_0(0)>0$) then $M_0(t) > 1$ (equivalently, $A_0(t)>0$) for all $t\in [0, T_{\max})$. Thus, since by assumption (c) we have
\eee{\label{r0}
R_0(0) = \text{Im} \left(\sigma \overline v_{0}(0)\right) > 0, 
}
it follows that
\eee{\label{m0d}
\dot M_0(0) > 0.
}
Moreover, differentiating \eqref{r0t} and using \eqref{al0} yields
\eee{\label{r0-eq}
\dot R_0(t) 
= \text{Im}\left(\sigma \overline{\dot v}_0(t)\right) 
= \text{Im}\left(\sigma \cdot i \, \overline \sigma A_0(t)\right)  
= |\sigma|^2 A_0(t)
}
so $\dot R_0(t) > 0$ for all $t \in [0, T_{\max})$, which implies
\eee{
R_0(t) > R_0(0) > 0, \quad t \in (0, T_{\max}).
}

Next, recalling that $v_0 \in C^1[0, T_{\max})$ by Theorem~\ref{lwp-t} and so $M_0(t) \equiv |v_0(t)|^2 \in C^1[0, T_{\max})$, we  differentiate~\eqref{m0-eq} and use \eqref{r0-eq} to obtain
\eee{\label{mdd}
\ddot M_0(t) 
= 2 \dot R_0(t) A_0(t) + 2R_0(t) \dot A_0(t)
=
2 |\sigma|^2 A_0(t)^2 + 4R_0(t)^2 A_0(t).
}
Hence, thanks to our earlier observation due to Theorem \ref{sep-t}, it follows that $\ddot M_0(t) > 0$ for all $t \in [0, T_{\max})$, i.e. $\dot M_0(t)$ is strictly increasing. Thus, in view of \eqref{m0d}, 
\eee{\label{m0dp}
\dot M_0(t) > \dot M_0(0) > 0, \quad t \in (0, T_{\max}).
}
Furthermore, noting that
\eee{
|\sigma|^2 M_0(t) = \left|\sigma \overline v_0(t)\right|^2 =   \left[\text{Re}\left(\sigma \overline v_0(t)\right)\right]^2 + \left[\text{Im}\left(\sigma \overline v_0(t)\right)\right]^2,
}
we infer via Corollary \ref{0-c} (which can be employed thanks to the assumptions (b) and (c)) that 
\eee{\label{r0eq}
R_0(t)^2 \equiv \left[\text{Im}\left(\sigma \overline v_0(t)\right)\right]^2 = |\sigma|^2 M_0(t), \quad t \in [0, T_{\max}).
}
Hence, the ODE \eqref{mdd} can be written in the form
\ddd{\label{mdd2}
\ddot M_0(t) 
&=
2 |\sigma|^2 \left(M_0(t)-1\right)^2 + 4  |\sigma|^2 M_0(t)  \left(M_0(t)-1\right)
\\
&=
2 |\sigma|^2 \left[3 M_0(t)^2 - 4M_0(t) + 1\right].
}
Multiplying both sides by $\dot M_0(t)>0$ and integrating, we obtain
\eee{\label{m0dot0}
\dot M_0(t)^2 = 4  |\sigma|^2 M_0(t) \left(M_0(t)-1\right)^2, 
\quad t \in [0, T_{\max}),
}
after noting that the relevant constant of integration $K := \dot M_0(0)^2 - 4  |\sigma|^2 M_0(0) \left(M_0(0)-1\right)^2$ is equal to zero. Indeed, evaluating \eqref{m0-eq} at $t=0$ (recall that $M_0(t) \in C^1[0, T_{\max})$) and using \eqref{r0eq}, we have
\aaa{
K 
&=  4 R_0(0)^2 A_0(0)^2 - 4  |\sigma|^2 M_0(0) \left(M_0(0)-1\right)^2
\nn\\
&=  4 |\sigma|^2 M_0(0) \left(M_0(0)-1\right)^2 - 4  |\sigma|^2 M_0(0) \left(M_0(0)-1\right)^2
= 0.
}
Therefore, in light of \eqref{m0dp} and the fact that   $M_0(t) > 1$ for all $t \in [0, T_{\max})$,  
\eee{\label{m0dot}
\dot M_0(t) = 2 |\sigma| \sqrt{M_0(t)} \left(M_0(t)-1\right), 
\quad t \in [0, T_{\max}).
}

Integrating the separable ODE \eqref{m0dot}, we have
\eee{
\int_0^t \frac{\dot M_0(t')}{\sqrt{M_0(t')} \left(M_0(t')-1\right)} \, dt'
=
2 |\sigma| t.
}
Evaluating the integral on the left side via the substitution $m(t) = \sqrt{M_0(t)}$, we obtain
\eee{
\frac{\sqrt{M_0(t)}-1}{\sqrt{M_0(t)}+1}
=
\frac{\sqrt{M_0(0)}-1}{\sqrt{M_0(0)}+1}
\, e^{2 |\sigma| t}
}
which via \eqref{m0def} implies
\eee{
\label{egfs}
|v_0(t)| 
=
\frac{|v_0(0)|+1+ \left(|v_0(0)|-1\right)
e^{2 |\sigma| t}}{|v_0(0)|+1- \left(|v_0(0)|-1\right)
e^{2 |\sigma| t}},
}
whose denominator vanishes at $t = T_{\max} = T_{\max}^{\textit{fs}}$ given by \eqref{ANTmax}.
\end{proof}

\begin{remark}
\label{phaseFSS}
One of the simplest examples of initial data satisfying the assumptions (a)-(c) of Proposition \ref{FSS} is given by~\eqref{eq:initA1}, which is a $\delta$-perturbation of the unit circle. 
More generally, the assumption (c) in  Proposition \ref{FSS} can be characterized geometrically as follows. Denote the phases of $u_j(0)$ by $\theta_j$ so that,  at $t=0$, 
$
u_0=|u_0|e^{i\theta_0}$, $u_{\pm 1}(0)=e^{i\theta_{\pm 1}}$. 
Then, a straightforward calculation yields 
\begin{gather}
 \textnormal{Re} \left(\sigma \overline u_{0}(0)\right)=-2|u_0(0)|\cos \left(\frac{\theta_1-\theta_{-1}}{2}\right)\cos\left(\frac{\theta_1+\theta_{-1}}{2}-\theta_0\right), \nn\\
\textnormal{Im} \left(\sigma \overline u_{0}(0)\right)=-2|u_0(0)|\cos \left(\frac{\theta_1-\theta_{-1}}{2}\right)\sin\left(\frac{\theta_1+\theta_{-1}}{2}-\theta_0\right). 
\nn
\end{gather}
Hence, in order to simultaneously have $ \textnormal{Re} \left(\sigma \overline u_{0}(0)\right)=0$ and $ \textnormal{Im} \left(\sigma \overline u_{0}(0)\right)>0$, one needs either
\eee{\nn
\theta_0= \frac{\theta_1+\theta_{-1}}{2} - \frac\pi 2
\ \text{ and } \ \frac{\theta_1-\theta_{-1}}{2}\in \left(\frac{\pi}{2},\frac{3\pi}{2}\right]
}
or  
\eee{\nn
\theta_0=\frac{\theta_1+\theta_{-1}}{2}-\frac{3\pi}{2} \ \text{ and } \ \frac{\theta_1-\theta_{-1}}{2}\in \left[0,\frac \pi 2\right)\cup \left(\frac{3\pi}{2},2\pi\right].
}
\end{remark}
\subsection{A second class of singular solutions and global existence above unit modulus}
\label{4-s}
 Having constructed the fundamental singular solution of Proposition~\ref{FSS}, we now construct a second class of singular solutions that satisfy stationary conditions only at $n=\pm 1$. More specifically, we drop assumption~(c) of Proposition~\ref{FSS} and retain only assumptions~(a) and~(b), which ensure that the quantity $\sigma$ still satisfies~\eqref{sigma} and remains nonzero. The derivation of this new class is markedly different and also yields a globally existing solution \emph{with the initial modulus of a site outside the unit circle}, giving rise to a novel and interesting dynamical phenomenon: the \emph{equilibration} of the defocusing AL lattice.

The key for the derivation of the above results is \emph{the derivation of a conserved quantity involving $M_0(t)$ and $R_0(t)$.}  Noting that equations \eqref{m0-eq} and \eqref{r0-eq} only require the validity of assumptions (a) and (b) of Proposition~\ref{FSS}, we combine these two equations to rewrite the former one as 
\eee{\label{eqmodulus3}
\frac{d}{dt} \left( |\sigma|^2 M_0 - R_0^2 \right)
=
\frac{d}{dt} \left( |\sigma|^2A_0 - R_0^2 \right) = 0.		
}
Hence, we have the conservation law
\eee{
	\label{cons1}
	J(t) := |\sigma|^2 \left(M_0(t) - 1\right) - R_0(t)^2 = |\sigma|^2 \left(M_0(0) - 1\right) - R_0(0)^2 =: J(0).
}
This observation is key to the construction of the broader class of singular solutions presented below and was not made for the fundamental singular solution of Proposition \ref{FSS}, which is why additional assumptions were placed in that case.

We proceed now to the analysis of the \emph{evolution of $R_0$ and $v_0$.} Combining \eqref{r0-eq} and \eqref{cons1}, we obtain the ODE
\begin{equation}
	\label{eq6}
\dot{R}_0 = R_0^2 + J(0).
\end{equation}
Next, we discuss the behavior of the analytical solutions to the initial value problem for \eqref{eq6} in the case of a generic initial datum $R_0(0)\in \r$. We have three cases: $J(0) > 0$; $J(0) = 0$; $J(0) < 0$.
\\[2mm]
\uline{Case 1}: $J(0)>0$.
Setting $J(0) = a^2$ with $a = \sqrt{J(0)} > 0$ and integrating \eqref{eq6}, we find
\begin{equation*}
	\int \frac{1}{R_0^2 + a^2} \, dR_0 = t + C.
\end{equation*}
Evaluating the integral on the left side via substitution $R_0 = a \tan \theta$, we obtain
\begin{equation}
		\label{eq10}
		R_0(t) = a\tan\left(at + \arctan\left(\frac{R_0(0)}{a}\right)\right).
\end{equation}
\textit{Regardless of the sign of $R_0(0)$,} the solution \eqref{eq10}  blows up when the argument of the tangent becomes $\frac \pi 2$, i.e. at time 
\begin{equation}
	\label{eq11}
		t =T_{\max}^{J(0)>0} :=
        \frac{1}{\sqrt{J(0)}} \left[ \frac{\pi}{2} - \arctan\left(\frac{R_0(0)}{\sqrt{J(0)}}\right) \right],\quad R_0(0)\in\mathbb{R}.
\end{equation}

Returning to equation \eqref{al0} for $v_0$, we use \eqref{m0def} and \eqref{r0-eq} to rewrite it as
\begin{equation}
	\label{eq12}
	\dot v_0(t) =  -\frac{i \sigma}{|\sigma|^2} \dot{R}_0(t),
\end{equation}
which can be integrated to yield 
\begin{equation} 
		\label{eq13}
		v_0(t) = v_0(0) - \frac{i \sigma}{|\sigma|^2} \left(R_0(t) - R_0(0)\right).
\end{equation}
Thus, $v_0(t)$ also blows up at the finite time $T_{\max}^{J(0)>0}$ given by \eqref{eq11}. 
\\[2mm]
\uline{Case 2}: $J(0) = 0$.
Then, equation \eqref{eq6} reduces to $\dot R_0 = R_0^2$ and 
\begin{equation} 
		\label{eq14}
		R_0(t) = \frac{R_0(0)}{1 - R_0(0)t}, \quad R_0(0)\in\mathbb{R}.
\end{equation}

If $R_0(0)>0$, then  \eqref{eq14} blows up at
	\begin{equation}
		\label{eq15}
			t =T_{\max}^{J(0)=0} := \frac{1}{R_0(0)},  \quad R_0(0)>0.
	\end{equation}
Consequently, the analytical solution \eqref{eq13} for $v_0$ also becomes  singular at $T_{\max}^{J(0)=0}$ given by  \eqref{eq15}.	

    If $R_0(0) = 0$, then $R_0(t) \equiv 0$ and \eqref{eq13} yields the equilibrium solution $v_0(t) \equiv v_0(0)$, which exists globally in time.

 If $R_0(0)<0$, then \eqref{eq14}  exists globally in time and $\displaystyle \lim_{t\to\infty}R_0(t)=0$. Consequently, by \eqref{eq13} and \eqref{r0t}, 	
\eee{
		\lim_{t \to \infty}v_0(t)=  v_0(0) + \frac{i \sigma}{|\sigma|^2} R_0(0) 
= v_0(0) + \frac{i \sigma}{|\sigma|^2} \frac{1}{2i} \left(\sigma \bar v_0(0) - \bar \sigma v_0(0)\right) = \frac 12 \left(v_0(0) + \frac{\sigma^2}{|\sigma|^2} \bar v_0(0)\right) 
			=: v_{\text{eq}}, 
}
while \eqref{cons1} implies
	$$
	|v_0(t)|^2 = 1 + \frac{J(0)+R_0(t)^2}{|\sigma|^2} = 1 + \frac{R_0(t)^2}{|\sigma|^2}
	$$
and, in turn, 
\eee{
		\lim_{t \to \infty}|v_0(t)|= |v_{\text{eq}}| = 1.
}
\vskip 2mm
\noindent
\uline{Case 3}: $J(0) < 0$.
Setting $J(0) = -a^2$ with $a = \sqrt{|J(0)|} > 0$, the ODE \eqref{eq6} becomes
	\begin{equation}
		\label{eq6a}
		\dot{R}_0=\left(R_0+a\right)\left(R_0-a\right).
	\end{equation}
	The fact that $\dot{R}_0(t)>0$ for all $t\in [0, T_{\max})$, which is implied by \eqref{r0-eq} in light of Theorem \ref{sep-t} and assumption (a) of Proposition \ref{FSS}, confines the admissible range of solutions to the ODE \eqref{eq6a} to 
	\begin{equation}
		\label{eq6b}
		R_0\in (-\infty, -a)\cup (a, +\infty).
	\end{equation}
    
	According to the analysis of the 1D-flow of \eqref{eq6a} around the equilibria $R_0^{-}=-a$ and  $R_0^{+}=a$, if $R_0(0)<-a$ then $\lim_{t \to  \infty}R_0(t)=-a$  	while if $R_0(0)>a$ then $R_0(t)$ blows up at some finite $t\in (0, \infty)$.
It is worth noting that, in general, i.e. without the aforementioned restriction of $\dot{R}_0(t)>0$, the equilibrium $R_0^{-}=-a$ is asymptotically stable and the equilibrium  $R_0^{+}=a$ is unstable: if  $R_0(0)\in (-\infty,a)$ then $\lim_{t \to  \infty}R_0(t)=-a$ while if $R_0(0)>a$ then, as noted above, the solution has a finite time singularity.

In order to precisely determine the blow-up time occurring in the case of $R_0(0)>a$, we employ the substitution
$R_0=a \coth \theta$, which implies $dR_0 = -a \mathrm{csch}^2 \theta \, d\theta$ and is consistent with the restriction \eqref{eq6b} since the range of values of $\coth \theta$ is $(-\infty, -1) \cup (1, \infty)$.
Thus, recalling the hyperbolic identity $\coth^2 \theta - 1 = \mathrm{csch}^2 \theta$,  we integrate~\eqref{eq6a} to obtain 	$\theta(t) = \mathrm{arcoth}\left( \frac{R_0(0)}{a} \right) - at$ and, in turn, 
	\begin{equation}
		\label{eq6h}
R_0(t) = a \coth \left[ \mathrm{arcoth}\left( \frac{R_0(0)}{a} \right) - at \right], \quad |R_0(0)|>a.
	\end{equation}
	This analytical solution is singular when the argument of the hyperbolic cotangent becomes zero, i.e. at time 
	\begin{equation}\label{tjm}
		T_{\max}^{J(0)<0} = \frac{1}{a} \mathrm{arcoth}\left( \frac{R_0(0)}{a} \right).
	\end{equation}
	For $R_0(0) > a$ we have $T_{\max}^{J(0)<0}>0$, which is relevant to our analysis ($R_0(0)<-a$ implies $T_{\max}^{J(0)<0}<0$, which does not concern us). Moreover, using the identity 
	$\mathrm{arcoth}(x) = \frac{1}{2} \ln \left( \frac{x+1}{x-1} \right)$, $x>1$, 
	we may express \eqref{tjm} in a logarithmic form analogous to \eqref{ANTmax}, namely
	\begin{equation}
		T_{\max}^{J(0)<0} 
        = \frac{1}{2a} \log \left( \frac{R_0(0) + a}{R_0(0) - a} \right).
	\end{equation}
    
Summarizing our analysis for $J(0)<0$:
	\begin{itemize}
		\item  If $R_0(0)> \sqrt{|J(0)|}$, then the solution $R_0(t)$ has a finite time singularity at 
		\begin{equation}
			\label{eq6i}
				t =T_{\max}^{J(0)<0} =\frac{1}{2\sqrt{|J(0)|}} \log\left(\frac{R_0(0)+\sqrt{|J(0)|}}{R_0(0)-\sqrt{|J(0)|}}\right) , \quad R_0(0)>\sqrt{|J(0)|}.
		\end{equation}
		\item If $R_0 (0)< -\sqrt{|J(0)|}$, then $R_0(t)$ exists globally and converges to the equilibrium at $-\sqrt{|J(0)|}$, i.e. $\lim_{t \to \infty}R_0(t)=-\sqrt{|J(0)|}$. Consequently, equations \eqref{eq13} and \eqref{cons1} imply
		\eee{\label{v0lim}
			\lim_{t \to \infty}v_0(t) =  v_0(0) - \frac{i \sigma}{|\sigma|^2} \left(-\sqrt{|J(0)|} - R_0(0)\right):=v_{eq},
			\quad
			\lim_{t \to \infty}|v_0(t)| = 1.
            }
\item  The case $R_0 (0)= \pm \sqrt{|J(0)|}$, which implies $R_0(0)^2 = -J(0)$ since $J(0)<0$, is not possible since it would require that $|\sigma|^2 (|v_0(0)|^2 - 1) = 0$, i.e. either $|v_0(0)|=1$ or $\sigma = 0$. The first scenario is ruled out due to the assumption $|u_0(0)|>1$, while the second scenario implies $R_0(0) = J(0) = 0$, contradicting the assumption $J(0) < 0$.
\end{itemize}

The results of our construction can be summarized in two propositions, which describe the dichotomy between finite-time blow-up and global existence for the second class of solutions derived above. We discuss each proposition in a separate paragraph below. 

\subsubsection{Blow-up for the second class of singular solutions}
We begin with a summary of the blow-up results for the second class of singular solutions, accompanied by examples of initial data that lead to singularities. 

\begin{proposition}[Blow-up above unit modulus]\label{bu-p}
Suppose that the assumptions (a) and (b) of Proposition \ref{FSS} hold  true, namely $|u_0(0)|>1$ and $|u_{-1}(0)| = |u_1(0)| = 1$  with $\sigma \neq 0$, 
and define $R_0(0)$ and $J(0)$ by~\eqref{r0} and \eqref{cons1} respectively. Then, the   lifespan $T_{\max}$ of Theorem \ref{lwp-t} is finite in the following cases:
\begin{enumerate}[label=\textnormal{(\roman*)}]
		\item $J(0) > 0$ and $R_0(0) \in \mathbb{R}$, with $T_{\max}<\infty$ given by \eqref{eq11}; 
		\item $J(0) = 0$ and $R_0(0) > 0$, with $T_{\max} < \infty$ given by \eqref{eq15};
		\item $J(0) < 0$ and $R_0(0) > \sqrt{|J(0)|}$, with  $T_{\max}<\infty$ given by \eqref{eq6i}.
	\end{enumerate}
In all of the above cases, $\displaystyle \lim_{t \to  T_{\max}} \no{u(t)}_{\ell^\infty} = \infty$ and, in particular,
\eee{\label{bu}
		\lim_{t \to  T_{\max}} R_0(t) = \lim_{t \to  T_{\max}}|v_0(t)|=\infty.
}
\end{proposition}

\begin{remark}
	The fundamental singular solution of Proposition \ref{FSS} satisfies $R_0(0) > 0$, thus its blow-up is also implied via the results of Proposition \ref{bu-p}. See also Example 1 in Section \ref{num-s}.  
\end{remark}

We now provide two concrete examples of localized data for which equation \eqref{eq6} yields singular solutions of the types described in Proposition \ref{bu-p}. 
\\[2mm]
\textbf{a. Exponentially decaying initial data.}
For any $\ve > 0$, consider the initial datum
\begin{equation}\label{eq:example1id}
	u_n(0) = \alpha \, \sech(\beta n), \quad n \in \mathbb Z,
\end{equation}
where $\alpha \in \mathbb C$ satisfies $|\alpha| = 1+\ve$ and $\beta = \log\big(1+\ve + \sqrt{\ve\,(2+\ve)}\, \big)$. Then, since $\sech(\beta) = \frac{1}{1+\ve}$, it follows that $|u_{\pm 1}(0)| = |v_{\pm 1}(0)| = |\alpha \, \sech(\beta)| = 1$, as required by the hypothesis of Proposition \ref{bu-p}. In addition, $|u_0(0)| = |v_0(0)| = |\alpha| = 1+\ve > 1$. Moreover, recalling \eqref{r0} and \eqref{cons1}, we compute $R_0(0) = \text{Im}\left(-2\alpha \, \sech(\beta) \cdot \overline \alpha\right) = 0$ and
\eee{\label{j0sech}
	J(0) 
	= \left|-2\alpha \, \sech(\beta)\right|^2 \left(|\alpha|^2-1\right) 
	=
	4|\alpha|^2 \sech^2(\beta) \left(|\alpha|^2-1\right)
	=
	4\ve\, (2+\ve)
	> 0.
}
Hence, our example falls into case (i) of Proposition \ref{bu-p} and blow-up occurs at time
\eee{\label{eq:timesec}
	T_{\max}^{J(0)>0} =\frac{\pi}{4\sqrt{\ve\,(2+\ve)}}, \quad \ve>0.
}
\textbf{b. Rationally decaying initial data.} 
Consider the rational initial condition
\eee{\label{eq:example3id}
	u_n(0) = \frac{\alpha}{\beta + \gamma n^{2p}}, \quad \alpha, \beta, \gamma \in \mathbb C, \ p \in \n,
}
where $|\alpha| > |\beta|$ and $|\alpha| = |\beta+\gamma|$ so that 
$|v_0(0)| = \frac{|\alpha|}{|\beta|}>1$ and $|v_{\pm 1}(0)| = 1$, as required by the hypothesis of Proposition \ref{bu-p}. From \eqref{r0} and \eqref{cons1}, we have
\aaa{
	R_0(0) 
	&=
	\text{Im}\left(-\left(\frac{2\alpha}{\beta+\gamma}\right) \frac{\overline \alpha}{\overline \beta}\right)
	=
	-\frac{2|\alpha|^2}{|\beta|^2 \left|\beta+\gamma\right|^2}\, \text{Im}\left(\beta\left(\overline \beta + \overline \gamma\right)\right)
	=
	-\frac{2}{|\beta|^2}\, \text{Im}\left(\beta\overline \gamma\right),
	\\
	J(0) 
	&=
	\frac{4|\alpha|^2}{|\beta+\gamma|^2} \left(\frac{|\alpha|^2}{|\beta|^2}-1\right)
	-
	\frac{4}{|\beta|^4} \left(\text{Im}\left(\beta\overline \gamma\right)\right)^2
	=
	\frac{4}{|\beta|^4} \left(|\alpha|^2 |\beta|^2 - |\beta|^4 - \left(\text{Im}(\beta \overline \gamma)\right)^2\right).
}
In the special case that $\beta \overline \gamma \in \r$,  we have $R_0(0) = 0$ and $J(0) = \frac{4}{|\beta|^2}\left(|\alpha|^2-|\beta|^2\right)>0$, which places us in case (i) of Proposition \ref{bu-p}  with blow-up time
\eee{\label{eq:maxtrat}
	T_{\max}^{J(0)>0}
	=
	\frac{\pi |\beta|}{4 \sqrt{|\alpha|^2-|\beta|^2}}.
}

\subsubsection{Global existence above unit modulus and equilibration}
We continue with a summary of the global existence results stemming from the analysis of the second class of singular solutions derived earlier.
\begin{proposition}[Global existence above unit modulus]\label{ge-p}
Suppose that the assumptions (a) and (b) of Proposition~\ref{FSS} hold  true, namely $|u_0(0)|>1$ and $|u_{-1}(0)| = |u_1(0)| = 1$  with $\sigma \neq 0$, and define $R_0(0)$ and $J(0)$ by \eqref{r0} and~\eqref{cons1} respectively. In addition, suppose that $\l \setminus \{0\} = \mathcal N_s \cup \mathcal N_b$. Then, we have the following scenarios of infinite lifespan $T_{\max} = \infty$ of Theorem \ref{lwp-t} and associated dynamics:
\begin{enumerate}[label=\textnormal{(\roman*)}]
	\item
	$J(0) = 0$ and $R_0(0) < 0$, or $J(0) < 0$ and $R_0(0) < -\sqrt{|J(0)|}$. Then, 
	\eee{\label{eq20}
	\lim_{t \to  \infty} R_0(t) =  -\sqrt{|J(0)|}, 
		\quad
		\lim_{t \to \infty}|v_0(t)|=1,
	}
	namely $v_0(t)$ converges to the unit circle along a linear segment in the complex plane defined by the direction of $-i\sigma$ (see \eqref{v0lim}).
    \item $J(0) = R_0(0) = 0$. Then, $v_0(t) = v_0(0)$ for all $t>0$. In fact, this scenario only occurs in the case of $\sigma = 0$, i.e. after dropping the relevant restriction from assumption (b) of Proposition \ref{FSS}.
	\end{enumerate}
In all of the above cases, $\no{u(t)}_{\ell^\infty} <  \infty$.
\end{proposition}

The conditions on the initial value at $n=0$ in Proposition~\ref{ge-p}, which require a modulus strictly greater than one, can also be applied at other lattice sites. This yields initial configurations distinct from that of Proposition~\ref{ge-p} that nevertheless generate global solutions to \eqref{al}, even when multiple sites initially exceed unit modulus. Importantly, case (i) of Proposition~\ref{ge-p} gives rise to the  novel 
phenomenon of \emph{equilibration}, which is discussed next.
\subsubsection{The equilibration phenomenon and its rigorous interpretation}
 We define the \emph{dynamical regime of equilibration} as the scenario in which lattice sites initially residing away from the equilibrium state on the unit circle asymptotically approach that state as $t\rightarrow\infty$. Remarkably, this behavior occurs even for sites with initial modulus strictly greater than one and with immediate neighbors pinned to the unit circle, thus preventing any transfer of excess energy into them.

We now elucidate the mechanism by which this asymptotic equilibration established in Proposition~\ref{ge-p}(i) (also referred to as \emph{global existence and convergence to the unit circle}) is achievable, despite the central site at $n=0$ being dynamically isolated from exchanging energy with its neighbors.
\begin{proposition}
\label{equib}
Let the conditions of Proposition \ref{ge-p} hold and consider the local energy functional $E_{\mathrm{loc}}$ defined in Lemma \ref{eloc-l}. Then, we have the following dynamical scenarios:
\\[1.5mm]
\textnormal{(i)}
\emph{Zero initial local energy $E_{\mathrm{loc}}(0)= 0$:} If $|u_0(0)|>1$ and
	\begin{equation}
		\label{condconst0}
		\cos\bigl(\arg u_0(0) - \arg u_{-1}(0)\bigr) + \cos\bigl(\arg u_0(0) - \arg u_1(0)\bigr) = 0,
	\end{equation}
	then $\frac{d}{dt} \arg u_0(t) \equiv 2$. This scenario describes an asymptotic state of equilibration manifested by the exact linear growth of $\arg u_0(t)$ in the limit $|u_0(t)| \to 1$.
\\[2mm]
\textnormal{(ii)}
\emph{Nonzero local initial energy $E_{\mathrm{loc}}(0)\neq 0$:} If $|u_0(0)|>1$ and
	\begin{equation}
		\label{condconst}
		\cos\bigl(\arg u_0(0) - \arg u_{-1}(0)\bigr) + \cos\bigl(\arg u_0(0) - \arg u_1(0)\bigr)  \neq 0,
	\end{equation}
	then $\frac{d}{dt} \arg u_0(t)\to 2$ as $|u_0(t)| \to 1$. This scenario describes an asymptotic state of equilibration manifested by a nonlinear evolution of $\arg u_0(t)$ while $|u_0(t)|$ is far from the unit circle, which relaxes into an asymptotic linear growth of $\arg u_0(t)$ in the limit $|u_0(t)| \to 1$.
\end{proposition}
\begin{proof}
To analyze the process, we revisit the local energy functional $E_{\mathrm{loc}}(t)$ from Lemma~\ref{eloc-l}. Under the assumptions of Proposition~\ref{ge-p}, the sites at $n=\pm 1$ remain pinned to the unit circle for all times, so $v_{\pm 1}(t) = v_{\pm 1}(0)$. Expressing the state at sites $n=0, \pm 1$ in polar coordinates, $v_n(t) = r_n(t)e^{i\theta_n(t)}$, the local energy functional simplifies to
\begin{equation}
	\label{Eloc-pol}
	E_{\mathrm{loc}}(t) = 2 r_0(t) \left[ \cos\bigl(\theta_0(t) - \theta_{-1}(0)\bigr) + \cos\bigl(\theta_0(t) - \theta_1(0)\bigr) \right],
\end{equation}
which is strictly conserved according to Lemma~\ref{eloc-l}, i.e., $E_{\mathrm{loc}}(t) = E_{\mathrm{loc}}(0)$.

Because the adjacent sites $v_{\pm 1}(t)$ remain stationary, \emph{energy transfer between the central site $v_0$ and its neighbors $v_{\pm 1}$ is entirely blocked}. Consequently, the central site cannot distribute its excess energy into neighboring sites to reduce its amplitude. Yet, as Proposition~\ref{ge-p}(i) asserts, $|v_0(t)| = r_0(t) \to 1$ as $t \to \infty$.

The resolution to this apparent paradox lies in the nonlinear dynamical interplay between the phase angle $\theta_0(t)$ and the amplitude $r_0(t)$. As $t \to \infty$, the amplitude $r_0(t)$ continuously decreases from $r_0(0) > 1$ toward the unit circle ($\lim_{t \to \infty} r_0(t) = 1$), while the phase $\theta_0(t)$ changes accordingly ensuring strict conservation of $E_{\mathrm{loc}}$. In fact, due to the conservation of local energy, $E_{\mathrm{loc}}(t) = E_{\mathrm{loc}}(0)$, its total time derivative vanishes yielding the balance relation
\begin{equation*}
	\frac{\partial E_{\mathrm{loc}}}{\partial r_0} \, \frac{dr_0}{dt} = -\frac{\partial E_{\mathrm{loc}}}{\partial \theta_0} \, \frac{d\theta_0}{dt}.
\end{equation*}
Hence, variations in local energy $E_{\mathrm{loc}}$ due to changes in amplitude and phase angle precisely compensate each other, preserving total local energy.

To analyze this interplay between the phase angle and amplitude for the site at $n=0$ in greater detail, we examine their equations of motion. The AL equation \eqref{al2} at $n=0$ 
in terms of polar coordinates yields
\begin{align}
	\dot{r}_0 &= (r_0^2 - 1)\left[ \sin(\theta_0 - \theta_{-1}) + \sin(\theta_0 - \theta_1) \right],\label{amplcond} \\
	\dot{\theta}_0 &= \frac{r_0^2 - 1}{r_0} \left[ \cos(\theta_0 - \theta_{-1}) + \cos(\theta_0 - \theta_1) \right].\label{phasecond}
\end{align}
By \eqref{Eloc-pol}, the above ODEs can be written as
\eee{
\label{prer0}
	\dot{r}_0 
    = 
    -\frac{r_0^2 - 1}{2r_0} \frac{\partial E_{\mathrm{loc}}}{\partial \theta_0},
\quad
	\dot{\theta}_0 
	= E_{\mathrm{loc}}(0)   \frac{r_0^2-1}{2r_0^2},
}	
which can be further simplified and we arrive at the final version of the ODEs for $r_0(t)$ and $\theta_0(t)$:
\begin{align}
	&\frac{d}{dt}\ln\bigl(r_0^2(t)-1\bigr) = -\frac{\partial E_{\mathrm{loc}}}{\partial \theta_0}, \label{rt1} \\
	&\frac{d}{dt}\theta_0(t) = E_{\mathrm{loc}}(0)  \, \frac{r_0^2(t)-1}{2r_0^2(t)}. \label{thetat1}
\end{align}
We infer from \eqref{rt1} that the monotonic decrease of the amplitude $r_0(t)$, resulting in $r_0(t) \to 1$ as $t \to \infty$, goes along with a positive partial derivative of local energy $E_{\mathrm{loc}}$ with respect to the phase angle $\theta_0$.  Recalling that $u_n(t) = e^{2it} v_n(t)$, the ODE \eqref{thetat1} allows us to distinguish the stated scenarios (i) and (ii) for $\arg u_0(t)$:
\\[2mm]
\emph{(i) Zero initial energy $E_{\mathrm{loc}}(0)= 0$}: If $|u_0(0)| \equiv r_0(0)>1$ and the condition \eqref{condconst0} holds
such that $E_{\mathrm{loc}}(0)=0$, then \eqref{thetat1} shows that $\dot{\theta}_0(t) \equiv 0$ i.e. $\frac{d}{dt} \arg u_0(t) \equiv 2$. This proves the asserted scenario (i).
\\[2mm]
\emph{(ii) Nonzero initial energy $E_{\mathrm{loc}}(0)\neq 0$}: If $|u_0(0)| \equiv r_0(0)>1$ and \eqref{condconst} holds such that $E_{\mathrm{loc}}(0)\neq 0$, then \eqref{thetat1} shows that $\dot{\theta}_0(t) \to 0$ as $r_0(t) \to 1$ i.e. $\frac{d}{dt} \arg u_0(t) \to 2$ as $|u_0(t)| \to 1$. This proves the claimed scenario (ii).
\end{proof}
\begin{remark}
We emphasize that the occurrence of equilibration is a distinct peculiarity of the AL lattice, arising because its Hamiltonian depends solely on intersite coupling (kinetic energy) explicitly depending on the phase differences between neighboring sites, rather than on an independent local potential. Consequently, even for a site dynamically isolated from its adjacent neighbors as described above, suitable initial conditions enable the coupled evolution of its two degrees of freedom, the phase angle and amplitude, to proceed such that the amplitude monotonically decreases toward unity while local energy is conserved through compensating phase dynamics.

Conversely, this equilibration mechanism ceases to exist in systems whose Hamiltonian includes an explicit local potential term in addition to the kinetic coupling, such as the defocusing cubic Discrete Nonlinear Schrödinger (DNLS) lattice, where the presence of local potential terms disrupts the phase-amplitude compensation mechanism that drives equilibration.
\end{remark}

\subsection{The case of general initial data}

We now establish the central result of Theorem \ref{mainthe1} on the finite-time blow-up of the AL equation~\eqref{al} in the case of general initial data in $\ell^\infty$ with the property that at least one lattice point has modulus greater than one.  
   
\begin{proof}[Proof of Theorem \ref{mainthe1}]
Given any initial condition $u(0) \in \ell^\infty$ with $|u_0(0)| > 1$, we first note that the set $\mathcal S$ is non-empty as there are infinitely many singular solutions satisfying \eqref{cond}. 
For any $u^{s} \in \mathcal S$ satisfying \eqref{cond} with associated blow-up time $T_{\max}^s < \infty$, according to Propositions \ref{FSS} and \ref{bu-p} we have 
\begin{equation}
\label{th1}
\lim_{t \to  T_{\max}^{s}}|u_0^{s}(t)| =\infty
\ \Rightarrow \
\lim_{t \to  T_{\max}^{s}} \no{u^{s}(t)}_{\ell^\infty}=\infty.
\end{equation}
The condition \eqref{cond} combined with the reverse triangle inequality readily implies
\eee{\label{th2fund1}
\no{u(0)-u^{s}(0)}_{\ell^\infty}
\geq 
\left|u_0(0) -  u_0^{s}(0)\right| > 0, 
}
i.e. $u(0)\neq u^{s}(0)$ as sequences in $\ell^\infty$.
In turn, the uniqueness of Theorem \ref{lwp-t} implies that $u(t)\neq u^{s}(t)$ as sequences in $\ell^\infty$ for all $t\in [0, \min\{T_{\max}, T_{\max}^{s}\})$,  i.e. 
\begin{equation}
	\label{th2fund3}
	\no{u(t)-u^{s}(t)}_{\ell^\infty} > 0 \quad \forall t\in [0, \min\{T_{\max}, T_{\max}^s\}).
\end{equation}
\uline{Case 1}: $T_{\max} \leq T_{\max}^s$ for some $u^s \in \mathcal S$. 
Then, since $T_{\max}^{s}<\infty$, it follows that $T_{\max}<\infty$. Consequently,  Segal’s alternative \cite[Theorem 1, p. 343]{s1963} implies blow-up in finite time:
\begin{equation}
\label{th2.1}
\lim_{t \to  T_{\max}} \no{u(t)}_{\ell^\infty}=\infty.
\end{equation}  
\uline{Case 2}: $T_{\max}^s < T_{\max}$ for all $u^s \in \mathcal S$. 
First, let us recall that the solution $u^{s}(t)$  is not defined at $T^{s}_{\max}$ due to the blow-up singularity of $u_0^{s}(t)$. Recall also that $|u_0(0)| > |u_0^{s}(0)|$ by \eqref{cond}. If there exists $u^{s} \in \mathcal S$ such that
 \eee{\label{u0up}
|u_0(t)| > |u_0^{s}(t)| \quad \forall t \in [0, T_{\max}^{s}), 
 }
then the blow-up of $u_0^{s}(t)$ at $T_{\max}^{s}$ readily implies blow-up at $T_{\max} \leq T_{\max}^{s}$ for $u_0(t)$, contradicting the fact that $T_{\max}^{s}<T_{\max}$. 
Hence, \eqref{u0up} does not hold. 
Then, since $|u_0(0)| > |u_0^{s}(0)|$, by continuity in $t$ (which is guaranteed by the local well-posedness of Theorem \ref{lwp-t}) we infer the existence of a $T^* \in (0, T_{\max}^{s})$ such that 
\eee{
|u_0(T^*)| = |u_0^{s}(T^*)|.
} 

\begin{figure}[h!]
	\begin{center}
\includegraphics[scale=0.68]{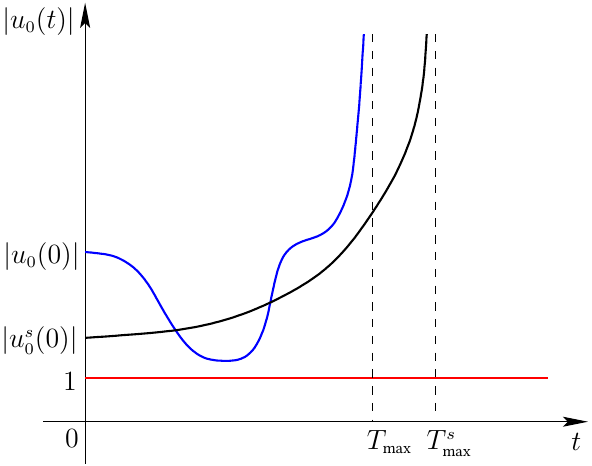}
\hskip 5mm
\includegraphics[scale=0.7]{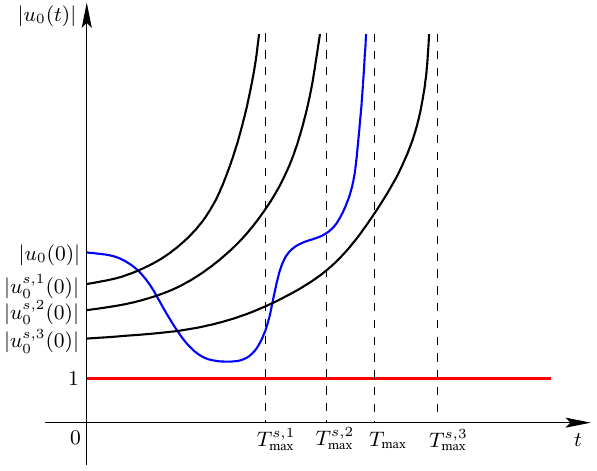}
\\[2mm]
\hspace{-1.5cm}\includegraphics[scale=1]{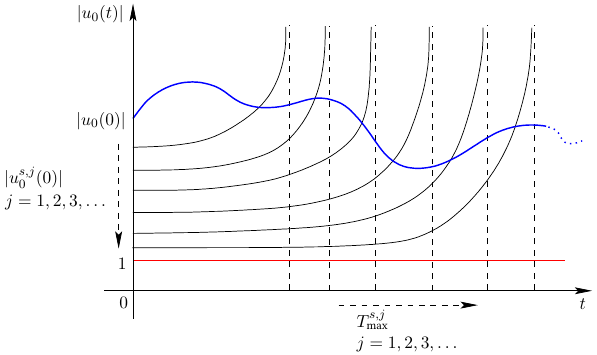}
	\end{center}
	\caption{Top left: The blow-up described in case 1 in the proof of Theorem \ref{mainthe1}. The black curve emanating from $|u_0^s(0)|$ corresponds to a subsolution which blows up at $T_{\max}^s$, leading to the blow-up of the solution emanating from $|u_0(0)|$ (blue curve) at time $T_{\max} < T_{\max}^s$. Note that, in general, the two curves may or may not intersect prior to $T_{\max}$. 
    Top right: Blow-up according to case 1 in the proof of Theorem \ref{mainthe1} in the scenario where, while the solution ``escapes'' from finitely many singular solutions, it is eventually trapped by a singular solution which acts as a 
    subsolution leading to collapse in finite time. 
     Bottom: The \textit{infinite} sequence of singular solutions $u^{s, j}$ that may lead to a global solution, as described in Case 2(ii) in the proof of Theorem \ref{mainthe1}.}
	\label{seq}
\end{figure}

\noindent There are two scenarios to consider:
\begin{enumerate}[label=(\roman*), itemsep=2mm, leftmargin=7mm]
\item There exists a $u^{s} \in \mathcal S$ satisfying \eqref{cond} and such that  $u(T^*) = u^{s}(T^*)$. Then, by uniqueness, starting at $t=T^*$ the two solutions would have to be the same for $t>T^*$. However, by time translation, this would imply that the two solutions coincide also for $t \in [0, T^*)$ and, in particular, at $t=0$, leading to a contradiction since $u(0) \neq u^{s}(0)$. 
\item
$u(T^*) \neq u^{s}(T^*)$ for all $u^{s} \in \mathcal S$ satisfying \eqref{cond}. Then, the only scenario in which blow-up  \textbf{may not} occur is the following:
\textit{for every $u^{s} \in \mathcal S$ satisfying \eqref{cond}, there exists an
$\ve = \ve(u^{s}) >0$  such that  $|u_0(t)| < |u_0^{s}(t)|$ for all $t\in (T_{\max}^{s}-\ve, T_{\max}^{s})$,
so that blow-up for $u_0$ may not readily be inferred from the blow-up of $u_0^{s}$.}  

This possibility is demonstrated as follows. One can construct a sequence  $u^{s, j}(t)$, $j\in \n$, of singular solutions in $\mathcal S$ with corresponding blow-up times  $T_{\max}^{s, j} < \infty$ and satisfying \eqref{cond}, which possess the following property: 
for each $j\in \n$, there exists an $\ve_j = \ve(u^{s, j})>0$ such that 
\eee{
|u_0(t)| < |u_0^{s, j}(t)|, \quad 
t \in (T_{\max}^{s, j} - \ve_j, T_{\max}^{s, j}).
}
Indeed, as shown in Figure \ref{seq}, since $|u_0(t)|>1$ for all $t\in [0, T_{\max})$, we can always fit a $|u_0^{s, j}(t)|$ between $1$ and $|u_0(t)|$ for an appropriate $j\in\n$. 
Through this construction, global existence cannot be ruled out due to the observation that, as $j\to \infty$, $|u_0^{s, j}(0)| \to 1^+$ and so, according to the expressions \eqref{ANTmax}, \eqref{eq11},~\eqref{eq15} and~\eqref{eq6i} for the blow-up times of the singular solutions $u^s \in \mathcal S$, 
we have $\lim_{j \to \infty} T_{\max}^{s, j} = \infty$ thus not allowing us to deduce blow-up for $u$.
\end{enumerate}
\vspace*{-6mm}
\end{proof}

\section{Numerical studies}
\label{num-s}
The class of analytical solutions $\mathcal{S}$ and the general blow-up result of Theorem~\ref{mainthe1} deserve detailed numerical studies aimed at illustrating the strength of the theoretical results.

To this end, we numerically integrate the defocusing AL equation on a periodic lattice of $N$ sites, using initial data that satisfy the blow-up conditions described in Propositions \ref{FSS} and \ref{bu-p} to illustrate the exactness of their analytical blow-up times. We also examine the dynamics emanating from representative general initial conditions to demonstrate the sharpness of Theorem \ref{mainthe1} and to show that the class of singular solutions $\mathcal{S}$ practically serves as subsolutions for the solutions initiated from these general data providing upper bounds for their blow-up times.  Furthermore, we numerically investigate Proposition \ref{ge-p} and highlight a key phenomenon: the domination of blow-up even within sublattices that otherwise support stability and convergence to equilibrium. This final example underlines  the combined theoretical mechanisms established in the preceding sections.

Numerical studies of blow-up phenomena are inherently challenging \cite{ss1999}. In particular, accurately determining the maximum time of existence remains a formidable task, as true blow-up must be reliably distinguished from numerical instabilities. For this reason, we outline the numerical scheme used in our simulations, noting that the class of singular solutions $\mathcal{S}$, with its exact analytical blow-up times, serves as a valuable benchmark for validating the numerical scheme in detecting finite-time singularities.
 
The numerical scheme is based on \eqref{al2} and proceeds as follows.   Decomposing $v_n$ into its real and imaginary parts recasts \eqref{al2} as a system of $2N$ real autonomous ODEs, which we integrate via the method of lines. To respect  the Hamiltonian and integrable structure of the AL equation, we discretize in time using the fourth-order, two-stage Gauss-Legendre implicit Runge-Kutta method. Being both $A$-stable and symplectic, this scheme preserves the quadratic invariants of \eqref{al2} (and consequently of \eqref{al}) to high precision. The solution of \eqref{al} is recovered by the relation $u_n = e^{2it}\,v_n$.

At each time step, the implicit Gauss-Legendre stage equations yield a nonlinear system $F(k)=\mathbf{0}$ for the stage derivatives $k$, which we solve using the complex-step Newton–Krylov framework of \cite{Mitsotakis2025}. The required Jacobian–vector products are evaluated via complex-step differentiation
\begin{equation}
  \mathcal J(\mathbf{x})\,\mathbf{w}
  \;\approx\; \frac{\operatorname{Im}\bigl[\,F(\mathbf{x}+\mathrm{i}\,h_{\text{cs}}\,\mathbf{w})\,\bigr]}{h_{\text{cs}}},
  \qquad h_{\text{cs}}\ll 1,
  \label{eq:cs}
\end{equation}
eliminating subtractive cancellation and thus,  achieving machine-precision accuracy. The resulting linear systems are solved matrix-free \cite{knoll2004jacobian}, using  the restarted Generalized Minimal Residual Method (GMRES) \cite{saad1985generalized}. In all the experiments we fixed $h_{\rm cs}=10^{-10}$, however, the method converges quadratically for any sufficiently small value of $h_{\rm cs}$.

The time step is adaptively controlled by monitoring the drift of the quartic invariant \cite{bckop2025},
\begin{equation}
	I(t) = \operatorname{Im} \sum_n  \Big [u_n\overline{u}_{n+2} \bigl(1 - |u_{n+1}|^2\bigr)
	- \tfrac{1}{2} \bigl(u_n \overline{u}_{n+1}\bigr)^2\Big ],
	\label{eq:inv}
\end{equation}
relative to its initial value $I_0 = I(0)$. We normalize this drift by a fixed reference scale,
\begin{equation}
	S_0 = \sum_n \left ( \Bigl| u_n^0 \overline{u}_{n+2}^0 \bigl(1 - |u_{n+1}^0|^2\bigr) \Bigr|
	+ \tfrac{1}{2} |u_n^0|^2 |u_{n+1}^0|^2 \right ),
	\label{eq:S0}
\end{equation}
which is constructed from the constituent terms of $I$ and evaluated once on the initial profile $u^0$. A tentative time step of size $\Delta t$ is accepted provided that
\begin{equation}
	\frac{|I(t + \Delta t) - I_0|}{S_0} < \varepsilon_{\mathrm{rel}},
	\label{eq:accept}
\end{equation}
where $\varepsilon_{\mathrm{rel}}$ denotes a dimensionless tolerance (typically chosen as $5 \times 10^{-13}$ unless specified otherwise). Failing this, $\Delta t$ is halved and the step is retried, down to a prescribed minimum threshold $\Delta t_{\min} = 5 \times 10^{-16}$. Following an accepted step, $\Delta t$ gradually relaxes back toward the baseline resolution $\Delta t_{\max} = T_{\mathrm{final}} / N_{\mathrm{step}}$. Because $S_0$ remains constant while the cumulative drift $|I(t) - I_0|$ grows monotonically near the singularity, this criterion dynamically detects blow-up, forcing $\Delta t \to 0$ as $t \to T_{\max}$. The normalization ensures that the tolerance is dimensionless and independent of the initial data. Finally, monitoring the drifts of the invariant, momentum, and Hamiltonian provides a robust, independent verification of computational accuracy.

All numerical experiments are conducted on a periodic lattice with $N = 401$ sites (indexed as $n \in \{-200, \dots, 200\}$). Time integration is initialized with a step size of $\Delta t = 10^{-4}$, which is dynamically adjusted according to the adaptive procedure outlined above. For visual clarity, the figures present close-up views of the solution profile centered around $n = 0$.
\begin{figure}[t]
	\begin{center}
		\begin{tabular}{cc}
		\includegraphics[width=0.9\textwidth]{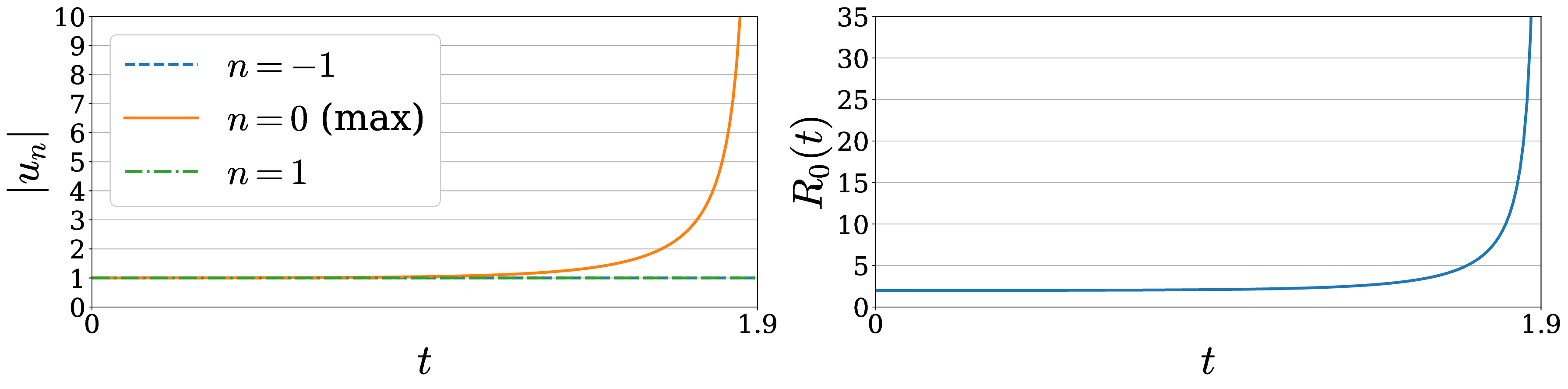}\\
         \includegraphics[width=0.9\textwidth]{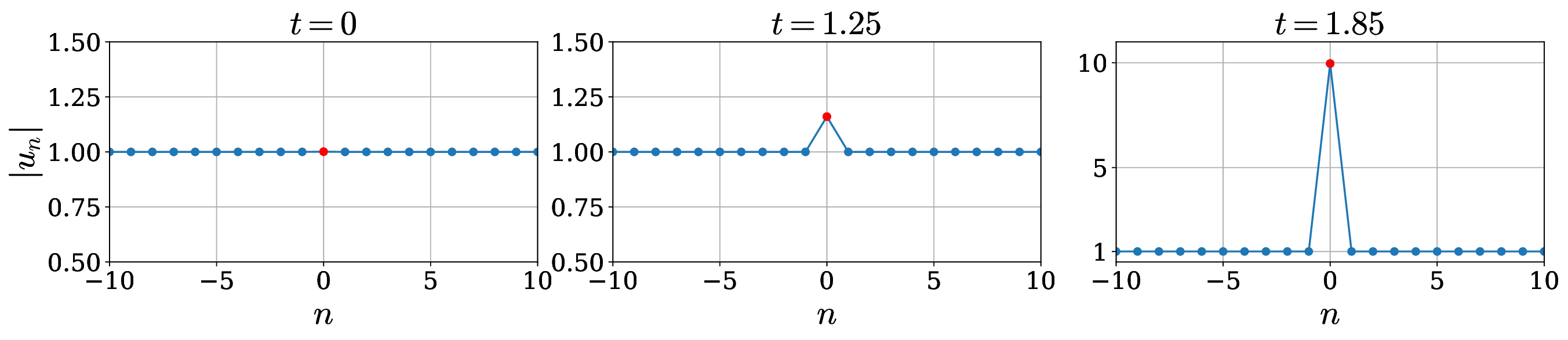}
		\end{tabular}
	\end{center}
	\caption{Example 1: \emph{Fundamental singular solution of Proposition \ref{FSS}}. Numerical dynamics of the fundamental singular solution from Proposition \ref{FSS}, generated by initial data \eqref{eq:initA1} with amplitude $A=1.001$ and phase $\theta=1$. The analytical blow-up time $T_{\mathrm{max}}^{J(0)<0} \approx 1.90035$ exhibits excellent agreement with the detected numerical blow-up time $T_{\mathrm{max}} \approx 1.900007$. Top left: Temporal evolution of $|u_n(t)|$ at the central node and its nearest neighbors ($n = 0, \pm 1$). Top right: Temporal evolution of $R_0(t)$. Bottom row: Spatial profile snapshots at $t = 0, 1.25, 1.85$ across lattice sites $n \in [-10, 10]$.}
	\label{figA1}
\end{figure}		 
\vskip 2mm
\noindent
\textbf{Example 1:} \emph{Fundamental singular solution of Proposition \ref{FSS}}.

In this first experiment, we examine the initial profile
\begin{equation}\label{eq:initA1}
	u_n(0) = \begin{cases}
		i A e^{i\theta}, & n = 0, \quad A > 1, \\
		e^{i\theta}, & n \neq 0,
	\end{cases}
\end{equation}
where all site amplitudes $|u_n(0)|$ lie on the unit circle except at the origin $n = 0$, where $|u_0(0)| = A > 1$. This serves as the prototype initial condition generating the fundamental singular solution of Proposition \ref{FSS}, \emph{representing a discrete $\delta$-function perturbation of the unit-modulus background.} Direct calculation for \eqref{eq:initA1} yields $R_0(0) = 2A > 0$, $M_0(0) = A^2$, $|\sigma|^2 = 4$, and $J(0) = |\sigma|^2(M_0(0)-1) - R_0^2(0) = -4 < 0$. Because $R_0(0) = 2A > \sqrt{|J(0)|}$, the analytical blow-up time given by \eqref{eq6i} simplifies to
\begin{equation}\label{eq:maxtimep}
	T_{\max}^{J(0)<0} = \frac{1}{4}\log\left(\frac{A+1}{A-1}\right).
\end{equation}
Comparing this value with the exact blow-up time \eqref{ANTmax} confirms that the fundamental singular solution falls within the class defined by Proposition \ref{bu-p}(iii), as $T_{\max}^{J(0)<0} = T_{\max}^{\mathrm{fs}}$.

For the numerical study, we choose $A = 1.001$ and $\theta = 1$, which corresponds to $T_{\max}^{J(0)<0} \approx 1.90035$, $R_0(0) = 2.002$, and $J(0) = -4$. With these parameters, the numerical integrator successfully tracks the solution up to $t \approx 1.900007$, extremely close to the exact singular time.

The Top left panel of Figure \ref{figA1} illustrates the temporal evolution of $|u_n(t)|$ at $n = -1, 0, 1$, whereas the Top right panel shows the evolution of $R_0(t)$. The bottom row of Figure \ref{figA1} presents spatial profile snapshots at $t = 0, 1.25, 1.85$ across $n \in [-10, 10]$. These results further demonstrate that the unit circle serves as a phase-space separatrix for the dynamics, consistent with Theorem \ref{sep-t} and Corollary \ref{inv-r}: the amplitudes remain strictly constant ($|u_n(t)| = 1$) for all off-center sites $n \neq 0$ up to $T_{\max}^{J(0)<0}$, while the central amplitude $|u_0(t)|$ grows toward finite-time blow-up at a rate consistent with the one of $\frac{1}{|\sigma|(T_{\max}-t)}$ obtained from the analytical formula \eqref{egfs}.

\begin{figure}[t]
	\begin{center}
		\begin{tabular}{cc}
		\includegraphics[width=0.9\textwidth]{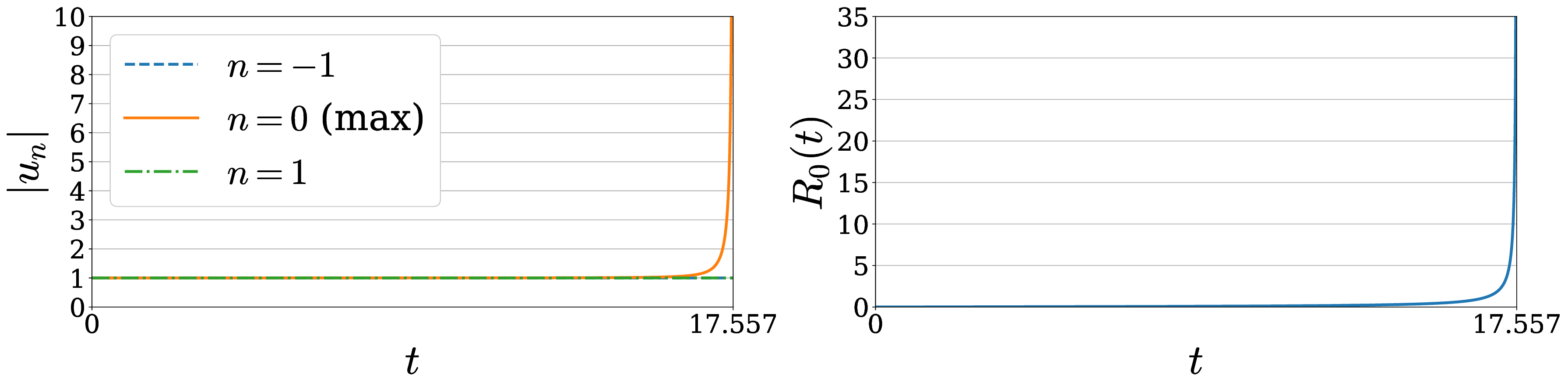}\\
        \hspace{-0.2cm}\includegraphics[width=0.9\textwidth]{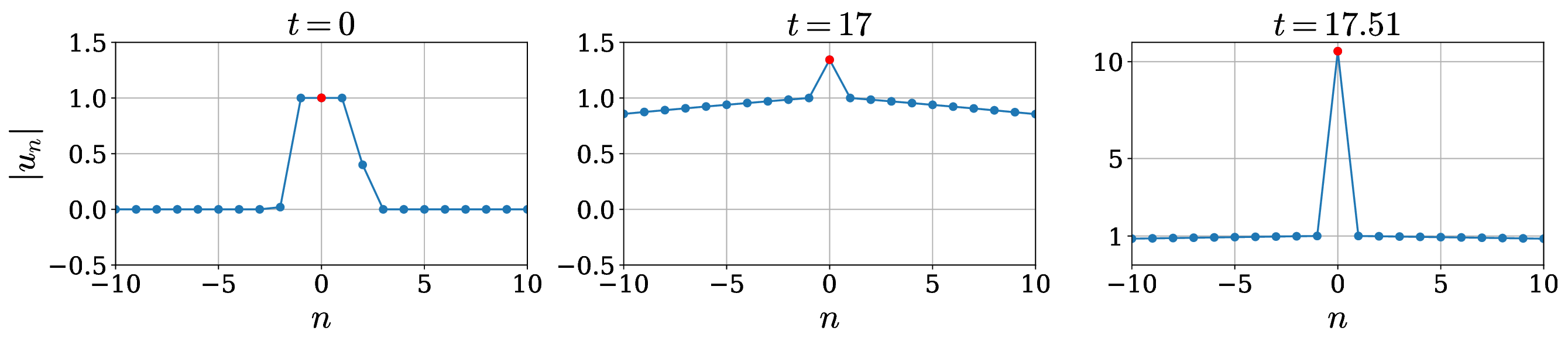}
		\end{tabular}
	\end{center}
	\caption{{Example 2: \emph{Singular solution of Proposition \ref{bu-p}(i) with $R_0(0)=0$ and $J(0)>0$}.} Numerical dynamics of the singular solution from Proposition \ref{bu-p}(i), generated by initial data~\eqref{eq:initA2} with uniform phase parameters $\theta_n = 1$ for all $n \in \mathbb Z$ and coefficients $A_0 = 1.001$, $A_1 = A_{-1} = 1$, $A_2 = 0.4i$, $A_{-2} = 0.02i$, $A_n = 0$, $n\notin \{0, \pm 1, \pm 2\}$. The analytical blow-up time estimate $T_{\mathrm{max}}^{J(0)>0} \approx 17.557$ shows excellent agreement with the detected numerical blow-up time $T_{\mathrm{max}} \approx 17.5575$. Top left: Temporal evolution of $|u_n(t)|$ at the central core and adjacent sites ($n = 0, \pm 1$). Top right: Temporal evolution of $R_0(t)$. Bottom row: Spatial profile snapshots at $t = 0, 17, 17.51$ across lattice sites $n \in [-10, 10]$.}
	\label{figA2}
\end{figure}		 
\vskip 2mm
\noindent
\textbf{Example 2:} \emph{Singular solution of Proposition \ref{bu-p}(i) with $R_0(0)=0$ and $J(0)>0$}.
 
 In this experiment, we investigate the evolution of the initial profile
\begin{equation}\label{eq:initA2}
	u_n(0) = 	A_n e^{i\theta_n}, \quad  n \in \mathbb Z,
\end{equation}
where $A_0 = 1.001$, $A_1 = 1$, $A_{-1} = 1$, $A_2 = 0.4i$, $A_{-2} = 0.02i$,  $A_n = 0$ for $n\notin \{0, \pm 1, \pm 2\}$, and $\theta_n = 1$ for all $n\in \mathbb Z$. Because $A_1 = A_{-1} = 1$ and $A_0 > 1$, this condition satisfies the criteria of Proposition \ref{bu-p}.
This example \emph{numerically models an asymmetric profile under vanishing boundary conditions}.

For these parameter values, $R_0(0) = 0$ and $J(0) \approx 0.008$. Consequently, the initial profile \eqref{eq:initA2} falls within the class characterized by Proposition \ref{bu-p}(i), yielding a predicted analytical blow-up time of $T_{\max}^{J(0)>0} \approx 17.557$. The numerical integrator successfully tracks the singularity up to $t \approx 17.55575$.

The presentation format of the plots in Figure \ref{figA2} (and in subsequent examples, unless otherwise specified) mirrors that of Example 1. The bottom row highlights the invariant partition of the lattice established in Theorem~\ref{sep-t} and Corollary \ref{inv-r}: amplitudes strictly inside the unit circle stay bounded within it for all $t < T_{\max}^{J(0)>0}$, the unit-modulus boundary sites at $n = \pm 1$ remain strictly stationary, and singular growth is driven entirely by the site outside the unit circle ($n = 0$). Interestingly, the interior site amplitudes are driven toward the unit circle, and shortly before the solution collapses ($t \approx 17.51$), \emph{the fundamental singular solution profile spontaneously emerges}. As demonstrated in subsequent examples, this recurrent behavior provides a clear physical justification for designating the solution in Proposition \ref{FSS} as ``fundamental''.

\vskip 2mm
\noindent
\begin{figure}[t]
	\begin{center}
		\begin{tabular}{cc}
		\includegraphics[width=0.9\textwidth]{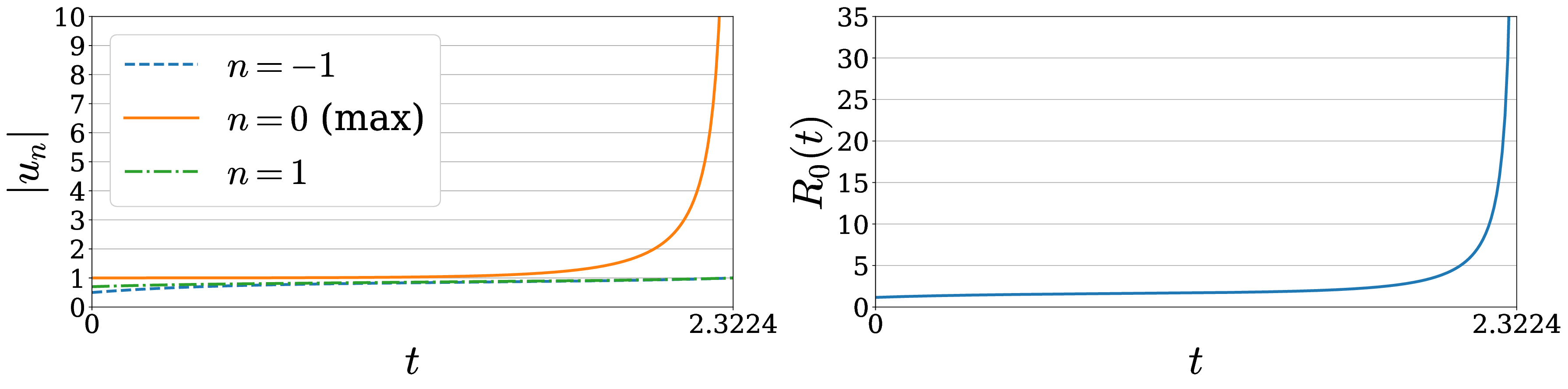}\\
		\includegraphics[width=0.9\textwidth]{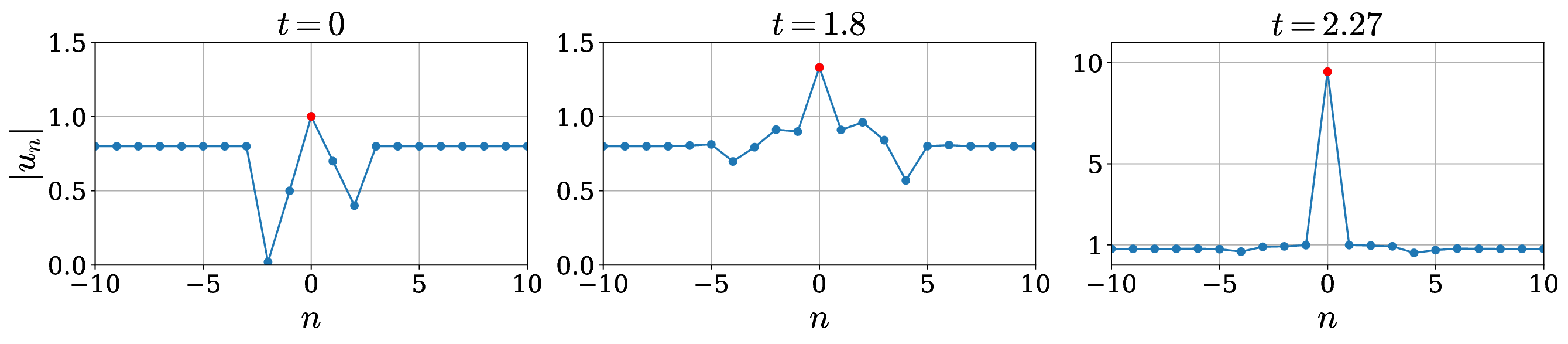}
		\end{tabular}
	\end{center}
	\caption{Example 3 Case A: \emph{General initial data on a nonzero background $0<q_0<1$}. Numerical simulation of the singular solution corresponding to initial data \eqref{eq:initA2} with phases $(\theta_0, \theta_1, \theta_{-1}, \theta_2, \theta_{-2}) = (\frac \pi 4, \frac \pi 6, \frac \pi 3, \pi, \frac \pi 5)$, $ \theta_n = 0$, $n\notin \{0, \pm 1, \pm 2\}$, and coefficients $A_0=1.001i$, $A_1=0.7$, $A_{-1}=0.5$, $A_2=0.4i$, $A_{-2}=0.02i$, $A_n=0.8$, $n\notin \{0, \pm 1, \pm 2\}$, yielding initial values $R_0(0) \approx 1.16$ and $J(0) \approx -1.34$. Top left: Temporal evolution of $|u_n(t)|$ at the central core and adjacent sites $n = 0, \pm 1$. Top right: Temporal evolution of $R_0(t)$. Bottom row: Spatial profile snapshots at $t = 0, 1.8, 2.27$ across lattice sites $n \in [-10, 10]$. Singularity formation occurs at $t = T_{\mathrm{max}} \approx 2.3224$, strictly bounded by the theoretical estimate $T_{\mathrm{max}}^{s,J(0)<0} \approx 3.2781$ associated with a subsolution $u^s \in \mathcal{S}$ sharing identical initial parameters $R_0(0)$ and $J(0)$. }
	\label{figB}
\end{figure}		 
\textbf{Example 3:} \emph{General asymmetric data on a nonzero background $q_0>0$}.

The following numerical experiments provide an initial illustration of Theorem \ref{mainthe1}. We consider general initial data $u(0)$ that fail to satisfy the invariance conditions required by Propositions \ref{FSS} and \ref{bu-p}. Nevertheless, the criteria on $J(0)$ and $R_0(0)$ established in Proposition \ref{bu-p} apply, allowing for the detection of a suitable singular subsolution $u^s \in \mathcal{S}$. This comparison guarantees finite-time blow-up for the solution $u(t)$, establishing that its true blow-up time is bounded above by that of the subsolution, that is, $T_{\mathrm{max}} < T^s_{\mathrm{max}}$. The examples distinguish between the cases $0<q_0<1$ and $q_0>1$.
\vskip 2mm
\noindent\emph{Case A: Nonzero background $0 < q_0 < 1$.}

Here, we examine an asymmetric initial profile of the form \eqref{eq:initA2} with the variety of phases $\theta_0 = \frac \pi 4$, $\theta_1 = \frac \pi 6$, $\theta_{-1} = \frac \pi 3$, $\theta_2 = \pi$, $\theta_{-2} = \frac \pi 5$, $\theta_n = 0$, $n\notin \{0, \pm 1, \pm 2\}$, and parameter values $A_0 = 1.001i$, $A_1 = 0.7$, $A_{-1} = 0.5$, $A_2 = 0.4i$, $A_{-2} = 0.02i$, and $A_n = 0.8$, $n\notin \{0, \pm 1, \pm 2\}$. {\em This parameter set models initial data supported on a constant background of $q_0 < 1$.} At $t=0$, the initial profile features two asymmetric dips below the background amplitude $q_0 = 0.8$, whereas the central site $n=0$ has an amplitude slightly exceeding the unit modulus. 

For this initial profile, $R_0(0) \approx 1.16$ and $J(0) \approx -1.34$, yielding $R_0(0) > \sqrt{|J(0)|}$. Although the unit-modulus condition $|u_{\pm 1}(0)| = 1$ is not fulfilled, the data meet the criteria of Proposition \ref{bu-p}(iii). We therefore anticipate finite-time singularity formation prior to the analytical blow-up time of the associated subsolution $u^s \in \mathcal{S}$, given by $T_{\mathrm{\max}}^{s,J(0)<0} \approx 3.2781$ as calculated from \eqref{eq6i}. Numerical integration confirms this upper bound, yielding an actual blow-up time of $T_{\mathrm{\max}} \approx 2.3224 \le T_{\mathrm{\max}}^{s,J(0)<0}$. We remark that the subsolution associated with the specific values of $R_0(0)$ and $J(0)$ above is not necessarily optimal with respect to bounding $T_{\mathrm{\max}}$. Indeed, another subsolution $u^s \in \mathcal{S}$ constructed with different parameters $R_0(0)$ and $J(0)$ may yield a sharper upper bound on the true blow-up time.

Figure \ref{figB} illustrates the blow-up dynamics of the above initial condition,  with its bottom row displaying snapshots of the spatial profiles at selected times.  While the amplitude at the central site $n=0$ grows, the initial dips emit counter-propagating waves, clearly visible in the snapshot at $t=1.8$,  that drive the off-center site amplitudes back toward the unit-modulus equilibrium state. As the system approaches the singularity, the solution gradually assumes the characteristic profile of the fundamental singular solution,  depicted in the snapshot at $t = 2.27$.
\begin{figure}[t]
	\begin{center}
		\begin{tabular}{cc}
		\includegraphics[width=0.9\textwidth]{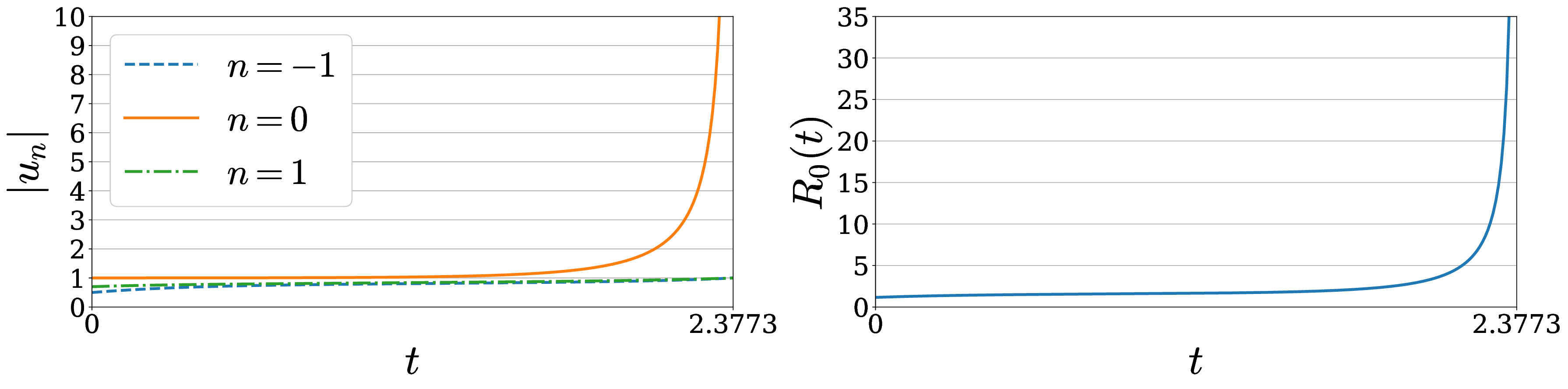}\\
	     \includegraphics[width=0.9\textwidth]{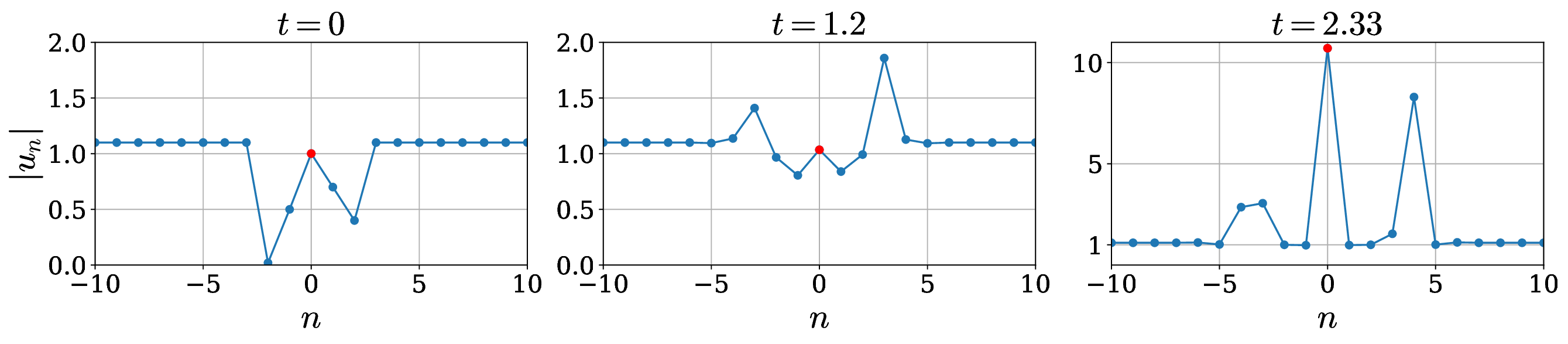}\\
          \includegraphics[width=0.9\textwidth]{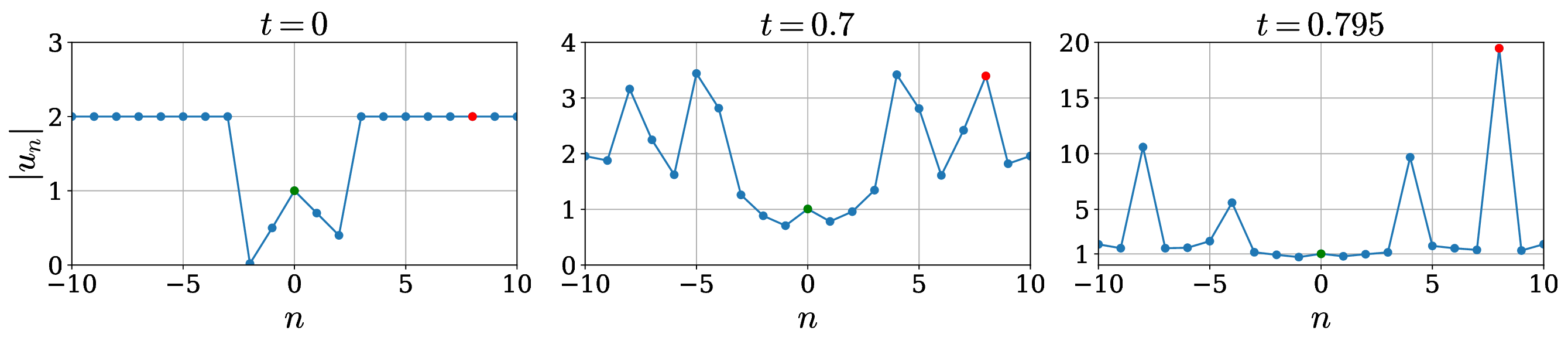}
		\end{tabular}
	\end{center}
	\caption{Example 3 Case B: \emph{General initial data on a nonzero background $q_0>1$}. Numerical simulation of the singular solution corresponding to initial data \eqref{eq:initA2} with phase parameters $(\theta_0, \theta_1, \theta_{-1}, \theta_2, \theta_{-2}) = (\frac \pi 4, \frac \pi 6, \frac \pi 3, \pi, \frac \pi 5)$, $\theta_n = 0$, $n\notin \{0, \pm 1, \pm 2\}$,  and coefficients $A_0=1.001i$, $A_1=0.7$, $A_{-1}=0.5$, $A_2=0.4i$, $A_{-2}=0.02i$, $A_n=1.1$, $n\notin \{0, \pm 1, \pm 2\}$. The initial values $R_0(0) \approx 1.16$ and $J(0) \approx -1.34$. Top left: Temporal evolution of $|u_n(t)|$ at the central core and adjacent sites $n = 0, \pm 1$. Top right: Temporal evolution of $R_0(t)$. Middle row: Spatial profile snapshots at $t = 0, 1.2, 2.33$ across lattice sites $n \in [-10, 10]$ for $A_n=1.1$, $n\notin \{0, \pm 1, \pm 2\}$. Singularity formation occurs at $t = T_{\mathrm{max}} \approx 2.37730$, strictly bounded by the theoretical estimate $T_{\mathrm{max}}^{s,J(0)<0} \approx 3.2781$ associated with a subsolution $u^s \in \mathcal{S}$ sharing identical values of $R_0(0)$ and $J(0)$. Bottom row: Evolution of the same initial profile for $A_n=2$, $n\notin \{0, \pm 1, \pm 2\}$, yielding an accelerated blow-up time $T_{\mathrm{max}} \approx 0.82$. An optimal fundamental singular subsolution can be constructed such that $T_{\mathrm{max}} \approx 0.82 < T^{\mathrm{fs}}_{\mathrm{max}} \approx 0.8467$.}
	\label{figBB}
\end{figure}		 
\vskip 2mm
\noindent
\emph{Case B: Nonzero background $q_0 > 1$.}

We next examine the initial condition \eqref{eq:initA2} with $A_n=1.1$, $n\notin \{0, \pm 1, \pm 2\}$,  keeping all remaining parameters identical to Case A. \emph{This choice models initial data supported on a background $q_0 > 1$.} At $t=0$, the profile features two asymmetric dips below the background level $q_0 = 1.1$, while the central site $n=0$ has an amplitude slightly exceeding the unit modulus, yet remaining strictly below $q_0$.

For this profile, $R_0(0) \approx 1.16027$ and $J(0) \approx -1.34$, as only $A_n$, $n\notin \{0, \pm 1, \pm 2\}$, have been modified while the local profile at $n = 0, \pm 1$ remains identical to Case A of $q_0=0.8$. Consequently, we predict a finite-time blow-up bounded by $T_{\mathrm{max}} < T_{\mathrm{max}}^{s,J(0)<0} \approx 3.2781$. Numerical integration yields $T_{\mathrm{max}} = 2.37730$, remarkably close to the singularity time observed in Case A. This close agreement is expected given that $q_0 = 1.1$ lies close to unity as the case A with $q_0=0.8$ and the local dynamics near $n=0$ initially coincide. Furthermore, {\em Theorem \ref{mainthe1} combined with the continuous dependence on the initial data established in Theorem \ref{lwp-t} carries an important physical implication for moderate background amplitudes $q_0 \gtrsim 1$}: since the singular subsolutions $\mathcal{S}$ demonstrate that central-site-driven blow-up is  independent of initial data away from the core $n = 0, \pm 1$, the central site is expected to remain the primary driver of singularity formation in this example as well. That is, for moderately small $q_0 \gtrsim 1$, local growth at $n=0$ outpaces any prospective instability arising from the moderate background $q_0>1$, ultimately leading to localized blow-up from the centered site. 

The predicted dynamics are confirmed by the numerical results presented in Figure \ref{figBB}. The snapshots of the spatial profiles in the middle row, clearly illustrate the invariant lattice partition established in Theorem \ref{sep-t} and Corollary \ref{inv-r}. Notably, although background sites experience growth as $t \to T_{\mathrm{max}}$, the profile of the fundamental singular solution spontaneously emerges by $t = 2.33$. Singularity formation remains strictly localized at the central site at $n=0$, whose amplitude growth outpaces the growth of the surrounding growing background nodes.

By increasing the background to $q_0=2$ (by choosing $A_n=2$, $n\notin \{0, \pm 1, \pm 2\}$), we expect the blow-up to be driven by sites located initially on the background. While the profile remains unchanged around the site $n=0$, and the values $R_0(0)>0$ and $J(0)<0$ remain the same for this considerably increased value of $q_0$, the continuous dependence on the initial data guarantees that the corresponding subsolution associated with these values of $R_0(0)$ and $J(0)$ will become even less optimal in bounding $T_{\mathrm{max}}$ from above. This is indeed the case, as we find numerically that $T_{\mathrm{max}}\approx 0.82$. In the bottom row of Figure \ref{figBB}, we show snapshots of the spatial profiles of the solution for $q_0=2$. Due to the elevated background amplitude, singularity formation occurs at a node initially situated on the background $q_0$, which in this example corresponds to site $n=8$. The central site $n=0$, having amplitude $|u_0(0)|=1.001$, does not manage to outpace the growth of the surrounding background sites. However, the blow-up mechanism remains the same: the snapshot at $t=0.79$ clearly shows that the growing sites tend to form the profile of the fundamental singular solution once again. 

This example underlines the following important fact: since the choice of the site playing the role of the central one is irrelevant in the construction of the fundamental singular solution (which, as discussed in Example 1, belongs to the class of singular solutions of Proposition \ref{bu-p}(iii)), one can systematically ``engineer'' a fundamental singular solution to serve as an optimal subsolution. For instance, selecting a central amplitude $|u_{\mathrm{center}}(0)| = 1.07$ yields an exact blow-up time of $T_{\mathrm{max}}^{\mathrm{fs}} \approx 0.8467$. By exploiting the arbitrary tuning of $A = |u_{\mathrm{center}}(0)| > 1$, this approach enables the construction of optimal subsolutions that approximate $T_{\mathrm{max}}$ from above with high precision. This strategy is naturally motivated by the initial profile itself: when $q_0 > 1$, infinitely many nodes in \eqref{eq:initA2} lie strictly outside the unit circle. 

In the subsequent example, we present a detailed theoretical construction of optimal singular solutions and subsolutions for a general solution $u$ within the framework of Theorem \ref{mainthe1}.
\begin{figure}[t]
	\begin{center}
		\begin{tabular}{cc}
		\includegraphics[width=0.9\textwidth]{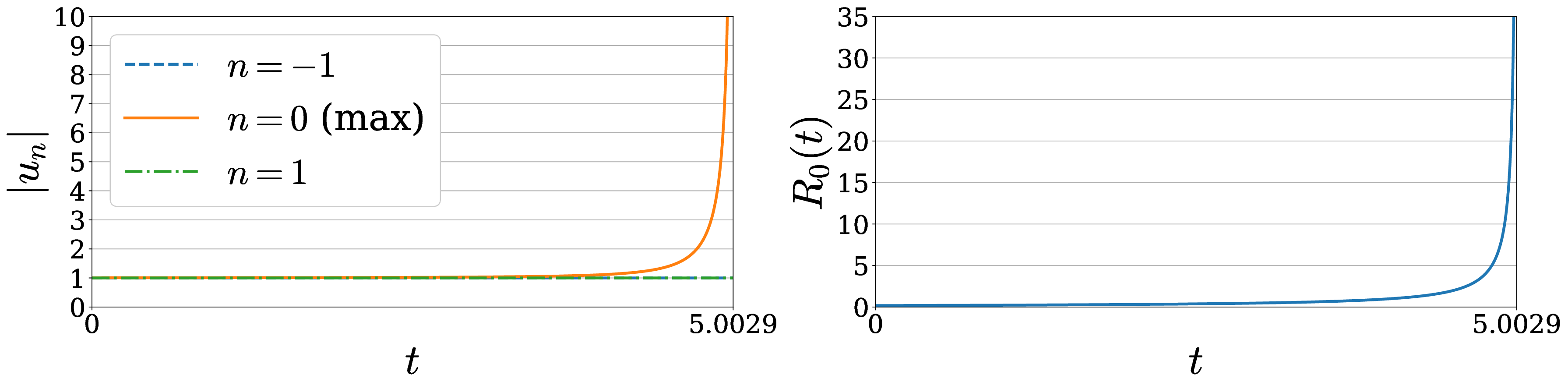}\\
		\includegraphics[width=0.9\textwidth]{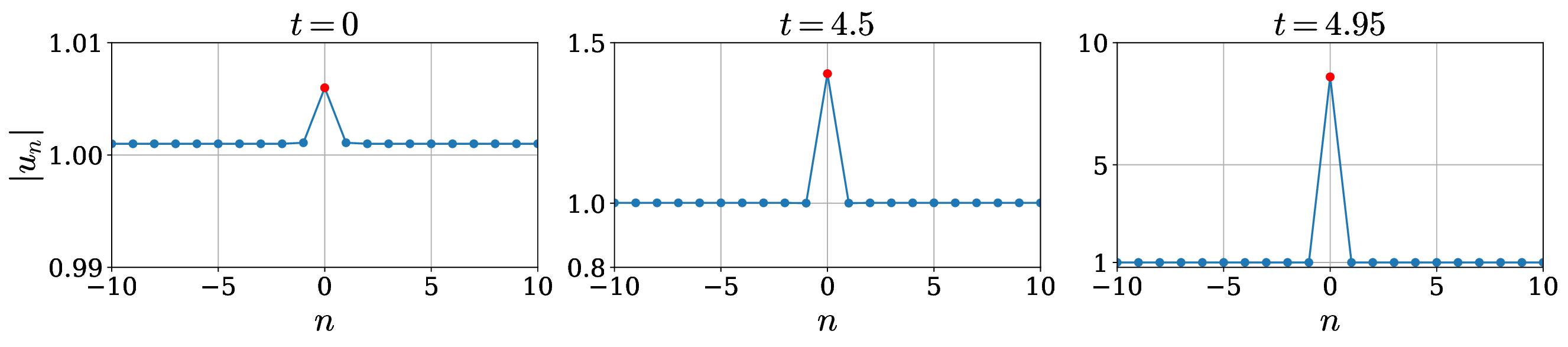}
		\end{tabular}
	\end{center}
\caption{Example 4: \emph{General initial data on a background $q_0>1$ with all lattice sites outside the unit circle}. 
Numerical simulation of the singular solution corresponding to initial data \eqref{eqabove} with parameters $q_0=1.001$, $A=0.1$, and $\theta=\frac \pi 2$, yielding initial values $R_0(0) \approx 0.1731$ and $J(0) \approx 0.0181$. Top left: Temporal evolution of $|u_n(t)|$ at the central core and adjacent nodes ($n = 0, \pm 1$). Top right: Temporal evolution of $R_0(t)$. Bottom row: Spatial profile snapshots at $t = 0, 4.5, 4.95$ across lattice sites $n \in [-10, 10]$. Singularity formation occurs at $t = T_{\mathrm{max}} \approx 5.0029$. The theoretical estimate $T_{\mathrm{max}}^{s,J(0)>0} \approx 4.9092$, associated with a solution $u^s \in \mathcal{S}$ sharing identical initial parameters $R_0(0)$ and $J(0)$, provides a tight lower bound for $T_{\mathrm{max}}$. Optimal subsolutions $\tilde{u}^s \in \mathcal{S}$ can be constructed to approximate $T_{\mathrm{max}}$ with high accuracy; see main text for details.}
\label{Fig4All}
\end{figure}		 
\vskip 2mm
\noindent
\textbf{Example 4:} \emph{General initial data on a background $q_0>1$ with all lattice sites outside the unit circle. Theoretical construction of optimal singular solutions and subsolutions $u^s\in\mathcal{S}$ within the framework of Theorem \ref{mainthe1}.}

Continuing from Example 3 regarding initial data on a nonzero background $q_0$, we consider here the initial condition
\begin{equation}\label{eqabove}
	u_n(0)=q_0+ Ae^{-2n^2+i\theta},
    \quad
    |u_n(0)|^2 = (q_0 + A e^{-2n^2} \cos\theta)^2 + A^2 e^{-4n^2} \sin^2\theta
    =
    q_0^2 + A^2 e^{-4n^2} + 2q_0 A e^{-2n^2} \cos\theta
\end{equation}
where $q_0>1$, $A>0$, and $\theta\in [-\frac \pi 2, \frac \pi 2]$. \emph{The initial condition \eqref{eqabove} models the regime where all lattice sites initially lie strictly outside the unit circle}, illustrated here by a Gaussian perturbation on the top of the background $q_0 > 1$.

Figure \ref{Fig4All} illustrates the blow-up dynamics for $q_0=1.001$, slightly above unit modulus, $A=0.1$, and phase $\theta=\frac \pi 2$. For these parameters, we compute $R_0(0)\approx 0.1731$ and $J(0)=0.0181$, placing the initial data within the class governed by Proposition \ref{bu-p}(i). Via formula \eqref{eq10}, the corresponding solution $u^s\in\mathcal{S}$ associated with these values of $R_0(0)$ and $J(0)$ yields the theoretical  estimate $T_{\mathrm{max}}^{s, J(0)>0}\approx 4.9092$. Numerically, the  blow-up time is determined to be $T_{\mathrm{max}}\approx 5.0029$. We observe that, while $T_{\mathrm{max}}^{s, J(0)>0}$ provides a sharp analytical estimate for $T_{\mathrm{max}}$, it only provides a \textit{lower bound} according to the top right panel of Figure \ref{seq}, as explained next.
\\[2mm]
\emph{A. Theoretical construction of optimal singular solutions and subsolutions in class $\mathcal{S}$ of Theorem \ref{mainthe1}.}
The example above belongs to the dynamical scenario depicted in the top right panel of Figure \ref{seq}. Although the exact solution $u$ may ``escape'' the bound provided by one singular solution $u^s\in \mathcal{S}$, it remains trapped by another singular solution  $\tilde{u}^s\in\mathcal{S}$ that acts as a subsolution. In this context, the initial profile \eqref{eqabove} with $q_0>1$ suggests that a fundamental singular solution from Proposition \ref{FSS} (generated by $\delta$-type initial data of the form~\eqref{eq:initA1}) can be tailored to serve as an optimal subsolution $\tilde{u}^s\in\mathcal{S}$ whose theoretical blow-up time $T_{\mathrm{max}}^{\mathrm{fs}}$ yields an exact upper bound. To demonstrate this, for any given blow-up time $T_{\mathrm{max}}>0$, we set 
\begin{equation}
	\label{eng1}
	T_{\mathrm{max}}=T_{\mathrm{max}}^{\mathrm{fs}}= \frac{1}{4}\log\left(\frac{A+1}{A-1}\right)
\end{equation}
and solve for the required amplitude $A>1$. The explicit solution of \eqref{eng1} is given by
\begin{equation}
	\label{eng2}
	A_{T_{\mathrm{max}}}^s = \frac{e^{4T_{\mathrm{max}}} + 1}{e^{4T_{\mathrm{max}}} - 1} = \frac{e^{2T_{\mathrm{max}}} + e^{-2T_{\mathrm{max}}}}{e^{2T_{\mathrm{max}}} - e^{-2T_{\mathrm{max}}}} = \coth(2T_{\mathrm{max}}).
\end{equation}
Consequently, given the blow-up time $T_{\mathrm{max}}$ of a general solution $u$, the initial data \eqref{eq:initA1} with amplitude $A_{T_{\mathrm{max}}}^s$ define an optimal fundamental singular solution blowing up precisely at $T_{\mathrm{max}}$. Furthermore, any choice of $A^s < A_{T_{\mathrm{max}}}^s$ generates a singular solution $\tilde{u}^s\in\mathcal{S}$ via Proposition \ref{FSS} that acts as a rigorous subsolution for $u$.  This is due to the fact that the right-hand side $\coth(2T)$ of \eqref{eng2} decreases monotonically from $+\infty$ downward to $1$ as $T \to \infty$. Applying this to our numerical example ($T_{\mathrm{max}}\approx 5.0029$), the formula \eqref{eng2} yields $A^s_{T_{\mathrm{max}}}=\coth(10.0058) \approx 1 + 4.074765 \times 10^{-9}$.

We emphasize that the construction of optimal subsolutions can be achieved via any family of singular solutions and, in particular, via the combined class $\mathcal S$ of Propositions \ref{FSS} and \ref{bu-p}.  For example, if one prefers to work directly with the parameter pair $(R_0(0), J(0))$ of the initial condition $u(0)$, an optimal subsolution $\tilde{u}^s\in\mathcal{S}$ can be constructed analogously by matching the blow-up time  $T_{\mathrm{max}}$ to the analytical formulas provided by Proposition~\ref{bu-p}, though the algebra can be slightly more involved. Specifically, when $R_0(0), J(0) > 0$, we fix $J(0)$ and solve the blow-up time relation \eqref{eq10},
\begin{equation}
	\label{eng3}
	T_{\mathrm{max}}= \frac{1}{\sqrt{J(0)}} \left[ \frac{\pi}{2} - \arctan\left(\frac{R_0(0)}{\sqrt{J(0)}}\right) \right] = T_{\mathrm{max}}^{J(0)>0},
\end{equation}
for the unknown parameter $R_0(0) \in \mathbb{R}$. Utilizing the trigonometric identity $\tan\left(\frac{\pi}{2} - \phi\right) = \cot(\phi)$, the solution to~\eqref{eng3} is
\begin{equation}
	\label{eng4}
	R_0^s(0)_{T_{\mathrm{max}}} = \sqrt{J(0)} \cot\left(T_{\mathrm{max}} \sqrt{J(0)}\right), \quad \text{valid for } T_{\mathrm{max}} \sqrt{J(0)} \in (0, \pi).
\end{equation}
The restriction $T_{\mathrm{max}} \sqrt{J(0)} \in (0, \pi)$ arises naturally because $T_{\mathrm{max}}\sqrt{J(0)} = \frac{\pi}{2} - \arctan\left(R_0(0)/\sqrt{J(0)}\right)$, and the range of $\frac{\pi}{2} - \arctan(x)$ is strictly $(0, \pi)$ for all $x \in \mathbb{R}$.

For initial data \eqref{eqabove} with $T_{\mathrm{max}}\approx 5.0029$ and $J(0)=0.0181$, substituting these values into \eqref{eng4} gives $R_0^s(0)_{T_{\mathrm{max}}} \approx 0.168747$. To verify the self-consistency of this construction, we observe that $R_0^s(0)_{T_{\mathrm{max}}} < R_0(0) = 0.1731$, with the two values being exceptionally close. This proximity directly accounts for why the initial theoretical estimate $T_{\mathrm{max}}^{s, J(0)>0}\approx 4.9092$ lies so close to $T_{\mathrm{max}}\approx 5.0029$: the former corresponds to a singular solution $u^s\in\mathcal{S}$ with $R_0(0)>0$ and $J(0)>0$ that fails to trap $u$, whereas replacing $R_0(0)$ with $R_0^s(0)_{T_{\mathrm{max}}}$ yields the optimal singular solution $\tilde{u}^s\in\mathcal{S}$ that blows up at the exact blow-up time $T_{\mathrm{max}}$. Finally, since $R_0(t)$ is strictly monotonic increasing (per \eqref{r0-eq}), any singular solution $u^s\in \mathcal{S}$ from Proposition \ref{bu-p}(i) satisfying $0<R_0^s(0) < R_0^s(0)_{T_{\mathrm{max}}}$ serves as a rigorous subsolution for $u$.
\\[2mm]
\emph{B. The role of the phase of general initial data in blow-up dynamics}.
Returning to initial condition \eqref{eqabove}, fixing $q_0 = 1.001$ and $A = 0.1$ while shifting the phase parameter to $\theta = \frac \pi 4$ yields initial values $R_0(0) \approx 0.1224$ and $J(0) \approx 0.6124$. Although $R_0(0)$ is smaller compared to the previous configuration ($\theta = \frac \pi 2$), the quantity $J(0)$ increases substantially. Consequently, via formula \eqref{eq10}, the subsolution $u^s \in \mathcal{S}$ associated with these updated initial values yields a theoretical estimate of $T_{\mathrm{max}}^{s, J(0)>0} \approx 1.8090$. This indicates that the singular solution for $\theta = \frac \pi 4$ should blow up significantly earlier than in the case of $\theta = \frac \pi 2$. Indeed, the numerical simulation detects the blow-up time at $T_{\mathrm{max}} \approx 1.8753$, again confirming that $T_{\mathrm{max}}^{s, J(0)>0}$ provides a sharp lower bound for $T_{\mathrm{max}}$. As discussed in paragraph A, optimal subsolutions $\tilde{u}^s \in \mathcal{S}$ can be similarly constructed for this regime.

Therefore, holding the amplitudes of a general initial condition $u(0)$ fixed, varying its phase can exert a crucial influence on the blow-up time of the corresponding singular solution. We recall that the role of the initial phase was also highlighted in the context of the fundamental singular solutions of Proposition \ref{FSS} (see Remark \ref{phaseFSS}).

\begin{figure}[t]
	\begin{center}
		\includegraphics[width=0.9\textwidth]{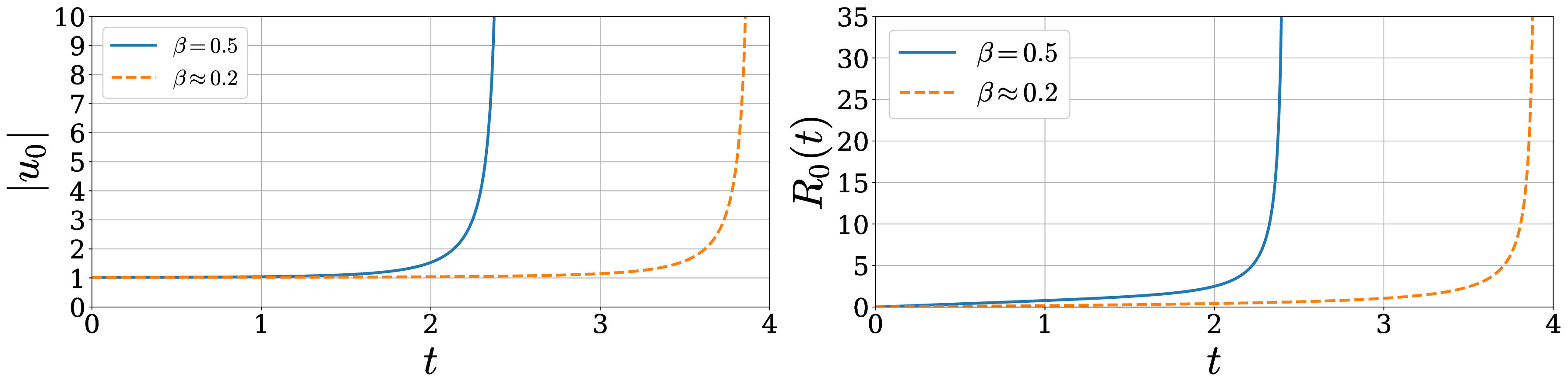}
		\includegraphics[width=0.9\textwidth]{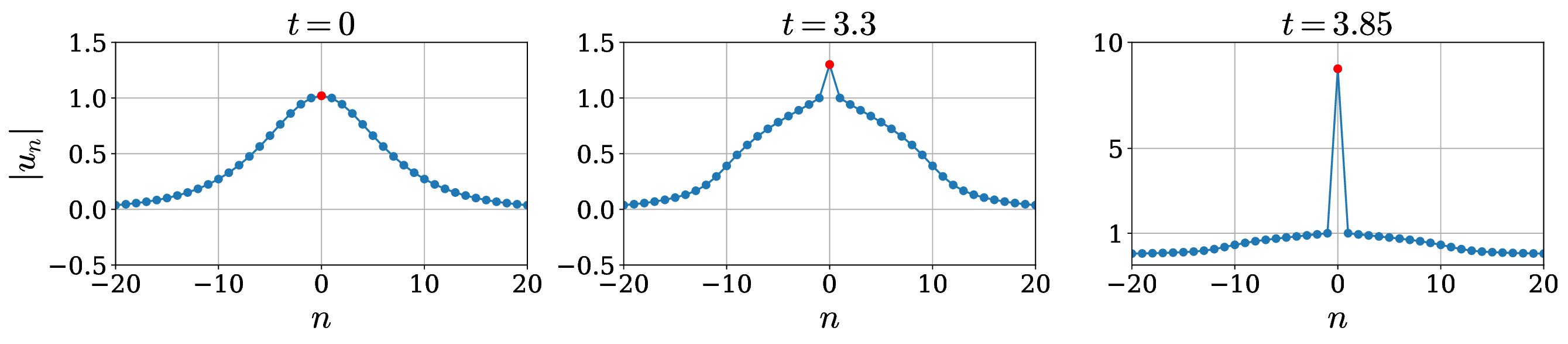}
		\includegraphics[width=0.9\textwidth]{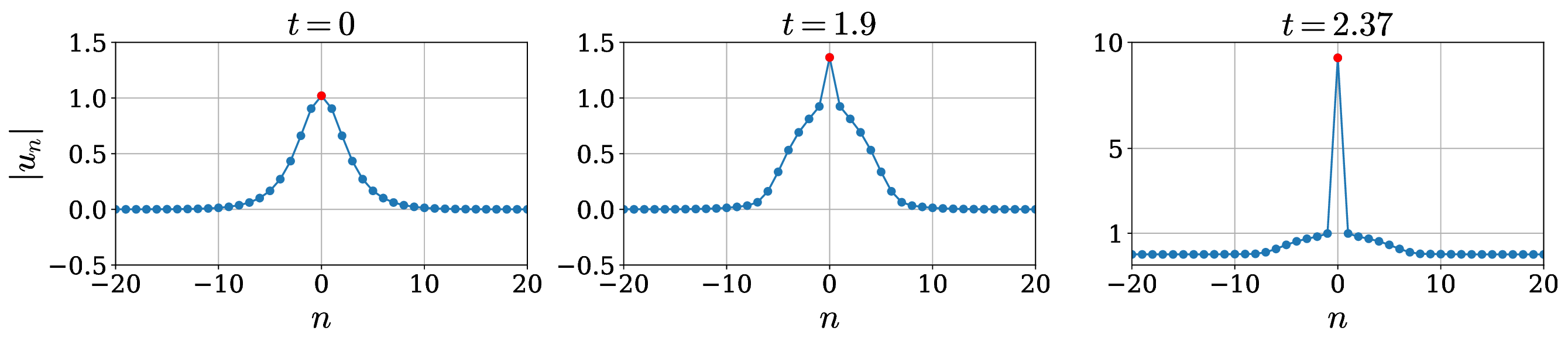}
	\end{center}
	\caption{Example 5 Case A: \emph{Exponentially decaying initial data \eqref{eq:example1id}}. Top row: Dashed orange curves trace the temporal evolution of $|u_0^s(t)|$ and $R_0^s(t)$ for initial data with $|u_0^s(0)|=|\alpha_1|=1.019$ and $\beta=\beta_1 \approx 0.1946$, exhibiting an analytical blow-up time of $T_{\mathrm{max}}^{s,J(0)}\approx 4.010005$. Solid blue curves trace $|u_0(t)|$ and $R_0(t)$ for initial data with $|u_0(0)|=|\alpha_2|=1.02$ and $\beta=\beta_2=0.5$, which undergoes numerical blow-up at $T_{\mathrm{max}}\approx 2.42287$. Here, $u^s$ serves as a subsolution for $u$. Middle row: Spatial profile snapshots of the subsolution $u^s$. Bottom row: Corresponding spatial profile snapshots of the solution $u$.}
	\label{figExpdata}
\end{figure}
\begin{figure}[t]
	\begin{center}
		\includegraphics[width=0.9\textwidth]{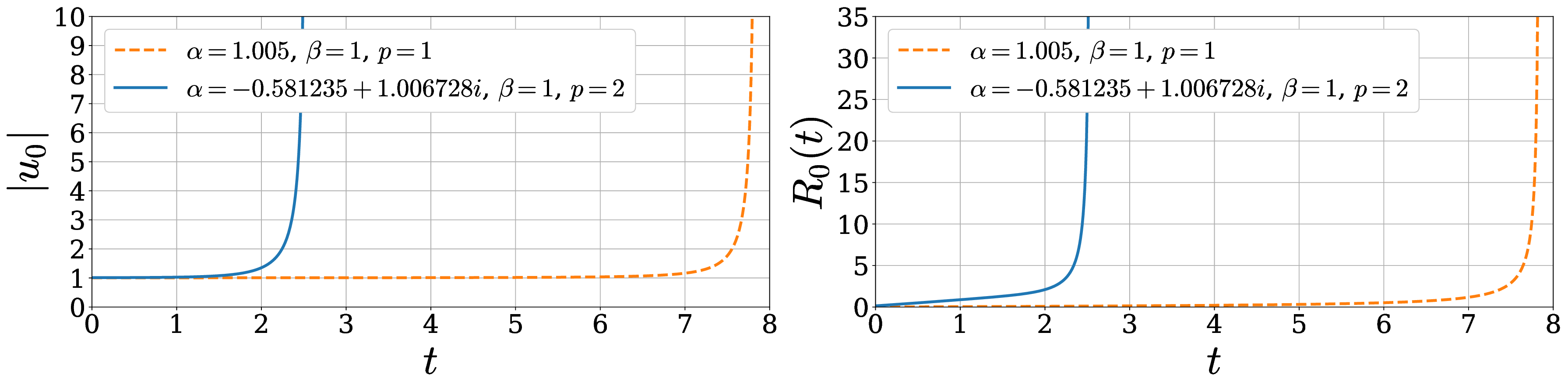}
		\includegraphics[width=0.9\textwidth]{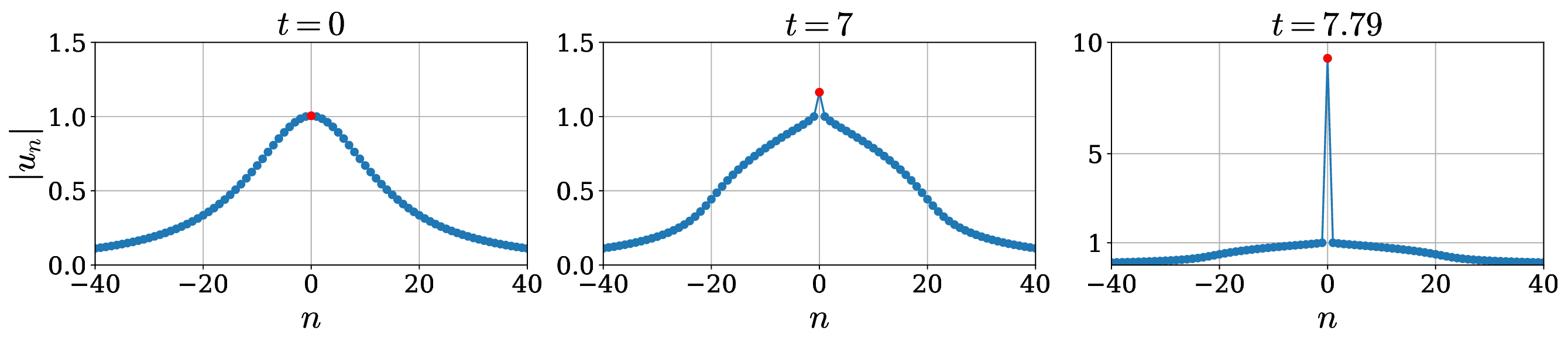}
		\includegraphics[width=0.9\textwidth]{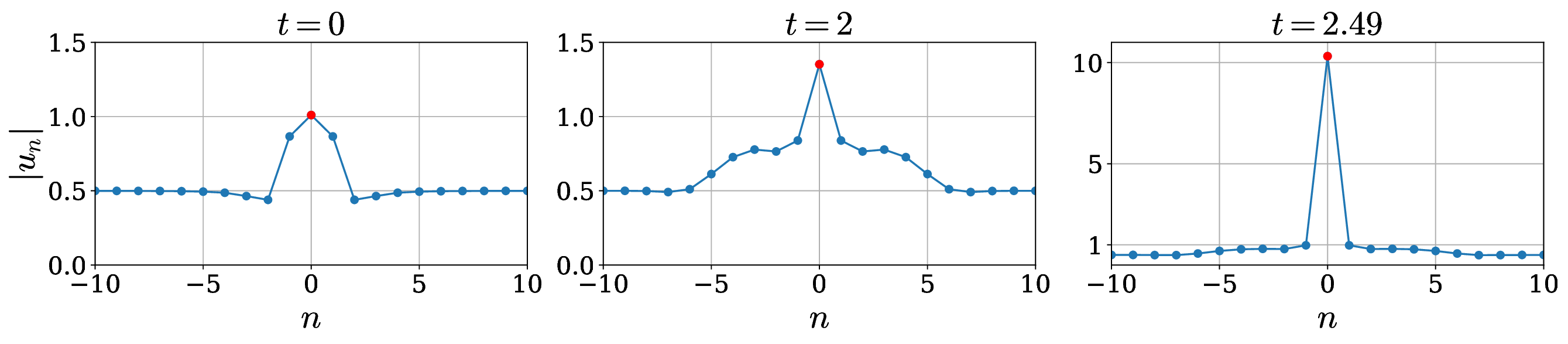}
	\end{center}
	\caption{Example 5 Case B: \emph{Algebraically decaying initial data \eqref{eq:example3id}}. Top row: Dashed orange curves trace the temporal evolution of $|u_0^s(t)|$ and $R_0^s(t)$ for initial data with $\alpha=1.005$, $\beta=1$, $p=1$, exhibiting a blow-up time of $T_{\mathrm{max}}\approx 7.8441$. Solid blue curves trace $|u_0(t)|$ and $R_0(t)$ for initial data with nonzero background as in \eqref{eq:nzbgrat} with $\alpha=-0.581235+1.006728i$, $\beta=1$ and $p=2$, which undergoes numerical blow-up at $T_{\mathrm{max}}\approx 2.5390$. Middle row: Spatial profile snapshots of the solution $u^s$ with initial condition \eqref{eq:example3id}  on a zero background. Bottom row: Corresponding spatial profile snapshots of the solution with the nonzero background.}
	\label{figExpdata2}
\end{figure}

\vskip 2mm
\noindent
\textbf{Example 5:} \emph{Localized initial data with exponential and algebraic decay.}

In this example, we investigate the dynamics of solutions generated by localized initial data exhibiting either exponential decay, such as the $\mathrm{sech}$-type profile \eqref{eq:example1id}, or algebraic decay \eqref{eq:example3id}, as constructed in Section \ref{4-s}.
\\[2mm]
\emph{Case A: Initial data with exponential decay.}

Figure \ref{figExpdata} summarizes the numerical simulations for two distinct choices of initial data \eqref{eq:example1id}. The first case considers parameter choices $\alpha$ and $\beta$ configured such that the corresponding solution $u^s \in \mathcal{S}$ possesses an analytical blow-up time given by \eqref{eq:timesec}. Setting $\varepsilon_1=0.019$ yields an initial amplitude $|u^s(0)|=|\alpha_1|=1.019$, with the parameter $\beta$ uniquely determined by $\beta=\beta_1=\log(1+\varepsilon_1+\sqrt{\varepsilon_1(2+\varepsilon_1)})\approx  0.1946$. Formula \eqref{eq:timesec} then predicts an analytical blow-up time of $T_{\mathrm{max}}^{s,J(0)>0}\approx 4.010005$. The second case selects $\varepsilon_2=0.02$, leading to an amplitude $|u(0)|=|\alpha_2|=1.02$ and $\beta=\beta_2=0.5$, which generates a solution $u \notin \mathcal{S}$ corresponding to general $\mathrm{sech}$-type initial data. Here, $|u(0)|>|u^s(0)|$, satisfying the conditions of Theorem \ref{mainthe1}. Under the framework of Proposition \ref{bu-p}(i), this initial condition exhibits metrics (cf. \eqref{j0sech}) $R_0(0)=0$ and $J(0)\approx 0.1322$. According to Theorem \ref{mainthe1}, $u^s$ acts as a subsolution for $u$, implying that the blow-up time of $u$ must satisfy $T_{\mathrm{max}} < T_{\mathrm{max}}^{s,J(0)>0} \approx 4.010005$.

The top row of Figure \ref{figExpdata} illustrates this dynamical behavior. Dashed (orange) curves track the temporal profiles $|u_0^s(t)|$ and $R_0^s(t)$, which lie strictly below the corresponding curves for $u$. The solution $u$ undergoes finite-time blow-up at $T_{\mathrm{max}}=2.42287 < T_{\mathrm{max}}^{s,J(0)} \approx 4.010005$ (see Figure \ref{seq}). 
Spatial snapshots of the subsolution $u^s$ and the solution $u$ are displayed in the middle and bottom rows, respectively. Both cases clearly demonstrate the spontaneous formation of the characteristic fundamental singular profile.
\\[2mm]
\emph{Case B: Initial data with algebraic decay.}

We examine two choices of initial data with algebraic decay. The first case considers the rational initial condition~\eqref{eq:example3id} with  parameter choices $\alpha=1.005=|u_0^s(0)|$, $\beta=1$, $\gamma =|\alpha|-\beta = 0.005$ and $p=1$, which satisfy the necessary conditions such that the corresponding solution $u^s \in \mathcal{S}$ possesses an analytical blow-up time given by \eqref{eq:maxtrat}. For this particular initial condition, $R_0^s(0)=0$ and $J^s(0)=0.0401$, while the solution's maximal time of existence given by \eqref{eq:maxtrat} is $T_{\max}^{s,J(0)>0}\approx 7.84418$. Using our numerical method, we computed the solution up to $T_{\max}\approx 7.8441$. 

The second case considers the rational initial condition \eqref{eq:example3id} with $p=2$ on a nonzero background, given specifically by
\begin{equation}\label{eq:nzbgrat}
u_n(0)=0.5+\frac{\alpha}{\beta+(|\alpha|-\beta)n^4},
\end{equation}
with $\alpha=-0.581235 + 1.006728i$ and $\beta=1$. For these parameters, we have $|u_0(0)|=1.01 > |u^s_0(0)|=1.005$, as well as $R_0(0)\approx 0.1407 > 0$ and $J(0)\approx 0.0405 > 0$. Thus, the corresponding solution is expected to blow up according to the criteria for $R_0(0)>0$ and $J(0)>0$. The numerical blow-up time is $T_{\max}\approx 2.5390 < T_{\max}^{s,J(0)>0}$; hence, the singular solution $u^s \in \mathcal{S}$ of the first case acts as a singular subsolution for $u(t)$. Figure \ref{figExpdata2} summarizes the numerical simulations illustrating these dynamics (in a manner analogous to Figure \ref{figExpdata}).

\begin{figure}[t]
	\begin{center}
		\begin{tabular}{cc}
		\includegraphics[width=0.9\textwidth]{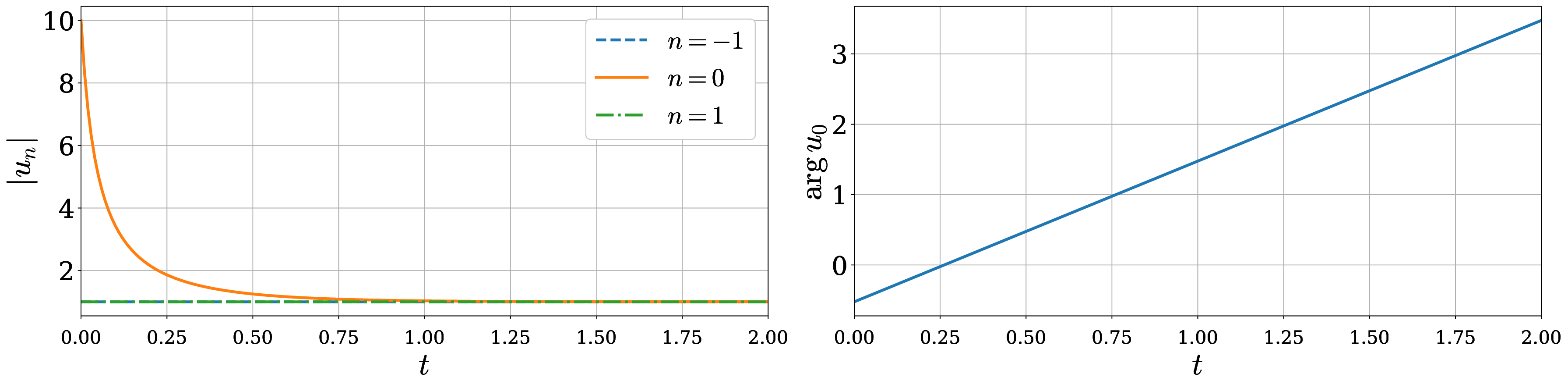}
		\end{tabular}
	\end{center}
	\caption{Example 6: \emph{Equilibration through scenario (i) of Proposition~\ref{equib}.} Initial data \eqref{eq:xmplege} with $A=10$, $B=1$, and $\theta=\frac \pi 3$. Left panel: evolution of $|v_n(t)|$ for the sites $n=0, \pm 1$. Right panel: exact linear evolution of $\arg u_0(t) = 2t - \frac \pi 6$.}
	\label{figB7bb}
\end{figure}

\begin{figure}[t]
	\begin{center}
		\begin{tabular}{cc}
		\includegraphics[width=0.9\textwidth]{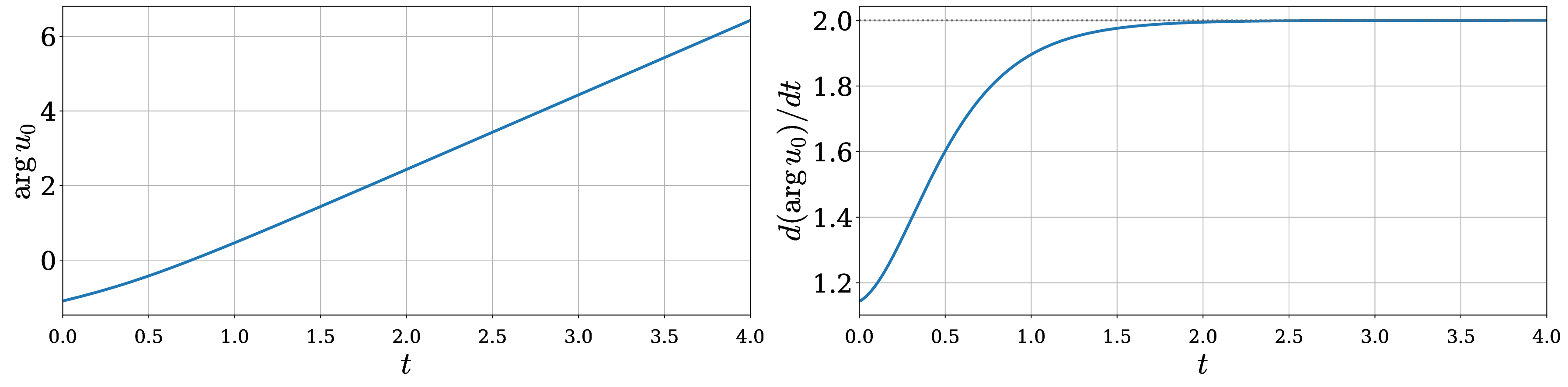}
		\end{tabular}
	\end{center}
	\caption{Example 6: \emph{Equilibration through scenario (ii) of Proposition~\ref{equib}.} Initial data \eqref{eq:initA2} with $A_0=10$, $A_n=1$ for $n\neq 0$, $\theta_0=\frac{\pi}{3}-\frac{1}{20}$, $\theta_1=\frac{\pi}{3}$, $\theta_{-1}=0$, and $\theta_n=\frac{\pi}{6}$ for $n\notin\{0,\pm 1\}$. Left panel: evolution of $\arg u_0(t)$. Right panel: evolution of $\frac{d}{dt}\arg u_0(t)$.}
	\label{figB7bbb}
\end{figure}
\vskip 2mm
\noindent
\textbf{Example 6.}  \emph{Global existence above unit modulus and equilibration.}

In this example, we present three experiments illustrating the global existence results of Proposition \ref{ge-p} and the phenomenon of equilibration.
\\[2mm]
\emph{I. Global existence dynamics.} The first two experiments consider the initial condition
\begin{equation}\label{eq:xmplege}
u_n(0)= \left\{\begin{array}{lll}
-iAe^{i\theta}, &A>1, &n=0,\\
e^{i\theta}, & & n=\pm 1,\\
Be^{i\theta}, & & \text{otherwise},
\end{array}\right.
\end{equation}
where $R_0(0)=-2A$, $M_0(0)=A^2$, $|\sigma|^2=4$ and $J(0)=-4<0$.  Hence the assumptions of Proposition \ref{ge-p}(i) are satisfied and we expect  $\lim_{t\to\infty}|u_0(t)|=1$ and $\lim_{t\to\infty}R_0(t)=-\sqrt{|J(0)|}=-2$. 

In the first experiment we set $B=0$ and in the second $B=1$, while in both experiments we set $A=10$ and $\theta=\frac \pi 3$. The top row of Figure \ref{figB7A} displays the evolution of $|u_n(t)|$ for $n=0,\pm 1$ and of $R_0(t)$ for the case $B=0$,  demonstrating the convergence to the unit circle for $|u_0(t)|$, the fact that $|u_{\pm 1}(t)|$ remain stationary, and the predicted asymptotic behavior for $R_0(t)$. Snapshots of the evolution of the initial condition are depicted in the bottom row of  Figure \ref{figB7A}. 
\begin{figure}[t]
	\begin{center}
		\begin{tabular}{cc}
		\includegraphics[width=0.9\textwidth]{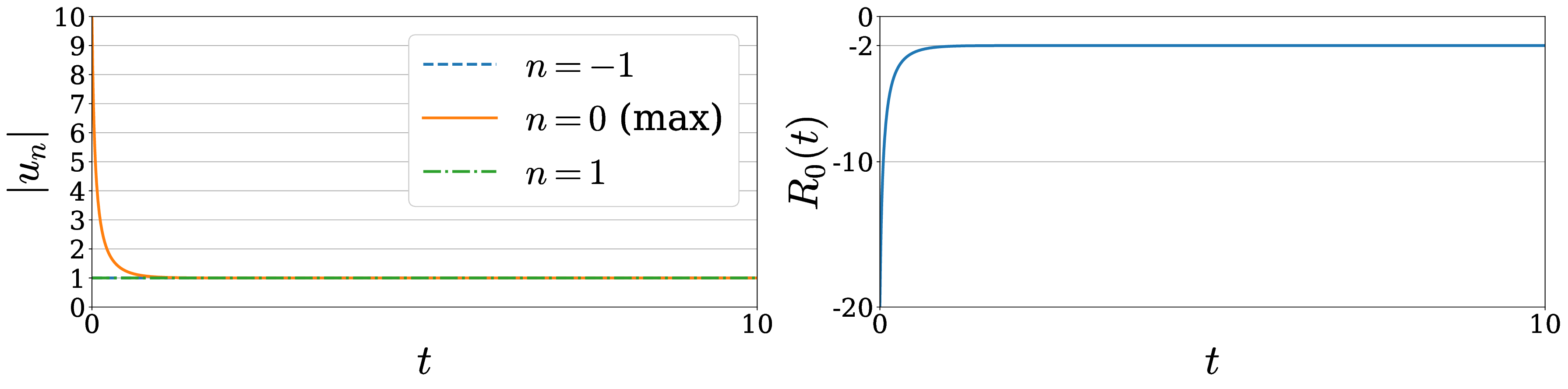}\\
       \includegraphics[width=0.9\textwidth]{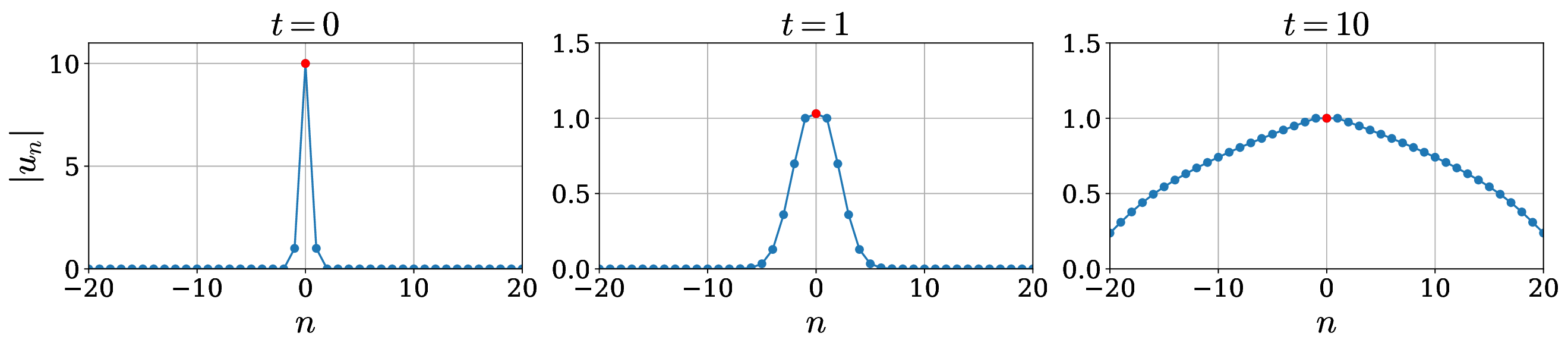}
		\end{tabular}
	\end{center}
	\caption{Example 6: \emph{Dynamics of the initial condition \eqref{eq:xmplege}}. Parameters $A=10$, $B=0$, and $\theta=\frac \pi 3$. Top row: the left panel depicts the evolution of $|u_n(t)|$ for the sites $n=0, \pm 1$, while the right panel shows the evolution of $R_0(t)$. Bottom row: snapshots of the evolution of the initial profile.}
	\label{figB7A}
\end{figure}		 
The evolution of $|u_n(t)|$ for $n=0, \pm 1$ and of $R_0(t)$ is identical in the case $B=1$. What differentiates the corresponding solutions is illustrated in Figure~\ref{figB7B}, which depicts the spatiotemporal evolution of $|v_n(t)|$ for both cases. The left panel of Figure \ref{figB7B} shows the dynamics for $B=0$. While the central site at $n=0$ converges to equilibrium, the remaining sites, except for those at $n=\pm 1$ which remain constant, disperse to the left and right across the entire lattice. This overall behavior for the initial condition \eqref{eq:xmplege} in the case $B=0$ resembles that of a  high-amplitude, ``bright-soliton'' initial datum in the continuous defocusing NLS equation, which disperses self-similarly as $t \to \infty$, as proved in \cite{dz2003}. On the other hand, the behavior of all sites below unit modulus is compliant with the analysis of \cite{chf2024}; the formation of the caustics cone analyzed therein is clearly observable.
\begin{figure}[t]
	\begin{center}
    \includegraphics[width=0.4\textwidth]{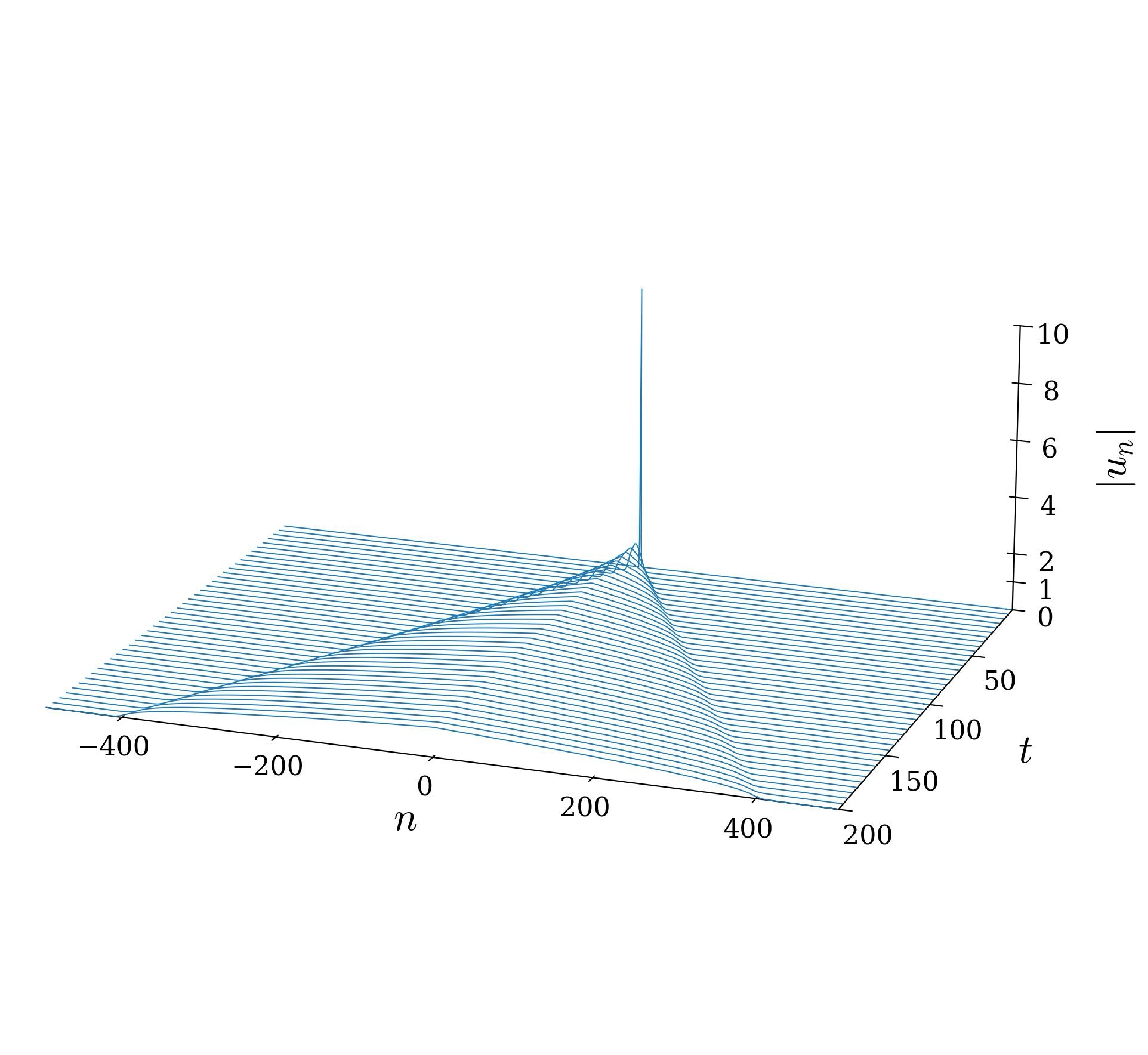}
    \includegraphics[width=0.4\textwidth]{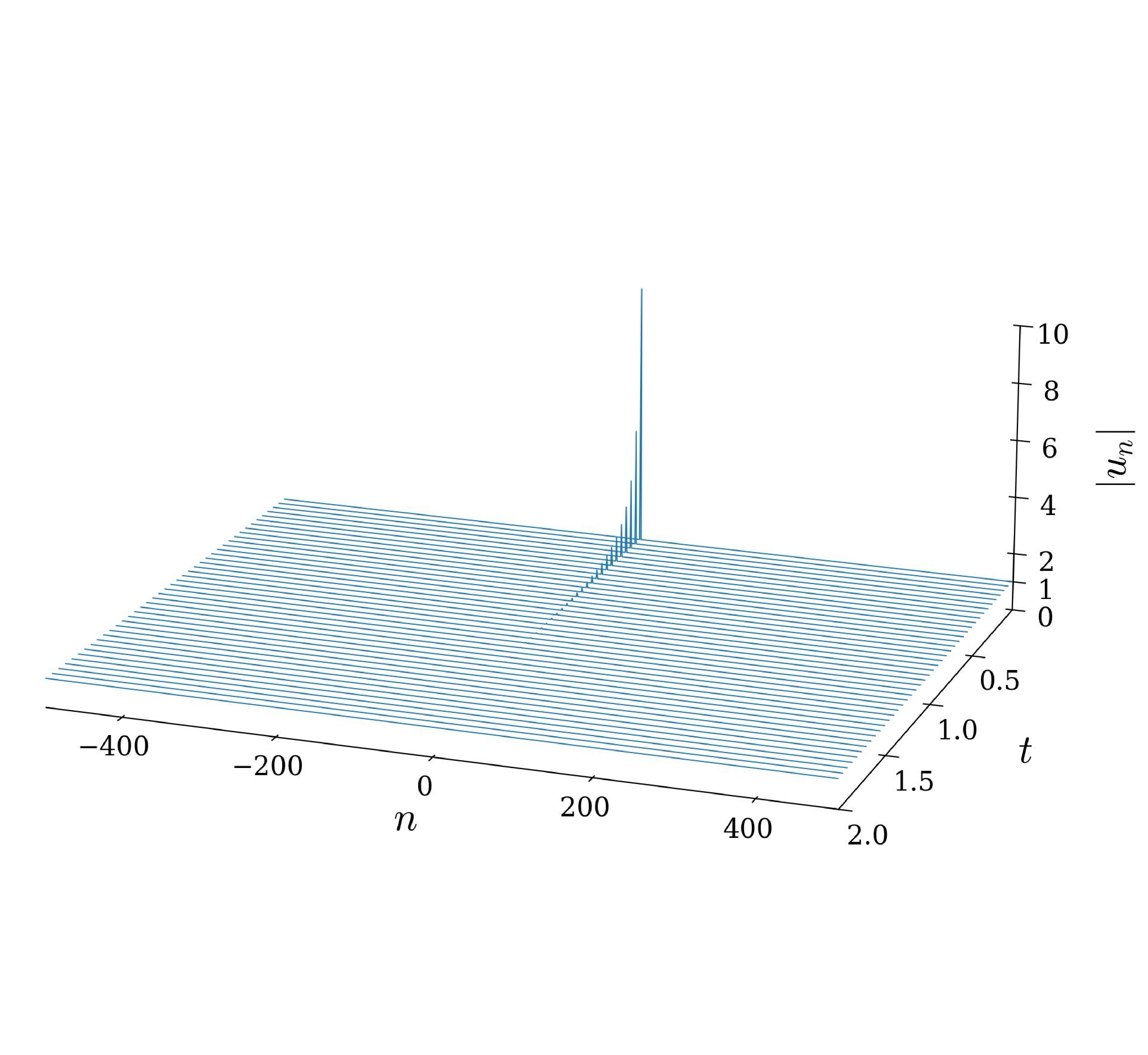}
	\end{center}
	\caption{Example 6: \emph{Spatiotemporal evolution of the initial condition \eqref{eq:xmplege}}. Left panel: parameters $A=10$, $B=0$, $\theta=\frac \pi 3$. Right panel: parameters $A=10$, $B=1$, $\theta=\frac \pi 3$.}
	\label{figB7B}
\end{figure}
The right panel of Figure~\ref{figB7B} shows the dynamics for $B=1$. Since all lattice sites except the central site at $n=0$ are initially located on the unit circle, the solution depicts the convergence of this high-amplitude central site to the unit circle, while all other sites remain stationary as predicted.
\\[2mm]
\emph{II. Equilibration dynamics: the case $B=1$.} What connects the dynamics between the cases $B=0$ and $B=1$ is the equilibration of the centered site at $n=0$. To illustrate the results of Proposition \ref{equib}  we focus on the case $B=1$ since the case $B=0$ is exactly the same for the dynamics of this centered site. 

The left panel of Figure \ref{figB7bb} shows the decay of $|u_0(t)|\to 1$, at an exponential rate as explained below. The right panel shows the evolution of the argument $\arg u_0$ exhibiting an exact linear growth. This is the case (i) of Proposition \ref{equib} as we analyze below.

Setting $w_n=|v_n|^2-1$, equation \eqref{al2} yields
\begin{equation}
 \dot{w}_n=2\,\mathrm{Re}\left(\overline{v}_n\dot{v}_n\right)=-2w_ng_n(t),
 \qquad g_n(t):=\mathrm{Im}\left(\overline{v}_n(v_{n+1}+v_{n-1})\right),\nonumber
\end{equation}
which is a linear scalar ODE for $w_n$, so that $w_n(t)=w_n(0)\exp\left(-2\int_0^tg_n(s)\,ds\right)$. The data \eqref{eq:xmplege} satisfy $w_{\pm1}(0)=0$, hence $w_{\pm1}\equiv0$, and therefore $\dot{v}_{\pm1}\equiv0$, that is $v_{\pm1}(t)=e^{i\theta}$ for all $t\ge0$.
Consequently $v_1+v_{-1}=2e^{i\theta}$ is constant and the equation for the central site decouples from the rest of the lattice,
\begin{equation}\label{eq:v0closed}
 \dot{v}_0=2i\left(|v_0|^2-1\right)e^{i\theta},\qquad v_0(0)=-iAe^{i\theta}.
\end{equation}
For every value of $v_0$ the vector field in \eqref{eq:v0closed} is a real multiple of $-ie^{i\theta}$. The ray $\{-i\varrho e^{i\theta}:\varrho>0\}$, on which the initial datum lies, is therefore invariant. Inserting $v_0(t)=-i\varrho(t)e^{i\theta}$ with $\varrho$ real reduces \eqref{eq:v0closed} to
\begin{equation}
 \dot{\varrho}=-2\left(\varrho^2-1\right),\qquad \varrho(0)=A,\nonumber
\end{equation}
and, the vector field in \eqref{eq:v0closed} being smooth, the Picard-Lindel\"of theorem guarantees that this is the solution. For $A>1$ it integrates to
\begin{equation}
 \varrho(t)=\coth\left(2t+c\right),\qquad c=\frac{1}{2}\ln\frac{A+1}{A-1},\nonumber
\end{equation}
so that $\varrho$ is global in time, strictly decreasing, $1<\varrho(t)\le A$, and $\varrho(t)-1=\mathcal{O}(e^{-4t})$, in agreement with the exponential decay of $|u_0|$ reported in Figure \ref{figB7bb}. Since $\varrho(t)>0$ for all $t$, the argument of $v_0$ is stationary, $\arg v_0(t) = \arg \left(-ie^{i\theta}\right)=\theta-\frac \pi 2$, and hence 
\begin{equation}
 \arg u_0(t)=2t+\theta-\frac{\pi}{2}\qquad (\mathrm{mod}\ 2\pi),\nonumber
\end{equation}
which for $\theta=\frac \pi 3$ is precisely the straight line $2t-\frac \pi 6$ of the right panel of Figure \ref{figB7bb}. Crucially, the condition~\eqref{condconst0} of Proposition \ref{equib}(i) is satisfied and  the local energy $E_{\text{loc}}=2\varrho\cos\left(\arg v_0-\theta\right)$ vanishes identically.  The central unit is in quadrature with its two neighbors, so its phase is frozen in the $v$ variables and the entire observed rotation of $u_0$ is carried by the background
factor $e^{2it}$, while the amplitude relaxes monotonically to the unit circle at the exponential rate $4$. We stress that the argument uses only $|u_{\pm1}(0)|=1$ and $\arg u_0(0)-\theta=-\frac \pi 2$, and is thus independent of the value of $B$.

Figure~\ref{figB7bbb} illustrates an example of equilibration via scenario (ii) of Proposition~\ref{equib}. The initial condition is~\eqref{eq:initA2} with $A_0=10$, $A_n=1$ for $n\neq 0$, $\theta_0=\tfrac{\pi}{3}-\tfrac{1}{20}$, $\theta_1=\tfrac{\pi}{3}$, $\theta_{-1}=0$, and $\theta_n=\tfrac{\pi}{6}$ for $n\notin\{0,\pm 1\}$. With this choice of phases, condition \eqref{condconst} of Proposition~\ref{equib} is satisfied, namely $E_{\mathrm{loc}}(0)\neq 0$. While the evolution of $|u_0(t)|$ is similar to that depicted in the left panel of Figure~\ref{figB7bb} for the initial condition \eqref{eq:xmplege}, the left panel of Figure~\ref{figB7bbb} now illustrates the transient nonlinear behavior of $\arg u_0(t)$ while $|u_0(t)|$ is far from equilibrium (approximately for $t\in [0, 0.75)$), as well as its asymptotic linear growth when $|u_0(t)|\approx 1$, exactly as predicted by Proposition~\ref{equib}(ii). Consequently, the right panel of Figure~\ref{figB7bbb} depicts the convergence $\frac{d}{dt}\arg u_0(t) \to 2$, in agreement with Proposition~\ref{equib}(ii).

%
\begin{figure}[t]
	\begin{center}
		\begin{tabular}{cc}
		\includegraphics[width=0.9\textwidth]{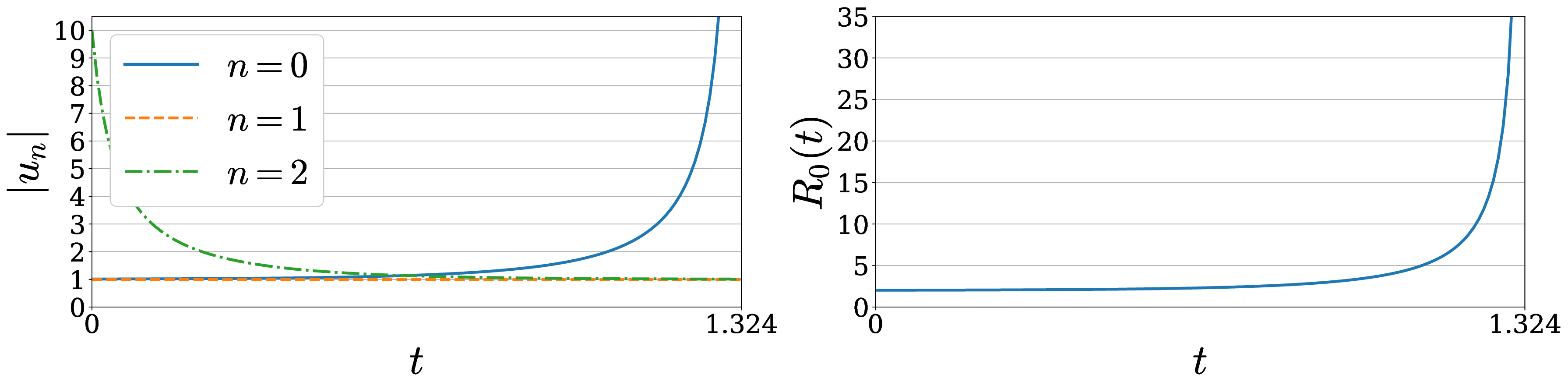}\\
        \includegraphics[width=0.9\textwidth]{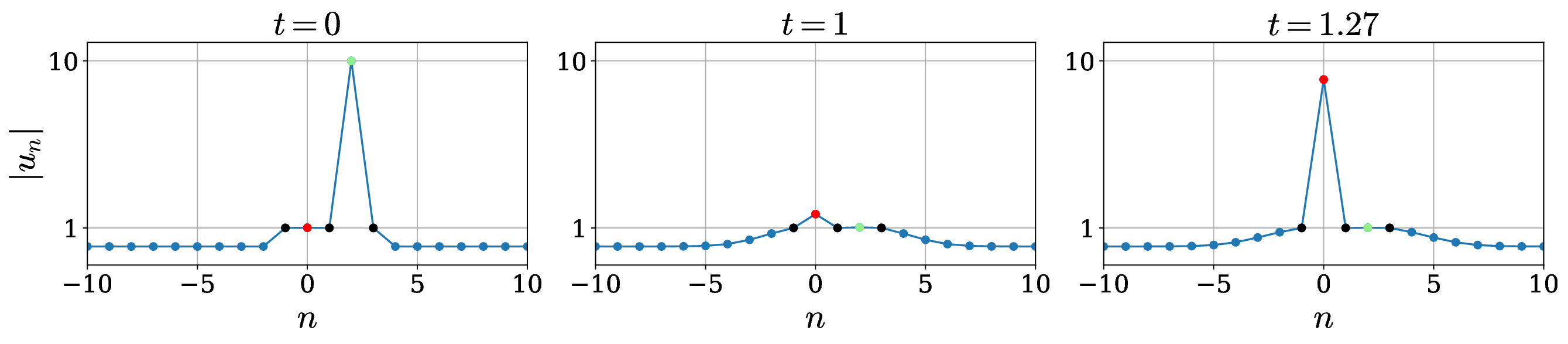}
		\end{tabular}
	\end{center}
	\caption{Example 7: \textit{Singularity of a singular solution with initial data \eqref{eq:xmplege22}}. Left panel: evolution of $|u_n(t)|$ for the sites $n=0,1,2$. Right panel: Evolution of the function $R_0(t)$. The singularity occurs at $t= T_{\max}\approx 1.3240$}
	\label{figB9}
\end{figure}		
\vskip 2mm
\noindent
\textbf{Example 7.} \emph{The instability of the globally existing solution of Proposition~\ref{ge-p} leading to blow-up: the combined effects of equilibration and finite-time singularity.}

We conclude with a final example considering the dynamics of the initial condition
\begin{equation}\label{eq:xmplege22}
u_n(0) = \begin{cases}
1.01 i e^{i}, & n=0,\\
e^{i}, & n=\pm 1, 3,\\
-10ie^{i}, & n=2,\\
0, & \text{otherwise}.
\end{cases}
\end{equation}
The initial condition \eqref{eq:xmplege22} is a superposition of two features: a component centered at site $n=2$ that satisfies $R_2(0)=-20$ and $J_2(0)=-4$ (analogous to the initial condition \eqref{eq:xmplege}) and is expected to undergo equilibration, and a component at site $n=0$ that defines a fundamental singular solution with an analytical blow-up time $T_{\max}^{J(0)<0} \approx 1.3258$, determined by formula \eqref{eq:maxtimep}.

Figure~\ref{figB9} illustrates the combined effects of equilibration at site $n=2$ and the singularity emerging at site $n=0$. The numerical blow-up time is $t = T_{\max} \approx 1.3240$, once again in excellent agreement with the analytical prediction.

Importantly, this example highlights the instability of global solutions guaranteed by Proposition~\ref{ge-p}(i). If $|u_0(0)| \leq 1$, the initial condition for site $n=2$ would define a globally existing solution with $|u_2(t)| \to 1$ as $t \to \infty$. However, in our setup, the initial condition at site $n=0$  being a small perturbation above unit modulus drives the solution to blow up in finite time (via the site at $n=0$), despite site $n=2$ satisfying the conditions of Proposition~\ref{ge-p}(i) and locally experiencing equilibration.

We conjecture that modulational instability of other analytical solutions with amplitudes greater than unity, such as plane waves or solutions on a background \cite{oy2014, op2019, bckop2025, cp2026}, will similarly trigger finite-time blow-up. This conjecture will be investigated in subsequent work, and the results will be reported elsewhere.
 
\bibliographystyle{myamsalpha}
\bibliography{references}

\newcommand{\etalchar}[1]{$^{#1}$}
\providecommand{\bysame}{\leavevmode\hbox to3em{\hrulefill}\thinspace}
\providecommand{\MR}{\relax\ifhmode\unskip\space\fi MR }
\providecommand{\MRhref}[2]{%
  \href{http://www.ams.org/mathscinet-getitem?mr=#1}{#2}
}
\providecommand{\href}[2]{#2}
\begin{thebibliography}{DPvdMV13}

\bibitem[AC91]{ac1991}
M.~J. Ablowitz and P.~A. Clarkson, \emph{Solitons, nonlinear evolution
  equations and inverse scattering}, London Mathematical Society Lecture Note
  Series, vol. 149, Cambridge University Press, Cambridge, 1991. \MR{1149378}

\bibitem[AG94]{AG94}
G.~S. Ammar and W.~B. Gragg, \emph{Schur flows for orthogonal {H}essenberg
  matrices}, Hamiltonian and Gradient Flows, Algorithms and Control, Fields
  Inst. Commun., vol.~3, Amer. Math. Soc., Providence, RI, 1994, pp.~27--34.

\bibitem[AL76a]{al1976b}
M.~J. Ablowitz and J.~F. Ladik, \emph{A nonlinear difference scheme and inverse
  scattering}, Stud. Appl. Math. \textbf{55} (1976), no.~3, 213--229.
  \MR{471341}

\bibitem[AL76b]{al1976a}
\bysame, \emph{Nonlinear differential-difference equations and {F}ourier
  analysis}, J. Mathematical Phys. \textbf{17} (1976), no.~6, 1011--1018.
  \MR{427867}

\bibitem[APT04]{apt2004}
M.~J. Ablowitz, B. Prinari, and A.~D. Trubatch, \emph{Discrete and continuous
  nonlinear {S}chr{\"o}dinger systems}, vol. 302, Cambridge University Press,
  2004.

\bibitem[Bal77]{b1977}
J. Ball, \emph{Remarks on blow-up and nonexistence theorems for nonlinear
  evolution equations}, The Quarterly Journal of Mathematics \textbf{28}
  (1977), no.~4, 473--486.

\bibitem[Bal78]{b1978}
\bysame, \emph{Finite-time blow-up in nonlinear problems, {N}onlinear
  {E}volution {E}quations ({MG} {C}randall, ed.)}, 1978.

\bibitem[BCK{\etalchar{+}}25]{bckop2025}
E.~C. Boadi, E.~G. Charalampidis, P.~G. Kevrekidis, N.~J. Ossi, and B. Prinari,
  \emph{On the discrete {K}uznetsov-{M}a solutions for the defocusing
  {A}blowitz-{L}adik equation with large background amplitude}, Wave Motion
  \textbf{134} (2025), 103496.

\bibitem[Caz03]{c2003}
T. Cazenave, \emph{Semilinear {S}chr\"{o}dinger equations}, Courant Lecture
  Notes in Mathematics, vol.~10, New York University, Courant Institute of
  Mathematical Sciences, New York; American Mathematical Society, Providence,
  RI, 2003. \MR{2002047}

\bibitem[CH98]{ch1998}
T. Cazenave and A. Haraux, \emph{An introduction to semilinear evolution
  equations}, Oxford Lecture Series in Mathematics and its Applications,
  vol.~13, Clarendon Press, Oxford University Press, New York, 1998.
  \MR{1691574}

\bibitem[CHF24]{chf2024}
M. Chen, J. He, and E. Fan, \emph{Long-time asymptotics for the defocusing
  {A}blowitz-{L}adik system with initial data in lower regularity}, Advances in
  Mathematics \textbf{450} (2024), 109769.

\bibitem[CJ16]{cj2016}
S. Cuccagna and R. Jenkins, \emph{On the asymptotic stability of {$N$}-soliton
  solutions of the defocusing nonlinear {S}chr\"odinger equation}, Comm. Math.
  Phys. \textbf{343} (2016), no.~3, 921--969. \MR{3488549}

\bibitem[CP26]{cp2026}
F. Coppini and B. Prinari, \emph{Novel periodic solutions and rogue waves of
  the defocusing scalar and coupled {A}blowitz-{L}adik systems on a nonzero
  background}, arXiv:2606.05123 (2026).

\bibitem[CS24]{cs2024}
F. Coppini and P.~M. Santini, \emph{Modulation instability, periodic anomalous
  wave recurrence, and blow up in the {A}blowitz-{L}adik lattices}, J. Phys. A
  \textbf{57} (2024), no.~1, Paper No. 015202, 28. \MR{4676830}

\bibitem[DPvdMV13]{dpmv2013}
F. Demontis, B. Prinari, C. van~der Mee, and F. Vitale, \emph{The inverse
  scattering transform for the defocusing nonlinear {S}chr\"odinger equation
  with nonzero boundary conditions}, Stud. Appl. Math. \textbf{131} (2013),
  no.~1, 1--40. \MR{3081238}

\bibitem[DZ03]{dz2003}
P. Deift and X. Zhou, \emph{Long-time asymptotics for solutions of the {NLS}
  equation with initial data in a weighted {S}obolev space}, Comm. Pure Appl.
  Math. \textbf{56} (2003), no.~8, 1029--1077. \MR{1989226}

\bibitem[FG99]{FG99}
L. Faybusovich and M. Gekhtman, \emph{On {S}chur flows}, J. Phys. A \textbf{32}
  (1999), no.~25, 4671--4680.

\bibitem[GN09]{GN09}
M. Gekhtman and I. Nenciu, \emph{Multi-{H}amiltonian structure for the finite
  defocusing {A}blowitz-{L}adik equation}, Comm. Pure Appl. Math. \textbf{62}
  (2009), no.~2, 147--182.

\bibitem[Gol06]{Gol06}
L. Golinskii, \emph{Schur flows and orthogonal polynomials on the unit circle},
  Sb. Math. \textbf{197} (2006), no.~7--8, 1145--1165.

\bibitem[Hal80]{h1980}
J.~K. Hale, \emph{Ordinary differential equations}, second ed., Robert E.
  Krieger Publishing Co., Inc., Huntington, NY, 1980. \MR{587488}

\bibitem[HKMM25]{hkmm2025}
D. Hennig, N.~I. Karachalios, D. Mantzavinos, and D. Mitsotakis, \emph{On the
  lifespan of nonzero background solutions to a class of focusing nonlinear
  {S}chr\"odinger equations}, Wave Motion \textbf{132} (2025), Paper No.
  103419. \MR{4812401}

\bibitem[HKMM26]{hkmm2026}
D. Hennig, N.~I. Karachalios, D. Mantzavinos, and D. Mitsotakis, \emph{Discrete
  nonlinear {S}chr\"odinger versus {A}blowitz-{L}adik: {E}xistence and dynamics
  of generalized focusing {NLS}-type lattices over a nonzero background},
  Nonlinearity \textbf{39} (2026), 035024.

\bibitem[KK04]{knoll2004jacobian}
D.~A. Knoll and D.~E. Keyes, \emph{Jacobian-free {N}ewton-{K}rylov methods: a
  survey of approaches and applications}, Journal of Computational Physics
  \textbf{193} (2004), no.~2, 357--397.

\bibitem[KN07]{KN07}
R. Killip and I. Nenciu, \emph{{CMV}: the unitary analogue of {J}acobi
  matrices}, Comm. Pure Appl. Math. \textbf{60} (2007), no.~8, 1148--1188.

\bibitem[LCM{\etalchar{+}}25]{lcmmkk2025}
M.~L. Lytle, E.~G. Charalampidis, D. Mantzavinos, J. Cuevas-Maraver, P.~G.
  Kevrekidis, and N.~I. Karachalios, \emph{On the proximity of
  {A}blowitz-{L}adik and discrete nonlinear {S}chr{\"o}dinger models: A
  theoretical and numerical study of {K}uznetsov-{M}a solutions}, Wave Motion
  \textbf{137} (2025), 103547.

\bibitem[Mit25]{Mitsotakis2025}
D. Mitsotakis, \emph{The complex-step {N}ewton method and its convergence},
  Numerische Mathematik \textbf{157} (2025), no.~3, 993--1021.

\bibitem[Nen05]{Nen05}
I. Nenciu, \emph{Lax pairs for the {A}blowitz-{L}adik system via orthogonal
  polynomials on the unit circle}, Int. Math. Res. Not. \textbf{2005} (2005),
  no.~11, 647--686.

\bibitem[OP19]{op2019}
A.~K. Ortiz and B. Prinari, \emph{Inverse scattering transform for the
  defocusing {A}blowitz-{L}adik system with arbitrarily large nonzero
  background}, Stud. Appl. Math. \textbf{143} (2019), no.~4, 373--403.

\bibitem[OY14]{oy2014}
Y. Ohta and J. Yang, \emph{General rogue waves in the focusing and defocusing
  {A}blowitz-{L}adik equations}, Journal of Physics A: Mathematical and
  Theoretical \textbf{47} (2014), no.~25, 255201.

\bibitem[Saa85]{saad1985generalized}
Y. Saad, \emph{A generalized minimal residual algorithm for solving
  nonsymmetric linear systems}, SIAM J. Sci. Stat. Comput \textbf{7} (1985),
  417--424.

\bibitem[Seg63]{s1963}
I. Segal, \emph{Non-linear semi-groups}, Ann. of Math. (2) \textbf{78} (1963),
  339--364. \MR{152908}

\bibitem[Sim05]{Sim05}
B. Simon, \emph{Orthogonal polynomials on the unit circle. {P}arts 1 and 2},
  Amer. Math. Soc. Colloq. Publ., vol.~54, Amer. Math. Soc., Providence, RI,
  2005.

\bibitem[Sim07]{Sim07}
\bysame, \emph{Zeros of {OPUC} and long time asymptotics of {S}chur and related
  flows}, Inverse Probl. Imaging \textbf{1} (2007), no.~1, 189--215.

\bibitem[SS99]{ss1999}
C. Sulem and P.-L. Sulem, \emph{The nonlinear {S}chr\"odinger equation},
  Applied Mathematical Sciences, vol. 139, Springer-Verlag, New York, 1999,
  Self-focusing and wave collapse. \MR{1696311}

\bibitem[Tao06]{t2006}
T. Tao, \emph{Nonlinear {D}ispersive {E}quations}, CBMS Regional Conference
  Series in Mathematics, vol. 106, American Mathematical Society, 2006, Local
  and global analysis. \MR{2233925}

\bibitem[XF18]{XiaFokas18}
B. Xia and A.~S. Fokas, \emph{Initial-boundary value problems associated with
  the {A}blowitz-{L}adik system}, Phys. D \textbf{364} (2018), 27--61.
  \MR{3737868}

\bibitem[Zei86]{z1986}
E. Zeidler, \emph{Nonlinear functional analysis and its applications. {I}},
  Springer-Verlag, New York, 1986, Fixed-point theorems, Translated from the
  German by Peter R. Wadsack. \MR{816732}

\bibitem[ZS71]{zs1971}
V.~E. Zakharov and A.~B. Shabat, \emph{Exact theory of two-dimensional
  self-focusing and one-dimensional self-modulation of waves in nonlinear
  media}, \v{Z}. \`Eksper. Teoret. Fiz. \textbf{61} (1971), no.~1, 118--134.
  \MR{0406174}

\end{thebibliography}
 
\end{document}